\documentclass{article}

\usepackage{PRIMEarxiv}
\usepackage[utf8]{inputenc} 
\usepackage[T1]{fontenc}
\usepackage{url}
\usepackage{booktabs}        
\usepackage{nicefrac}       
\usepackage{lipsum}
\usepackage{fancyhdr}       
\usepackage{graphicx}       
\graphicspath{{media/}} 
\RequirePackage{amsthm,amsmath,amsfonts,amssymb}
\usepackage{dsfont}
\usepackage{newtxtext}
\usepackage{enumitem}
\usepackage{anyfontsize}
\usepackage{array}
\RequirePackage[numbers,sort&compress]{natbib}
\RequirePackage[colorlinks,citecolor=blue,urlcolor=blue,linkcolor=blue]{hyperref}
\usepackage{xcolor}

\theoremstyle{plain}
\newtheorem{theorem}{Theorem}

\newtheorem{corollary}{Corollary}
\newtheorem{lemma}{Lemma}

\theoremstyle{definition}
\newtheorem{assumption}{Assumption}

\newtheorem{remark}{Remark}

\def\R{\mathbb{R}}

\def\N{\mathbb{N}}

\newcommand{\Ex}{\mathbb{E}} 
\renewcommand{\Pr}{\mathbb{P}}
\newcommand{\Var}{\mathbb{V}}

\DeclareMathOperator*{\argmin}{arg\,min}

\title{Debiasing the Lasso under Weaker Tail Assumptions}

\author{
  Leonardo Voltarelli\\
  IMPA\\
  \texttt{leonardo.voltarelli@impa.br} \\
\And
  Roberto Imbuzeiro Oliveira\\
  IMPA\\
  \texttt{rimfo@impa.br} \\
}

\begin{document}
\maketitle

\begin{abstract}
We consider the problem of high-dimensional inference with the lasso estimator. Different methods including `double selection' techniques and multiple versions of the `debiased lasso' have been proposed for this task with noticeable success. However, most guarantees assume strong hypotheses on the underlying data process and the errors in the linear regression model, such as subgaussian designs and independence between errors and the data itself. We show that `standardizing' one's dataset -- a natural procedure in practical penalized regression -- leads to the same results under much weaker hypotheses, paying only a small price for not assuming light tails. The key technical point allowed by this step is exploiting the concentration properties of self-normalized processes. Importantly, we prove our results for two different methods closely related to the `debiased lasso'. The second method performs valid inference even for a misspecified linear model, under mild sparsity conditions similar to the `double selection' literature.
\end{abstract}

\keywords{Self-normalization \and High-dimension \and Inference \and Regression}

\section{Introduction} Many modern statistical problems are concerned with complex data in which the number of variables greatly exceeds that of available samples, with famous examples including genomics \cite{Peng2010-lw,buhlmann2017high,Melsen2025-fg}, machine learning \cite{Clemmensen01112011,GENUER20102225,Chen2022-hu,Fan2021-dy} and economics \cite{10.1093/restud/rdt044,https://doi.org/10.3982/ECTA9626,10.1257/jep.28.2.29}. Research on regression methods in this high-dimensional setup has been especially fruitful, with the lasso regression \cite{10.1111/j.2517-6161.1996.tb02080.x} posing as one of the most widely recognized methods in recent years due to its success in parameter estimation and variable selection.

However the bias introduced by the lasso penalty severely hinders valid statistical inference. Starting with the works of Zhang \& Zhang \cite{10.1111/rssb.12026}, Van de Geer et al. \cite{f29b4853-f902-378f-99ba-4ec108f6737e} and Javanmard \& Montanari \cite{JMLR:v15:javanmard14a}, different methods have been proposed to address this challenge by correcting the bias term, leading to the `debiased lasso'. Since then, many other procedures for high-dimensional inference have been successfully developed. Some examples on the lasso include \cite{10.1214/20-EJS1713, 10.3150/21-BEJ1348,10.1214/17-AOS1630}, but other models such as Generalized Linear Models \cite{f29b4853-f902-378f-99ba-4ec108f6737e, 10.5555/3648699.3648995, 10.1111/biom.13587}, `double-selection' methods and linear instrumental variables models \cite{10.1257/jep.28.2.29,10.1093/restud/rdt044,https://doi.org/10.3982/ECTA9626}, Gaussian Graphical Models \cite{10.1214/15-EJS1031,Meng2025-pm} and general penalized M-estimators have also been studied \cite{10.1214/16-AOS1448}.

Notably, the vast majority of such works assume relatively strong conditions on the data samples and regression errors. This is the case for the foundational works of Zhang \& Zhang \cite{10.1111/rssb.12026}, Van de Geer et al. \cite{f29b4853-f902-378f-99ba-4ec108f6737e} and Javanmard \& Montanari \cite{JMLR:v15:javanmard14a} on the `debiased lasso' and Belloni et al. on `double selection' \cite{10.1257/jep.28.2.29,10.1093/restud/rdt044}, among many others. For instance, \cite{f29b4853-f902-378f-99ba-4ec108f6737e} and \cite{JMLR:v15:javanmard14a} assume that the observations follow a sub-gaussian distribution with invertible covariance matrix and that errors are independent. \cite{10.1214/16-AOS1448} allows for strong dependency between data samples and regression errors in some cases, but still requires sub-gaussian or even bounded data assumptions. \cite{https://doi.org/10.3982/ECTA9626} removes homoskedasticity assumptions by adding a matrix of loadings on the lasso penalty term, but implicitly assumes sub-gaussian/bounded designs to ensure the required restricted eigenvalue properties for the lasso consistency. In practice, data can be much more heavy-tailed and errors are not always independent from the data, which calls for more robust techniques.

In the present work, we study two methods for the debiased lasso. The first one is based on \cite{JMLR:v15:javanmard14a}, for which we show that the usual assumptions on the data samples can be severely relaxed and the same results can be recovered in more generality, without much additional effort. The second method is a variant of the approach proposed by \cite{f29b4853-f902-378f-99ba-4ec108f6737e} where debiasing is directly formulated as a lasso-type problem without additional steps. However, our assumptions on the data distribution are significantly weaker. In particular, we allow for model misspecification and heteroskedastic data, in the spirit of Belloni et al \cite{10.1093/restud/rdt044}. In both cases, the moment conditions we require are significantly weaker than in related work. The price to pay is a slightly stronger sparsity requirement that depends on which moments of the data and error distributions are assumed to be bounded. This price gets smaller the more one allows for higher-order bounded moments.

The key point in our analysis is to assume that, prior to computing the regression coefficients, the practitioner has `standardized' the data. This procedure is ubiquitous in applications, in particular when dealing with a penalized setting, which often assumes an equal scale for the regression variables. This leads to a very favorable setup for exploiting self-normalized concentration inequalities, which is the primary technical tool used to obtain our results. Our main contributions are as follows:
\begin{itemize}
    \item We show that inference in high-dimensional regression with the lasso can be performed without sub-gaussian tails or independent error assumptions by explicitly proving an asymptotic distribution of the debiased estimator.
    \item We formally demonstrate benefits of `standardizing' data in a regression setting as usually done in practical applications, but often without theoretical justification. Notably, standardization brings us into the domain of self-normalized concentration inequalities \cite{Pena2009-av}, which hold under weak assumptions.
    \item We obtain novel theoretical guarantees for debiasing methods that nearly match the best available guarantees \cite{10.1257/jep.28.2.29,10.1093/restud/rdt044} in terms of sparsity, model misspecification and heteroskedasticity, while allowing for significantly more general data distributions. The price to pay is controlled by finite moment conditions on the data; the more moments one assumes, the smaller the price.
\end{itemize}

The remainder of this paper is organized as follows. The next section, Section \ref{sec:notation}, highlights the basic notation and definitions used throughout the text. Section \ref{sec:debiased} discusses the debiased lasso estimator in more detail. Section \ref{sec:assumptions} provides relevant examples and explanations regarding the required assumption. Our main result is presented in Section \ref{sec:results}; additional consequences of this result for high-dimensional inference are presented in the same section. Section \ref{sec:experiments} contains numerical experiments. Proofs are collected in Appendix \ref{appendix:A}.

\section{Definitions and notation}\label{sec:notation}
\subsection{Basic notation}$\N$ represents the set of natural numbers. Given $n \in \mathbb{N}$, we set $[n] = \left\{ 1, \dots, n \right\}$. We write $a \vee b = \max\{a,b\}$. Also, given a vector $v \in \R^n$, we denote its $i$-th coordinate by $v_i$, for each $i \in [n]$. We also write $\|v\|_p$, $1 \leq p \leq \infty$, for the $\ell^p$ norm of $v$. The `$\ell^0$ norm' of a vector $v$ is denoted by:
\begin{equation*}
    \|v\|_0 = \sum_{i=1}^{n} \mathbf{1}_{\left\{v_i \neq 0\right\}} = |S(v)| \, ,
\end{equation*}
where $\mathbf{1}_A$ is used to represent the indicator function of a set and $S(v)$ is the support of $v$. Throughout the text, $\mathds{1}_n$ represents the column vector in $\R^n$ with all entries equal to 1. Also, $e_i$, $i \in [n]$, is used for the $i$-th canonical basis vector of $\R^n$, for some $n$ which is often implicit. Given $V \in \R^{n \times p}$, $i \in [n]$, and $j \in [p]$, $V_{ij}$ will be used to represent the entries of $V$. When $V$ is a square matrix we use $\lambda_{min}(V)$ and $\lambda_{max}(V)$ to represent its smallest and largest eigenvalues, respectively. Lastly, $I_{n \times n}  \in  \R^{n \times n}$ is reserved for the identity matrix.

Given a sequence $\left\{ a_n\right\}_{n\geq 1}$, we use the following asymptotic notation: $a_n = o(1)$ means that $a_n \rightarrow 0$; $a_n = O(1)$ is used when the sequence is bounded by a constant that does not depend on $n$. We also employ the probabilistic equivalents $o_p(1)$ and  $O_p(1)$ to say that a sequence of random variables converges to zero in probability or is bounded in probability.

\subsection{Data standardization}  We let $\left\{ (X_i,Y_i)\right\}_{i=1}^{n}$ be a sample of independent and identically distributed observations, with each $X_i \in \R^p$ and $Y_i \in \R$. We denote $X \in \R^{n \times p}$ the data matrix $X = (X_1, \dots,X_n)^T$ and set $Y = (Y_1,Y_2,\dots,Y_n)^T \in \R^n$. Then, define, for every $j \in [p]$:
\begin{equation*}
    \mu_j = \Ex\left[X_{1j}\right] \;, \; \; \hat{\mu}_j = \frac{1}{n} \sum_{i=1}^{n} X_{ij} \; ,
\end{equation*}
\begin{equation*}
    \sigma^2_j = \Var\left[X_{1j}\right] \; , \; \; \hat{\sigma}^2_j = \frac{1}{n} \sum_{i=1}^{n} (X_{ij} - \hat{\mu}_j)^2 \; .
\end{equation*}
The key point in our analysis is to work with a `standardized' version of the matrix $X$, in which we subtract the empirical mean $\hat{\mu}_j$ and divide by the empirical standard deviation $\hat{\sigma}_j$ in each column. More explicitly, we let $\tilde{X}$ be the matrix of entries:
\begin{equation}
\label{eq:tilde_x}
\tilde{X}_{ij} = 
\begin{cases}
\frac{X_{ij}-\hat{\mu}_j}{\hat{\sigma}_j} & \mathrm{if} \; \hat{\sigma}_j > 0\\
0 & \mathrm{if} \; \hat{\sigma}_j = 0
\end{cases}
 \; , \; \; \;  \mathrm{for} \;  i \in [n]  \; ,  \; j \in [p] \; ,
\end{equation}
and also define the `standardized' version with the population mean and standard deviation in place of their sample counterparts:
\begin{equation*}
\bar{X}_{ij} = 
\begin{cases}
\frac{X_{ij}-\mu_j}{{\sigma}_j} & \mathrm{if} \; {\sigma}_j > 0\\
0 & \mathrm{if} \; {\sigma}_j = 0
\end{cases}
 \; , \; \; \;  \mathrm{for} \;  i \in [n]  \; ,  \; j \in [p] \; .
\end{equation*}
Although we separate the cases in which the denominators above are possibly zero, our assumptions below guarantee this is almost never the case.
It will also be useful to think of $\tilde{X}$ in terms of matrix operations. We write:
\begin{equation*}
    \tilde{X} = A_n X D_X^{-1/2} \; ,
\end{equation*}
where:
\begin{equation*}
    A_n = I_{n \times n}-\frac{1}{n} \mathds{1}_{n}\mathds{1}_{n}^T \;  , \; D_{X} = \mathrm{diag}\left(\hat{\sigma}_1^2, \dots, \hat{\sigma}_p^2 \right)  \; .
\end{equation*}
This transformation is standard practice when using penalized regression models, as often the penalty term deals with regression variables as if they were in the same scale. We also define:
\begin{equation*}
    \hat{\Sigma}_n = \frac{1}{n} X^TX \;  , \; \; {\Sigma} = \frac{1}{n}\ \mathbb{E}{X}^T{X}\; , \; \; \tilde{\Sigma}_n = \frac{1}{n} \tilde{X}^T\tilde{X}\;  , \; \; \bar{\Sigma} = \frac{1}{n}\mathbb{E}\bar{X}^T\bar{X} \; ,
\end{equation*}
in such way that $\hat{\Sigma}_n$ and ${\Sigma}$ are the sample and population design matrices, respectively. $\tilde{\Sigma}_n $ is the sample correlation matrix of the covariates, while $\bar{\Sigma}$ is its population equivalent.
This specific standardization procedure leads to self-normalization of the matrix data, which lays the groundwork for exploiting concentration inequalities which do not assume sub-gaussian or sub-exponential tails but only finite moments (see chapter 12 in \cite{10.1093/acprof:oso/9780199535255.001.0001}, for example).
\section{Debiased lasso}\label{sec:debiased}
We consider the linear model:
\begin{equation}
\label{eq:linear_model}
    Y = X \beta^0 + \gamma^0 \mathds{1}_n + \rho + \epsilon \; ,
\end{equation}
where $\beta^0 \in \mathbb{R}^p$ and $\gamma^0 \in \R$ are fixed unknowns, $\rho \in \R^n$ is an approximation error and $\epsilon \in \mathbb{R}^n$ represents a random error term. For now, we highlight that this allows for a misspecified linear model; the difference between $\rho$ and $\epsilon$ is made clearer in the assumptions below. For simplicity, we  assume $\left\{(X_i,Y_i)\right\}_{i=1}^{n}$ are independent and identically distributed random variables. We also assume that $(\gamma^0,\beta^0)$ is the best linear approximation of $\Ex [Y|X]$ in the following sense:
\begin{equation*}
    (\gamma^0,\beta^0) \in \argmin_{\gamma \in \R, \beta \in \mathbb{R}^p} \Ex\left[\left(\Ex[Y_1|X_1]- \langle {X}_1,\beta\rangle - \gamma \right)^{2}\right]  \; ,
\end{equation*} 
which imply, by the optimality conditions, the following normal equations for $\rho$:
\begin{equation}
\label{eq:normal_equations}
\Ex\left[ \rho_1 \right]=0 \; , \; \; \Ex \left[\rho_1 X_{1k} \right]=0 \; , \; \; \forall k \in [p] \; .
\end{equation}
The lasso estimator is given by:
\begin{equation}
\label{eq:lasso}
    (\hat{\gamma} , \hat{\beta}) \in \argmin_{\gamma \in \R, \beta \in \mathbb{R}^p} \frac{1}{2n} \|Y- \tilde{X}\beta - \gamma \mathds{1}_n\|_{2}^{2} + \lambda \|\beta\|_1 \; ,
\end{equation}
where $\lambda > 0$ is the penalty parameter and $\tilde{X}$ is given by equation \eqref{eq:tilde_x}. This is a convex optimization problem, for which KKT conditions imply, with $\hat{g}$ a subgradient of $\|\cdot\|_1$  at $\hat{\beta}$:
\begin{equation*}
    \begin{cases} 
   -\frac{\mathds{1}^{T}_{n}}{n}\left(Y - \tilde{X}\hat{\beta} - \hat{\gamma} \mathds{1}_n\right) = 0 \Rightarrow \hat{\gamma} = \frac{\mathds{1}^{T}_{n}}{n}Y
   \\
   \begin{aligned}
   -\frac{\tilde{X}^T}{n}\left(Y - \tilde{X}\hat{\beta} - \hat{\gamma} \mathds{1}_n\right) + \lambda \hat{g} = 0 \Rightarrow \tilde{\Sigma}_n \left(\hat{\beta} - D_{X}^{1/2} \beta^0\right) + \lambda \hat{g} = \frac{1}{n}\tilde{X}^{T}(\epsilon + \rho) \; ,
   \end{aligned}
\end{cases}
\end{equation*}
with the simple observations that $\mathds{1}_n^T\tilde{X}=0$ and $\frac{1}{n}\tilde{X}^TX = \tilde{\Sigma}_n D_X^{1/2}$. If the matrix $\tilde{\Sigma}_n$ was invertible, we could easily isolate the difference $\hat{\beta} - D_X^{1/2}\beta^0$ and proceed with the analysis.

However, this is never the case in the $p > n$ setting. The idea is then to introduce a matrix $M \in \R^{p \times p}$ that will serve as an approximate inverse. We propose two methods for computing $M$, the first method being based on Theorem 21 in \cite{JMLR:v15:javanmard14a}, while the second is closely related to the estimator in \cite{f29b4853-f902-378f-99ba-4ec108f6737e}.\\

\textit{Method 1:} Define $M= (m^{(1)}, \dots,m^{(p)})^T$ with the row vectors $(m^{(j)})^T$ given by the optimization problem below, for each $j \in [p]$:
\begin{eqnarray}
\label{eq:opt1}
m^{(j)} \in \argmin_{m \in \mathbb{R}^p} \quad & m^T \tilde{\Sigma}_n m \\
\mathrm{s.t.} \quad & \|e_j - \tilde{\Sigma}_n m\|_{\infty} \leq \mu \; \; \; ,\nonumber\\
\quad & \|\tilde{X}m\|_{\infty} \leq \zeta \nonumber
\end{eqnarray}
where $\mu, \zeta$ are positive parameters chosen a priori. If solving for $m^{(j)}$ gives an infeasible problem, set $e_j$ as its solution. Note that, as we work with standardized data, we replace $\hat{\Sigma}_n$ by $\tilde{\Sigma}_n$ in the original formulation by \cite{JMLR:v15:javanmard14a}.\\

\textit{Method 2:} Define $M= (m^{(1)}, \dots,m^{(p)})^T$ with the row vectors $(m^{(j)})^T$ given by the optimization problem below, for each $j \in [p]$:
\begin{eqnarray}
\label{eq:opt2}
m^{(j)} \in \argmin_{m \in \mathbb{R}^p} \quad & \frac{1}{2}m^T \tilde{\Sigma}_n m - e_j^T m + \lambda_M \|m\|_1 
\end{eqnarray}
where $\lambda_M$ is a positive parameter chosen a priori. If no solution exists, set $m^{(j)} = e_j$. Although it might not be obvious at first sight, this method is closely related to the one proposed in \cite{f29b4853-f902-378f-99ba-4ec108f6737e}. The relationship between them is made explicit in Lemma \ref{lemma:nodewise_lasso}. In short, it can be shown that the estimator in \cite{f29b4853-f902-378f-99ba-4ec108f6737e} is indeed a minimizer of the problem above once we remove the regularization on the $j$-th coordinate and replace $\tilde{\Sigma}_n$ by $\hat{\Sigma}_n$.\\

For both methods, we then define the following estimator:
\begin{equation}
\label{eq:debiased}
    \hat{\beta}^u_j = \hat{\beta}_j + \lambda (m^{(j)})^T\hat{g} = \hat{\beta}_j + \frac{1}{n} (m^{(j)})^T\tilde{X}^T\left(Y-\tilde{X}\hat{\beta} - \hat{\gamma}\mathds{1}_n\right)\; , \; \; j \in [p] \; .
\end{equation}
It should be clear by the above equation that adding a new term that is proportional to a subgradient of the $\ell^1$ norm compensates the bias introduced by this penalty. Also, the above equation shows that if we are interested in estimating only the $j$-th coordinate of $\beta^0$, it is only necessary to compute $m^{(j)}$ using either one of the above methods. Then, from the KKT conditions for the lasso estimator in \eqref{eq:lasso} given above:
\begin{equation}
\label{eq:basic_eq}
    \sqrt{n} \left(\hat{\beta}^u_j-\hat{\sigma}_j\beta^0_j\right) = \underbrace{\frac{1}{\sqrt{n}} (m^{(j)})^T \tilde{X}^T (\epsilon+\rho)}_{(Z_j)} + \underbrace{\sqrt{n}(e_j^T-(m^{(j)})^T\tilde{\Sigma}_{n})(\hat{\beta}-D_{X}^{1/2}\beta^0)}_{(\Delta_j)} \; .
\end{equation}
We will show that, under some hypotheses (which differ depending on the chosen method), for the `debiased lasso' $\hat{\beta}^u_j/\hat{\sigma}_j$:
\begin{equation}
\label{eq:convergence}
     \sqrt{n} \left(\hat{\beta}^u_j/\hat{\sigma}_j-\beta^0_j\right) \xrightarrow{d} \mathcal{N}(0, \Ex[\langle{X}_1-\Ex X_1,\theta^{(j)}\rangle^2\epsilon_1^2]) \; ,
\end{equation}
where $\theta^{(j)}$ satisfies $\bar{\Sigma}\theta^{(j)} = e_j$. This will be done by first ensuring that the $\Delta_j$ term in equation \eqref{eq:basic_eq} vanishes in probability. The whole convergence will follow from establishing the desired gaussian limit for the $Z_j$ term. Construction of confidence interval and hypothesis tests are given as a direct consequence of this asymptotic result.

Our main assumptions, which we highlight beforehand, are related to sparsity. For both methods we need to assume sparsity of the regression vector $\beta^0$:
\begin{equation*}
    |S(\beta^0)|= o\left(\frac{n^{1/2-1/q-1/r}}{\log p}\right) \; ,
\end{equation*}
where $q$ and $r$ are related to moment conditions and are the price we pay for not assuming subgaussian designs.

We also require some level of sparsity of the vector $\theta^{(j)}$ in equation \eqref{eq:convergence} above, which is the population counterpart of the $j$-th column of the matrix $M$, computed with either Method 1 or Method 2. Specifically, we require:
\begin{equation*}
\begin{gathered}
    \|\theta^{(j)}\|_1 = o\left(\frac{n^{1/2-1/r}}{\sqrt{\log p}}\right) \mathrm{for \; Method \; 1} , \\
    s_\theta := \|\theta^{(j)}\|_0 = o\left(\frac{n^{1/2-1/r-1/q}}{{\log p}}\right) \mathrm{for \; Method \; 2} .
\end{gathered}
\end{equation*}
The difference between methods is mainly due to the presence of the approximation error $\rho$ in equation \eqref{eq:linear_model} for Method 2, which we will assume to be absent in Method 1. Clearly, the inclusion of this error term calls for more restrictive hypotheses. A thorough discussion of the requirements above and the other necessary assumptions is given in the next section.

\section{Assumptions}\label{sec:assumptions}

In the discussion below and the results to come, we implicitly assume the following setup: we consider a sequence of experiments in which $\{(X^{(n)}_i,Y_i^{(n)})\}_{i=1}^{n}$ is an i.i.d. sample for each $n \in \N$. Also $p = p(n) \gg n$ and $\beta^0 \in \R^{p(n)}$ with $s = |S(\beta^0)| = s(n)$. We will be explicit about parameters that are not allowed to depend on $n$, which we will also call absolute constants. These parameters will be listed in parentheses in the beginning of each assumption, so that ``{\sc Assumption x}~($a,b,c,d$)''~means that $a,b,c,d$ are the absolute constants involved in Assumption x. Any other parameters in our assumptions may in principle depend on $n$. The first set of hypotheses is concerned with the errors $\epsilon$ and $\rho$. 

\begin{assumption}[$\sigma_\epsilon$, $q'$, $M_{q'}$, $\sigma_{\rho}$]
\label{cond1}
Assume $\left\{(X_i,\epsilon_i,\rho_i)\right\}_{i=1}^{n}$, as given in equation \eqref{eq:linear_model}, are independent and identically distributed. Furthermore, assume that the following hold (almost surely): (i) For every $ i \in [n]$, $\; \Ex[\epsilon_i|X_i]=0$; (ii) $\max_{i \in [n]}\mathbb{V} [\epsilon_i|X_i] \leq \sigma_\epsilon^2 $ for some absolute constant $\sigma^2_{\epsilon} > 0$; (iii) $ \; \max_{i \in [n]} \Ex[|\epsilon_i|^{q'}|X_i] \leq M_{q'}^{q'} $ for some ${q'}>2$ and $M_{q'}>0$ both independent of $n$; (iv) For Method 1, assume $\rho = 0$ whereas for Method 2 take $\Ex \rho_i^2 \leq \sigma_\rho^2s_{\theta}/n$ for every $i \in [n]$ with $\sigma_\rho^2 > 0$ some absolute constant, where $s_{\theta}\leq n$.
\end{assumption}

These conditions are more general than the traditional lasso inference setup in the sense that they allow for dependency between $X$ and $\epsilon$. For example, \cite{10.1111/rssb.12026, JMLR:v15:javanmard14a,f29b4853-f902-378f-99ba-4ec108f6737e,10.1214/20-EJS1713,10.1214/17-AOS1630,10.3150/21-BEJ1348} all assume $\epsilon$ to be independent; in some cases, it is also assumed to follow a gaussian or sub-gaussian distribution, which we do not impose. Importantly, these works do not account for the presence of approximation errors in the linear model as in equation \eqref{eq:linear_model}. \cite{JMLR:v15:javanmard14a} does not assume identically distributed observations, but still require independent errors and independence of the rows of the sub-gaussian design matrix. On the other hand, \cite{10.1214/16-AOS1448,https://doi.org/10.3982/ECTA9626,10.1257/jep.28.2.29,10.1093/restud/rdt044} all allow for weaker dependence requirements and non-identically distributed observations, but still implicitly impose bounded or sub-gaussian design assumptions, which are not present in this work. In particular, \cite{https://doi.org/10.3982/ECTA9626,10.1257/jep.28.2.29} rely on an Instrumental Variable model with `double selection' to achieve such weak conditions. \cite{10.1093/restud/rdt044} also applies a `double selection' procedure, using the lasso estimator twice followed by a least squares estimator. The existence of an approximation error term is considered in these latter works, under hypotheses that matches the one given in point (iv) of Assumption \ref{cond1} for Method 2.

We highlight that the assumption of i.i.d. samples above greatly simplify our proofs, allowing for us to focus on concentration aspects of its self-normalized data hypothesis.
When this assumption is not present, the works mentioned above employ intricate proofs based on Lindeberg-Feller's Central Limit Theorem. Ours, on the other hand, are based on the classical Central Limit Theorem. Lastly, we point out that assumption (iii) is used mainly as a natural way to ensure lasso consistency in this dependent setup without the use of strong concentration inequalities for $\epsilon$.\\

The second set of assumptions is related to fast rates for the consistency of the base lasso estimator. Many different conditions have been proposed to ensure such rates, some famous examples including coherence \cite{10.1109/18.959265,10.1109/TIT.2010.2048506}, sparse eigenvalues \cite{https://doi.org/10.3982/ECTA9626}, uniform uncertainty principles \cite{10.1109/TIT.2006.885507,10.1214/009053606000001523} and restricted eigenvalues \cite{10.1214/08-AOS620}. We focus on this last property, which is among the least restrictive of the ones listed. We refer to \cite{10.1214/09-EJS506} for a comparative study which highlights this feature of the restricted eigenvalues assumption. Define:
\begin{equation}
\label{eq:def_cone}
    \Delta_{C,S} = \left\{ v  \in \mathbb{R}^p : \|v_{S^c}\|_1 \leq C\|v_{S}\|_1 \right\} \; .
\end{equation}
Then, the restricted eigenvalue of the triple $(A,S,C)$ is given by:
\begin{equation}
\label{eq:def_re}
    \mathrm{re}(A,S,C) = \max \left\{ R>0 : \forall v  \in \Delta_{C,S} , R^2 \|v_{S}\|_{2}^{2} \leq  v^T A v\right\} \; .
\end{equation}
With these definitions, our second set of assumptions is as follows:
\begin{assumption}[$g$, $\kappa_0$, $\kappa_1$, $h$, $h_{\star}$, $w$, $k$]
\label{cond2}
Suppose $\Ex[\bar{X}_{1j}^{2g}] < + \infty$ for every $j \in \left[p\right]$ and some $g > 2$ independent of $n$. Also, assume $\kappa_0 \leq \min_{k \in [p]}\Ex\left[\bar{X}_{1k}^2\right]$ and $\max _{k \in [p]} \Ex\left[\bar{X}_{1k}^4\right] \leq \kappa_1$ for $\kappa_0, \kappa_1 > 0$ absolute constants. Further assume that there exist $h, h_{\star} \in \left( 1,+\infty \right)$ and $w, k>0$ absolute constants such that:
\begin{equation}
\label{eq:cond2_prod}
    \forall v \in \mathbb{R}^p : \|v\|_0 \leq n \Rightarrow \sqrt{\mathbb{E}\left[\left(v^T\bar{X}_1\right)^4\right]}\leq h v^T \bar{\Sigma} v \; ,
\end{equation}
\begin{equation}
\label{eq:cond2_moment}
    \forall j \in \left[p\right] : \mathbb{E}\left[\bar{X}_{1j}^{2g}\right]^{1/g}\leq h_{\star}\; ,
\end{equation}
and
\begin{equation}
\label{eq:cond_restrict}
\mathrm{re}(\bar{\Sigma},S(\beta^0),w)\geq k \; .
\end{equation}
\end{assumption}

Equation \eqref{eq:cond2_prod} holds for a wide range of design matrices $\Sigma$, as mentioned in \cite{Oliveira2016}. In particular, if the entries of $X_1$ are all independent with finite fourth moment, a simple calculation shows that it is enough to consider
\begin{equation*}
    h =  6 \vee \left(\max_{j \in [p]: \mathbb{E}\bar{X}_{1j}^2 > 0} \frac{\sqrt{\mathbb{E}\bar{X}_{1j}^4}}{\mathbb{E}\bar{X}_{1j}^2}\right) = 6 \vee \left(\max_{j \in [p]: \sigma_j^2 > 0} \frac{\sqrt{\mathbb{E}({X}_{1j}-\mathbb{E}X_{1j})^4}}{\sigma_j^2}\right). 
\end{equation*}
The same calculation also applies to the case of four-wise independent entries. More importantly, it can be checked that \eqref{eq:cond2_prod} holds for sub-gaussian and log-concave vectors, as well as their affine transformations. In fact, one can easily see that condition \eqref{eq:cond2_prod} is invariant by linear transformations. 
Equation \eqref{eq:cond2_moment} is only a slight modification of \eqref{eq:cond2_prod} when applied to the canonical basis vectors of $\R ^p$, the difference being that we need moments of order strictly greater than 4. It also holds for the cases mentioned previously. 

The connection between these assumptions and the restricted eigenvalue property is made in \cite{Oliveira2016}. There, it is shown that, under data `standardization', one can obtain restricted eigenvalues without assuming sub-gaussian or bounded designs, as usually done in the reference literature. Some works that assume this kind of stricter tail behavior include \cite{10.3150/21-BEJ1348,10.1214/17-AOS1630,JMLR:v15:javanmard14a,10.1214/20-EJS1713,f29b4853-f902-378f-99ba-4ec108f6737e,10.1111/rssb.12026}. It is also implicitly assumed in \cite{10.1257/jep.28.2.29,10.1093/restud/rdt044,https://doi.org/10.3982/ECTA9626} by demanding bounded restricted or even sparse eigenvalues (see, for instance, Comment 3.2 in \cite{10.1093/restud/rdt044}).

The assumption on the restricted eigenvalue in equation \eqref{eq:cond_restrict} is common in our high-dimensional setting. We draw attention to the fact that this condition is much weaker than requiring the design matrix ${\Sigma}$ to be invertible. This stronger hypothesis is usually the one assumed throughout the reference literature on debiased lasso. \cite{JMLR:v15:javanmard14a}, for example, requires the smallest (global) eigenvalues of $\Sigma$ be bounded away from zero by some positive absolute constant. There, it is also assumed that the highest eigenvalue is bounded by some absolute constant. Similar hypotheses appear in \cite{10.1214/16-AOS1448,10.3150/21-BEJ1348,f29b4853-f902-378f-99ba-4ec108f6737e,10.1214/17-AOS1630}, among many others. 

\begin{assumption}[$q$, $q'$, $r$]
\label{cond3}Set $s = \|\beta^0\|_0 = |S(\beta^0)|$, with $\beta^0$ as in equation \eqref{eq:linear_model}. Assume:
\begin{equation}
\label{eq:sparse_truth}
    s = o\left(\frac{n^{1/2-1/q-1/r}}{\log p}\right) \; ,
\end{equation}
with $q'>q>2$ independent of $n$, for $q'$ as in Assumption \ref{cond1}, and $r>2$ also independent of $n$.
\end{assumption}

This sparsity condition is close to the usual sparsity requirement for high-dimensional inference, that is 
\begin{equation}
\label{eq:lit_sparse}
    s = o\left(\frac{\sqrt n}{\log p}\right) \; .
\end{equation}
This slightly weaker requirement is ubiquitous in the literature; a non-exhaustive list of works with such assumption includes \cite{JMLR:v15:javanmard14a,10.1214/20-EJS1713,f29b4853-f902-378f-99ba-4ec108f6737e,10.1111/rssb.12026,10.1257/jep.28.2.29,https://doi.org/10.3982/ECTA9626,10.1093/restud/rdt044}.
This highlights that the price we pay for not assuming light tails or independent errors are the additional terms with $r$ and $q$ in the exponent of $n$. 

An extensive amount of work has been done in order to approximate this rate to the much weaker requirements for prediction with the lasso, which assumes only $ s = o \left({n}/{\log p}\right)$. For instance, \cite{10.1214/17-AOS1630} is able to achieve better rates; namely $s = o\left({n}/{(\log p)^2}\right)$, but only in the case of gaussian designs with known covariance. \cite{10.3150/21-BEJ1348} gives optimal results in this same special scenario. To our knowledge, the more restrictive rate has not been weakened even for general sub-gaussian designs. Importantly, \cite{10.1214/16-AOS1461} showed that, from a minimax perspective, $\sqrt{n}$- confidence intervals  for a single coordinate can only be obtained under the strong assumption on $s$ given in equation \eqref{eq:lit_sparse}.

\begin{assumption}[$K_0$, $K_1$, $r'$, $h_0$, $K_2$, $w'$, $k'$]
\label{cond4}
Fix $j \in [p]$ and assume there exists $\theta^{(j)} \in \R^p$ such that $\bar{\Sigma}\theta^{(j)} = e_j$. Suppose (i) $K_0 \leq \min_{k \in [p]}\Ex\left[(\bar{X}_{1k}\bar{X}_1^T\theta^{(j)})^2\right]$ for $K_0 > 0$ some absolute constant; (ii) $\max _{k \in [p]} \Ex\left[(\bar{X}_{1k}\bar{X}_1^T\theta^{(j)})^4\right] \leq K_1$, for some positive absolute constant $K_1$; (iii)  there exists $r'>2$ and $h_0>1$ independent of $n$ such that:
\begin{equation}
\label{eq:bound_theta}
\mathbb{E}\left[|\langle\bar{X}_{1},\theta^{(j)} \rangle|^{r'}\right]\leq h_{0} \langle \theta^{(j)},\bar{\Sigma} \theta^{(j)} \rangle^{r'/2} = h_{0} \left(\theta_{j}^{(j)}\right)^{r'/2} \; ;
\end{equation}
(iv) Assume $\max_{k \in [p]}\mathbb{E}\left[|\bar{X}_{1k}|^{3r'}\right] \leq K_2$, for some absolute constant $K_2>0$;
(v) Only for Method 2, also assume:
\begin{equation}
\label{eq:cond_restrict_2}
\mathrm{re}(\bar{\Sigma},S(\theta^{(j)}),w')\geq k' \; ,
\end{equation}
for some positive absolute constants $w'$ and $k'$.
\end{assumption}

As before in the conditions of Assumption \ref{cond2}, the requirement in equation \eqref{eq:bound_theta} holds for a large class of random vectors, such as sub-gaussians, log-concave distributions and ones with four-wise independent entries. Importantly, Assumption \ref{cond4} requires only the existence of a surrogate for the $j$-th column of `$\Sigma^{-1}$' if one wishes to inference on this coordinate, which again allows for $\Sigma$ to be singular.
\begin{assumption}[$r$, $r'$, $K_{\theta}$]
\label{cond5}
For $\theta^{(j)}$ and $r'$ as given in Assumption \ref{cond4}, take $r$ as in Assumption \ref{cond3} satisfying $r'>r>2$. Suppose
\begin{equation}
\label{eq:bounded_theta}
    \theta_{j}^{(j)} \leq K_{\theta} \; ,
\end{equation}
for some positive absolute constant $K_{\theta}$. For Method 1 assume:
\begin{equation}
\label{eq:sparse_theta}
    \|\theta^{(j)}\|_1 = o\left(\frac{n^{1/2-1/r}}{\sqrt{\log p}}\right) .
\end{equation}
For Method 2, take:
\begin{equation}
\label{eq:sparse_theta_2}
    \|\theta^{(j)}\|_1 = o\left(\frac{n^{1/2-1/r}}{\sqrt{\log p}}\right) \; , \; \; s_{\theta} = \|\theta^{(j)}\|_0 = o\left(\frac{n^{1/2-1/r - 1/q}}{{\log p}}\right) ,
\end{equation}
with $s_\theta$ as in Assumption \ref{cond1}.
\end{assumption}

The requirement given in equation \eqref{eq:bounded_theta} is connected to the smallest eigenvalue of the matrices $\bar{\Sigma}$ and ${\Sigma}$. Indeed, if we were willing to assume for a moment that $\lambda_{min}(\bar{\Sigma})\geq \lambda_0>0$, it would hold that:
\begin{equation*}
    \theta_{j}^{(j)} \leq \frac{\sigma_j}{\lambda_0} \; ,
\end{equation*}
which implies condition \eqref{eq:bounded_theta} as long as the variance of a single coordinate of $X_1$ stays bounded. Again, we recall that such bounded eigenvalues requirements are widely present in other works such as \cite{JMLR:v15:javanmard14a,10.1214/16-AOS1448,10.3150/21-BEJ1348,f29b4853-f902-378f-99ba-4ec108f6737e,10.1214/17-AOS1630,10.1214/20-EJS1713}. 

Likewise, the assumption in equation \eqref{eq:sparse_theta} for Method 1 is very close to the usual sparsity requirements in high-dimensional inference; to wit:
\begin{equation}
\label{eq:lit_sparse_sigma}
\max_{j \in [p]}\|\Sigma^{-1}e_j\|_0 = o \left({\frac{n}{\log p}}\right) .
\end{equation}
This is required in the original version of the debiased lasso as proposed in \cite{f29b4853-f902-378f-99ba-4ec108f6737e}; \cite{10.1214/16-AOS1448} also employs a similar assumption. \cite{JMLR:v15:javanmard14a} is able to remove this sparsity requirement, but still under suppositions such as sub-gaussian design, non-singular covariance matrix with bounded eigenvalues and independent errors. \cite{10.1214/17-AOS1630}, which deals only with gaussian designs, demands a similar sparsity assumption in the case of unknown covariance. Additionally, \cite{10.3150/21-BEJ1348} argues that even in this simple scenario a sparsity close to \eqref{eq:lit_sparse_sigma} is necessary. The unknown covariance case is contemplated in our results without further hypotheses than the ones given in this section.

Further note that, under bounded eigenvalues requirement, \eqref{eq:lit_sparse_sigma} implies:
\begin{equation}
\label{eq:l1_bound_eigenvalue}
    \lambda_0{\max_{j \in [p]}\|\Sigma^{-1}e_j\|_1} = o \left(\sqrt{\frac{n}{\log p}}\right) \; ,
\end{equation}
which closely matches our assumption in equation \eqref{eq:sparse_theta}, recalling the definition of $\theta^{(j)}$ as a substitute for the $j$-th column of $\Sigma^{-1}$.
This again highlights that the main differences between the usual sparsity rates and our assumptions for Method 1 are the additional terms in the exponent of $n$ and the lack of necessity to assume sparsity of all columns of $\Sigma^{-1}$, or even its existence, if one is interested in a single coordinate. Also, it should be clear that, even though the sparsity assumptions in this work are slightly stronger, they get progressively less restrictive the more one allows for higher values of $r$ and $q$. In the particular case of sub-gaussian designs, $r$ and $q$ can be taken arbitrarily large.

Equation \eqref{eq:sparse_theta_2} for Method 2, however, includes an additional requirement on the sparsity of $\theta^{(j)}$. Similarly to the discussion above, under strong hypothesis on $\bar{\Sigma}$, it can be shown that this assumption on $\|\theta^{(j)}\|_0$ implies the one on $\|\theta^{(j)}\|_1$, although the implication is not true in general. We state the hypothesis on $\|\theta^{(j)}\|_0$ and $\|\theta^{(j)}\|_1$ separately since we do not impose such strong restrictions.

Still, equation \eqref{eq:lit_sparse_sigma} shows that, under non-singular $\Sigma$ with bounded eigenvalues, the usual assumption in the `debiased lasso' literature removes the square root inside the order term for $\|\theta^{(j)}\|_0$. This is the price to pay for considering the additional approximation error $\rho$ as given in the linear model \eqref{eq:linear_model}. We highlight that the works that assume the weaker rate given in \eqref{eq:lit_sparse_sigma} do not consider this difficult scenario. As mentioned above, some examples for this case include \cite{10.3150/21-BEJ1348,10.1214/17-AOS1630,JMLR:v15:javanmard14a,10.1214/20-EJS1713,f29b4853-f902-378f-99ba-4ec108f6737e,10.1111/rssb.12026}. Works that do consider approximation errors also require a stronger hypothesis. For example, \cite{10.1093/restud/rdt044} asks for $s_{\theta} = o(\sqrt{n}/\log p)$, among other requirements we do not impose. This highlights how both methods we consider demand assumptions that are very close to the state of the art, even under weaker hypotheses on the data distribution for the data and regression errors.

\section{Main result}\label{sec:results}

With the conditions in the previous section we are ready to write our main result.
\begin{theorem}
\label{thm:main}
Take $\hat{\beta}^u_j$ as defined above in equation \eqref{eq:debiased}, using either Method 1 or Method 2 to compute the vector $m^{(j)}$. Set 
\begin{equation*}
    \lambda =  \frac{c_0\sqrt{\log p}}{ n^{1/2-1/q}} \; , \; \; \zeta = c' n^{1/r} \; , \; \; \mu =  \frac{C\sqrt{\log p}}{n^{1/2-1/r}} \; ,  \; \lambda_M =  \frac{c_2\sqrt{\log p}}{n^{1/2-1/r}} \; ,
\end{equation*}
where $c_0,c',c_2, C > 0$ are sufficiently large absolute constants. Fix $j \in [p]$ and consider Assumptions \ref{cond1} through \ref{cond5}.
Then:
\begin{equation*}
    \sqrt{n} \left(\hat{\beta}^{u}_{j}/\hat{\sigma}_j-\beta^0_j\right) \xrightarrow{d}\mathcal{N}(0,\mathbb{E}[\langle \bar{X}_{1},\theta^{(j)}\rangle^2\epsilon_{1}^{2}]/\sigma_j^2) \; .
\end{equation*}
\end{theorem}

Theorem \ref{thm:main} is a consequence of other intermediary results discussed in Appendix \ref{appendix:A}. Its proof is also in Appendix \ref{appendix:A}. We highlight that, although the theorem above is given as an asymptotic statement, some of the results given in the appendix are non-asymptotic. 

\begin{remark}
    The above theorem gives a confidence interval for a single coordinate. It is direct to generalize it to a bounded number of coordinates of $\beta^0$, but one cannot hope to extend it further under weak moment assumptions in our very high-dimensional setup. Recent works, such as \cite{KOCK_2024110149}, have shown the impossibility of obtaining such uniform results under an exponential growth of $p$ relative to $n$ in this scenario; it is necessary to assume a polynomial relation of the form $p = o(n^c)$.
\end{remark}

\begin{remark}
    For simplicity, we state Theorem \ref{thm:main} with precise choices for the regularization parameters, as is typical for lasso-type methods. Our proofs can in principle accommodate a range of such choices. In practice, one may tune these parameters. Cross-validation techniques can be used for this purpose. The choice of the lasso regularization $\lambda$ is standard. The other parameters are related to moments of the standardized random variables $(X_{1j}-\mu_j)/\sigma_j$, which are ``dimension-free''~quantities that can be bounded under weak assumptions. Guided by the theorems below and our simulation experiments in Section \ref{sec:experiments}, we suggest setting $c' = 2$ and $c_0,C,c_2$ close to $0.5$ for $r=q \approx 5$. Of course, these should balance the choices of $r$ and $q$. A lower value for $C$ should be preferred for `harder' problems since it imposes a stronger restriction on the optimization problem for Method 1, while higher values of $c_0, c_2$ impose more regularization and should also be preferred in such cases.
\end{remark}

\begin{remark} The asymptotic variance for the limit distribution in Theorem \ref{thm:main} is bounded above by $\sigma_{\epsilon}^2 \theta_j^{(j)}/\sigma^2_j$, from Assumption \ref{cond1} and Assumption \ref{cond4}. This is indeed the optimal variance as obtained in \cite{f29b4853-f902-378f-99ba-4ec108f6737e}. There, it is argued that the best linear unbiased estimator for $\beta^0_{j}$ has variance given by $\sigma_\epsilon^2/\mathbb{V} [X_{1j}-X_{1,-j}\gamma_j]$, where $X_{1,-j} \in \R^{p-1}$ is the vector $X_1$ without its $j$-th coordinate and $\gamma_j$ is the coefficient obtained by regressing the $j$-th column of $X$, $X_{.,j}$, on $X$ with its $j$-th column removed, $X_{-j}$. This follows from the Gauss-Markov Theorem and the observation that 
\begin{equation*}
    Y = \beta^0_j (X_{.,j} - X_{-j}\gamma_j) + X_{-j}(\beta^0_{-j} + \beta^0_j\gamma_j) + \epsilon \; 
\end{equation*}
is a linear submodel of \eqref{eq:linear_model} in the case $\rho=\gamma=0$. Then, by the definition of $\theta^{(j)}$ given in Assumption \ref{cond4} and standard linear algebra calculations, it can be shown that:
\begin{equation*}
    \frac{1}{\mathbb{V} [X_{1j}-X_{1,-j}\gamma_j]}= \frac{\theta^{(j)}_j}{\sigma^2_j} \; ,
\end{equation*}
which implies the claim for the variance of our estimator. We also highlight that this is equivalent to the variance obtained in \cite{10.1093/restud/rdt044} for the homoskedastic case.
\end{remark}

From this asymptotic distribution, we can easily establish valid confidence intervals. The variance term $\Ex \left[\langle \bar{X}_1,\theta^{(j)}\rangle^2\epsilon_1^2\right]$ is not known a priori and must be estimated from the data in order to construct such intervals in a practical application. Still, one can use:
\begin{equation*}
    \hat{I}_j = \left[ \frac{\hat{\beta}^u_j}{\hat{\sigma}_j} - \Phi^{-1}(1-\alpha/2)\sqrt\frac{\hat{V}}{\hat{\sigma}_j^2n}, \frac{\hat{\beta}^u_j}{\hat{\sigma}_j} + \Phi^{-1}(1-\alpha/2)\sqrt\frac{\hat{V}}{\hat{\sigma}_j^2n}\right] \; ,
\end{equation*}
where $\hat{V}$ is any consistent estimator of $\Ex \left[\langle \bar{X}_1,\theta^{(j)}\rangle^2\epsilon_1^2\right]$ and $\Phi$ is the cumulative function of the standard normal distribution. We argue that the plug-in estimator is a valid choice in our setting.

\begin{corollary}
\label{cor:interval_2}
Take $\hat{\beta}^u_j$ as defined above in equation \eqref{eq:debiased}, using either Method 1 or Method 2 to compute the vector $m^{(j)}$. Set 
\begin{equation*}
    \lambda =  \frac{c_0\sqrt{\log p }}{ n^{1/2-1/q}} \; , \; \; \zeta = c' n^{1/r} \; , \; \; \mu =  \frac{C\sqrt{\log p}}{n^{1/2-1/r}} \; ,  \; \lambda_M =  \frac{c_2\sqrt{\log p}}{n^{1/2-1/r}} \; ,
\end{equation*}
where $c_0,c',c_2, C > 0$ are sufficiently large absolute constants. Fix $j \in [p]$ and consider Assumptions \ref{cond1} through \ref{cond5}. Define:
    \begin{equation*}
        \hat{V}_n = \frac{1}{n}\sum_{i=1}^{n} \langle\tilde{X}_i,m^{(j)}\rangle^2 (Y_i - \tilde{X}_i^T\hat{\beta}- \hat{\gamma})^2
    \end{equation*}
    Then, the interval:
    \begin{equation}
    \label{eq:confidence_interval}
        \hat{I}_j = \left[ \frac{\hat{\beta}^u_j}{\hat{\sigma}_j} - \Phi^{-1}(1-\alpha/2)\sqrt\frac{\hat{V}_n}{\hat{\sigma}_j^2n}, \frac{\hat{\beta}^u_j}{\hat{\sigma}_j} + \Phi^{-1}(1-\alpha/2)\sqrt\frac{\hat{V}_n}{\hat{\sigma}_j^2n}\right] \, 
    \end{equation}
    is asymptotically valid with confidence level $1-\alpha \in (0,1)$, where $\Phi$ is the cumulative distribution function of the standard gaussian distribution. That is:
    \begin{equation*}
        \lim_{n \rightarrow \infty} \Pr \left[\beta_j^0 \in \hat{I}_j\right] = 1-\alpha \; .
    \end{equation*}
\end{corollary}
\begin{proof}[Proof of Corollary \ref{cor:interval_2}]
This follows directly from Slutsky's Theorem, Theorem \ref{thm:main} above and Lemma \ref{lemma:plug_in} in Appendix \ref{appendix:A}.
\end{proof}

Similarly, in terms of hypothesis testing, the following corollary follows easily.

\begin{corollary}[Proof omitted]
\label{cor:hypothesis}
Under the same notation and assumptions from Theorem \ref{thm:main}, consider the following hypotheses:
\begin{equation*}H_0 : \beta_j^0 = 0 \; , \; \; H_1 : \beta_j^0 \neq 0 \; .
\end{equation*}
Let $V_j = \mathbb{E}[\langle\bar{X}_{1},\theta^{(j)}\rangle^2\epsilon_{1}^{2}]$ and $\hat{V}_j$ be a consistent estimator for $V_j$. Define the test statistic:\begin{equation*}T_j = \frac{\sqrt{n} \hat{\beta}^u_j}{\sqrt{\hat{V}_j}}
\end{equation*}
For any fixed $\alpha \in (0,1)$, the test that rejects $H_0$ when $|T_j| > \Phi^{-1}(1-\alpha/2)$ has an asymptotic significance level of $\alpha$, where $\Phi$ is cumulative function of the standard gaussian distribution.
\end{corollary}

\begin{remark}
    The above corollary related to hypothesis testing is straightforward but only mentions how to construct a test for a given level $\alpha$. A test that randomly rejects the null hypothesis with probability $\alpha$ would achieve the same level, making relevant the trade-off between level and test power. For the sake of brevity, we mention that similar results to the ones obtained in \cite{JMLR:v15:javanmard14a} could be given for our methods, following the proof ideas in their analysis. In particular, control of familywise error rate (FWER) can also be achieved.
\end{remark}

\section{Numerical experiments}\label{sec:experiments}
We report results obtained by the proposed method in different scenarios, using synthetic data. We consider the linear model in equation \eqref{eq:linear_model}, but exclude the intercept term $\gamma^0$ for simplicity. Four different distributions for the i.i.d pairs $(X_i,\epsilon_i,\rho_i) \; , i \in [n]$ are used as made explicit below. For all of them we report results for different choices of the parameters $(n, p, s, b)$, where $b$ is a value assumed by all entries of $\beta^0$ in its support $S(\beta^0)$. Once the matrix $X$ is constructed, we sample a total of 500 independent realizations of $\epsilon$ and report the average result. The significance level is set $\alpha = 0.05$. Throughout the experiments, the constants in Theorem \ref{thm:main} are set to $c_0 = 0.3$, $c' = 2$, and $C = c_2 = 0.1$. For simplicity, we take $r=q=4.9$ in all cases, although they can clearly be taken much higher in some of them. These constants encapsulate the dependency on other absolute constants that appear in our assumptions. The studied metrics are as follows:
\begin{equation*}
\begin{gathered}
    \widehat{L} = \frac{1}{p} \sum_{j=1}^{p} \mathrm{AvgLength}(\hat{I}_j) \; , \; \; \widehat{L}_S = \frac{1}{s} \sum_{j \in S(\beta^0)} \mathrm{AvgLength}(\hat{I}_j) \; , \\
    \widehat{L}_{S^c} = \frac{1}{p-s} \sum_{j \in S(\beta^0)^c} \mathrm{AvgLength}(\hat{I}_j) \; ,
\end{gathered}
\end{equation*}
representing the average size of the confidence interval given in equation \eqref{eq:confidence_interval}, also discriminating by entries inside and outside the support of $\beta^0$. We also consider their average empirical coverage:
\begin{equation*}
\begin{gathered}
    \widehat{C} = \frac{1}{p} \sum_{j=1}^{p} \hat{\Pr} [\beta^0_j \in \hat{I_j}] \; , \; \; \widehat{C}_S = \frac{1}{s} \sum_{j \in S(\beta^0)} \hat{\Pr} [\beta^0_j \in \hat{I_j}] \; , \\ 
    \widehat{C}_{S^c} = \frac{1}{p-s} \sum_{j \in S(\beta^0)^c} \hat{\Pr} [\beta^0_j \in \hat{I_j}] \; .
\end{gathered}
\end{equation*}
The four data sampling strategies are as follows:
\begin{enumerate}[label=\Alph*.]
    \item The rows of $X$ are i.i.d. samples from a gaussian distribution $\mathcal{N}(0,\Sigma)$ where $\Sigma$ is a symmetric matrix with entries given by:
    \begin{equation}
    \label{eq:toepliz}
        \Sigma_{ij} = 
    \begin{cases}
    1 & \text{if } i = j \, , \\
    0.1 & \text{if } j \in \{i+1, \dots, i+5\} \cup \{i+p-5, \dots, i+p-1\}\\
    0 & \text{for all other } i \leq j 
    \end{cases}
    \; ,
    \end{equation}
    the errors $\epsilon_i \; , i \in [n]$ are i.i.d. and independent from $X$, sampled from a standard normal distribution and $\rho = 0$. This is a basic case, as done in \cite{JMLR:v15:javanmard14a} for example.
    
    \item We build a matrix $K$ of size $(n,p)$ with i.i.d. entries sampled from a t-student with 5 degrees of freedom. The entries are then scaled so they have unit variance to produce the matrix $\bar{K}$. $X$ is defined by $X = \bar{K} U^T$ where $\Sigma = U U^T$ with $\Sigma$ as given in equation \eqref{eq:toepliz} above. This guarantees that the rows of X are i.i.d. with covariance $\Sigma$; they are not samples of a multivariate t-student, however. The errors $\epsilon_i \; , i \in [n]$ are i.i.d. and independent from $X$, sampled from a t-student distribution with 5 degrees of freedom. We take $\rho=0$ in this setting.
    
    \item We employ a heteroscedastic design to evaluate the estimator under conditional variance fluctuations, as permitted by Assumption \ref{cond1}. The design matrix $X$ is generated as in the previous case, with the rows being i.i.d. and entries given by linear combination of t-student distributions with given covariance structure. The errors are constructed as $\epsilon_i = \xi_i \sqrt{0.5 + 0.5X_{i1}^2/\left(1+X_{i1}^2\right)} \; , i \in [n]$, where $\xi_i$ are sampled i.i.d. from a standard normal distribution and $\rho=0$. This ensures $\mathbb{E}[\epsilon_i|X_i]=0$ but allows the conditional variance $\mathbb{V}[\epsilon_i|X_i]$ to explicitly depend on the first feature, breaking the standard homoskedasticity assumption while respecting our bounded moment conditions.

    \item The data matrix $X$ and the errors $\epsilon$ are given as in configuration A. We set $\rho_i = \sigma_\rho (X_{i1}^2 -1) \sqrt{s_\theta/2n}  $, with $\sigma_\rho^2 = 5$ and $s_\theta = \sum_{j=1}^{p} \mathds{1}_{\Sigma^{-1}_{1j}>0.001} $. This configuration accounts for the presence of an approximation error due to a quadratic misspecification term. Importantly, $\rho$ satisfies the normal equations given in \eqref{eq:normal_equations} and the conditions in Assumption \ref{cond1}.
\end{enumerate}
The results are given in Table \ref{tab:results} below for the different data configurations and choices of $(n,p,s,b)$. The table demonstrates that both methods approximately achieve or surpass the nominal coverage level. Notably, the worst scenario is the one with non-zero approximation error in the smaller sample case ($n=600$). Still, both methods exhibit a performance close to the ideal one, with Method 2 being slightly superior in this regime in terms of coverage of the active set and interval length. Nonetheless, these results highlight how the methods adapt to multiple heavier-tailed and heteroskedastic error scenarios.

\begin{table}[htbp]
\centering
\setlength{\tabcolsep}{2.5pt}
\caption{Simulation results: the nominal 0.95 coverage is attained by both
methods, inside and outside the support, in the majority of settings. Both
methods have similar results in most cases. In the presence of quadratic misspecification, Method~2 performs slightly better in terms of coverage of the active set and interval lengths.}
\label{tab:results}
\begin{tabular}{l!{\hspace{8pt}}cccccc!{\hspace{8pt}}cccccc}
\toprule
 & \multicolumn{6}{c}{Method 1} & \multicolumn{6}{c}{Method 2} \\
\cmidrule(lr){2-7}\cmidrule(lr){8-13}
Data configuration & $\widehat{L}$ & $\widehat{L}_S$ & $\widehat{L}_{S^c}$ & $\widehat{\mathrm{C}}$ & $\widehat{\mathrm{C}}_S$ & $\widehat{\mathrm{C}}_{S^c}$ & $\widehat{L}$ & $\widehat{L}_S$ & $\widehat{L}_{S^c}$ & $\widehat{\mathrm{C}}$ & $\widehat{\mathrm{C}}_S$ & $\widehat{\mathrm{C}}_{S^c}$ \\
\midrule
\multicolumn{13}{l}{A. Gaussian} \\
(600, 1000, 10, 0.5) & 0.216 & 0.221 & 0.216 & 0.957 & 0.951 & 0.957 & 0.216 & 0.216 & 0.216 & 0.958 & 0.949 & 0.958 \\
(600, 1000, 30, 0.1) & 0.225 & 0.224 & 0.225 & 0.962 & 0.962 & 0.962 & 0.225 & 0.228 & 0.225 & 0.962 & 0.961 & 0.962 \\
(1500, 2000, 50, 0.5) & 0.129 & 0.129 & 0.129 & 0.963 & 0.952 & 0.963 & 0.128 & 0.129 & 0.128 & 0.962 & 0.942 & 0.962 \\
(1500, 2000, 25, 0.1) & 0.117 & 0.117 & 0.117 & 0.956 & 0.951 & 0.956 & 0.117 & 0.117 & 0.117 & 0.956 & 0.953 & 0.956 \\
\addlinespace
\multicolumn{13}{l}{B. t-Student} \\
(600, 1000, 10, 0.5) & 0.216 & 0.217 & 0.216 & 0.958 & 0.958 & 0.958 & 0.218 & 0.214 & 0.218 & 0.958 & 0.955 & 0.958 \\
(600, 1000, 30, 0.1) & 0.225 & 0.228 & 0.225 & 0.961 & 0.962 & 0.961 & 0.225 & 0.225 & 0.225 & 0.963 & 0.961 & 0.963 \\
(1500, 2000, 50, 0.5) & 0.128 & 0.129 & 0.128 & 0.962 & 0.954 & 0.962 & 0.130 & 0.132 & 0.130 & 0.965 & 0.938 & 0.966 \\
(1500, 2000, 25, 0.1) & 0.117 & 0.119 & 0.117 & 0.957 & 0.954 & 0.957 & 0.117 & 0.117 & 0.116 & 0.957 & 0.957 & 0.957 \\
\addlinespace
\multicolumn{13}{l}{C. Bounded heteroskedastic} \\
(600, 1000, 10, 0.5) & 0.181 & 0.188 & 0.181 & 0.961 & 0.956 & 0.961 & 0.183 & 0.191 & 0.183 & 0.962 & 0.966 & 0.962 \\
(600, 1000, 30, 0.1) & 0.193 & 0.198 & 0.193 & 0.968 & 0.968 & 0.968 & 0.193 & 0.199 & 0.193 & 0.967 & 0.959 & 0.967 \\
(1500, 2000, 50, 0.5) & 0.113 & 0.116 & 0.113 & 0.967 & 0.950 & 0.967 & 0.113 & 0.115 & 0.113 & 0.969 & 0.960 & 0.969 \\
(1500, 2000, 25, 0.1) & 0.099 & 0.100 & 0.099 & 0.960 & 0.964 & 0.960 & 0.098 & 0.101 & 0.098 & 0.959 & 0.952 & 0.959 \\
\addlinespace
\multicolumn{13}{l}{D. Quadratic misspecification} \\
(600, 1000, 10, 0.5) & 0.305 & 0.329 & 0.305 & 0.942 & 0.922 & 0.942 & 0.298 & 0.290 & 0.298 & 0.941 & 0.927 & 0.941 \\
(600, 1000, 30, 0.1) & 0.285 & 0.282 & 0.285 & 0.950 & 0.953 & 0.950 & 0.295 & 0.293 & 0.295 & 0.945 & 0.957 & 0.945 \\
(1500, 2000, 50, 0.5) & 0.150 & 0.150 & 0.150 & 0.958 & 0.931 & 0.959 & 0.150 & 0.149 & 0.150 & 0.958 & 0.936 & 0.958 \\
(1500, 2000, 25, 0.1) & 0.143 & 0.143 & 0.143 & 0.953 & 0.933 & 0.953 & 0.137 & 0.137 & 0.137 & 0.953 & 0.947 & 0.953 \\
\bottomrule
\end{tabular}
\end{table}

\begin{appendix}
\section{Appendix}\label{appendix:A}

\subsection{Lasso}\label{appendix:lasso_consistency} 
As a first step towards the proof of our main theorem, we give a proof of the prediction and estimation properties of the lasso. We state a very general result that will be used multiple times. This is a variant of several well-known results in the literature \cite{10.1214/08-AOS620,buhlmann2011statistics}.

\begin{lemma}
\label{lemma:general_lasso}
Take $\Theta \subset \R^p$, $A \in \R^{n \times p}$, $b \in \R^p$, $v \in \R^n$ and $l, h \in \R$. Suppose $A^T\mathds{1}_n =0$, where $\mathds{1}_n \in \R^n$ has all its entries equal to one. Define:
\begin{equation*}
    L: (m,\gamma) \in \Theta \times \R \mapsto \frac{1}{2} \|Am+\frac{\gamma}{\sqrt n}\mathds{1}_n\|_2^2 -b^Tm -v^TAm + l - h\gamma \; .
\end{equation*}
Take $\xi > 0$ and consider:
\begin{equation}
\label{eq:opt_problem_general_lasso}
    (\hat{\theta}_{\xi},\hat{\gamma}) = \argmin_{m \in \Theta, \gamma \in \R} L(m,\gamma) + \xi \|m\|_1  \; .
\end{equation}
Assume there exists $c>1$, $k_0 > 0$ absolute constants and $S \subset [p]$ such that $|S| = s$ which satisfy:
\begin{equation}
\label{eq:assump_restric}
    \|Ax\|_2^2 \geq k_0^2 \|x\|_2^2 \geq k_0^2 \|x_S\|_1^2/s\; ,
\end{equation}
for all $x \in \Delta_{S,C} = \{ y \in \R^p : \|y_{S^c}\|_1 \leq C \|y_S\|_1 \}$, where $C = 2(c+1)/(c-1)$. Further, assume there exists $\theta_0 \in \Theta$ for which:
\begin{equation}
\label{eq:assump_bound}
    \|A^TA \theta_0 - b \|_{\infty} \leq \xi/c \; .
\end{equation}
Then, there exists $L>0$, depending only on $c>1$ and $k_0>0$ such that:
\begin{equation}
\label{eq:goal1_lasso}
    \|A(\hat{\theta}_{\xi}-\theta_0)\|_2 \leq L\max \left\{\xi \sqrt{s}, \frac{\|(\theta_0)_{S^c}\|_1}{\sqrt{s}}+ \frac{\|v\|_2^2}{\xi \sqrt{s}}\right\} \; ,
\end{equation}
and
\begin{equation}
\label{eq:goal2_lasso}
    \|\hat{\theta}_{\xi}-\theta_0\|_1 \leq  L\max \left\{\xi s, {\|(\theta_0)_{S^c}\|_1}+ \frac{\|v\|_2^2}{\xi}\right\} \; .
\end{equation}
\end{lemma}

\begin{proof}[Proof of Lemma \ref{lemma:general_lasso}]
Take $\delta = \hat{\theta}_\xi - \theta_0$. By the definition of $(\hat{\theta}_{\xi},\hat{\gamma})$ and the assumption that $\theta_0 \in \Theta$:
\begin{equation*}
    L(\theta_0 + \delta,\hat{\gamma}) + \xi \|\theta_0 + \delta\|_1 \leq L(\theta_0,\hat{\gamma}) 
    + \xi \|\theta_0\|_1 \; ,
\end{equation*}
implying, since we assume $A^T\mathds{1}_n =0$:
\begin{equation*}
    \frac{1}{2}\|A\delta\|_2^2 
    + \langle A\delta, A\theta_0\rangle 
    - b^T\delta 
    - v^TA\delta
    \leq \xi ( \|\theta_0\|_1 - \|\theta_0 + \delta\|_1) \; ,
\end{equation*}
which can be rewritten as:
\begin{equation}
\label{eq:viability_ineq}
    \frac{1}{2}\|A\delta\|_2^2 
    + \delta^T (A^TA\theta_0 -b-A^Tv)
    \leq \xi ( \|\theta_0\|_1 - \|\theta_0 + \delta\|_1) \; .
\end{equation}
By the triangle inequality:
\begin{equation*}
    \|\theta_0+\delta\|_1 
    \geq \|(\theta_0)_S\|_1 
    - \|\delta_S\|_1 
    + \|\delta_{S^c}\|_1 
    - \|(\theta_0)_{S^c}\|_1 \; ,
\end{equation*}
which implies:
\begin{equation}
\label{eq:l1_decomposition}
    \|\theta_0\|_1 - \|\theta_0+\delta\|_1 
    \leq \|\delta_S\|_1 - \|\delta_{S^c}\|_1 + 2\|(\theta_0)_{S^c}\|_1 \; .
\end{equation}
Also, by the assumption in equation \eqref{eq:assump_bound} and Hölder's inequality:
\begin{equation}
\label{eq:bound_holder}
\delta^T (A^TA\theta_0 -b-A^Tv) \geq -\frac{\xi}{c}(\|\delta_S\|_1 + \|\delta_{S^c}\|_1) - \|A\delta\|_2 \|v\|_2 \; .
\end{equation}
Then, using equations \eqref{eq:bound_holder} and \eqref{eq:l1_decomposition} with equation \eqref{eq:viability_ineq}:
\begin{equation*}
    \frac{1}{2}\|A\delta\|_2^2 -\frac{\xi}{c}(\|\delta_S\|_1 + \|\delta_{S^c}\|_1) - \|A\delta\|_2 \|v\|_2 \leq \xi ( \|\delta_S\|_1 - \|\delta_{S^c}\|_1 + 2\|(\theta_0)_{S^c}\|_1 ) \; ,
\end{equation*}
which gives:
\begin{equation}
\label{eq:main_lasso}
    \frac{1}{2}\left(\|A\delta\|_2 - \|v\|_2\right)^2 - \frac{1}{2}\|v\|_2^2 \leq 2\xi \|(\theta_0)_{S^c}\|_1 +  \xi \left(\frac{c+1}{c}\right)\|\delta_S\|_1 - \xi \left(\frac{c-1}{c}\right)\|\delta_{S^c}\|_1 \; .
\end{equation}
This implies, bounding the leftmost quadratic term in the equation above by zero:
\begin{equation}
\label{eq:def_alpha}
     \|\delta_{S^c}\|_1 \leq \underbrace{\left(\frac{2c}{c-1}\right) \|(\theta_0)_{S^c}\|_1  +  \frac{c}{2\xi (c-1)}\|v\|_2^2}_{\alpha} +  \left(\frac{c+1}{c-1}\right)\|\delta_S\|_1 \; ,
\end{equation}
where we also use the assumption that $c>1$. Also, omitting the $\|\delta_{S^c}\|_1$ term in equation \eqref{eq:main_lasso}:
\begin{equation*}
    \left(\|A\delta\|_2 - \|v\|_2\right)^2  \leq \|v\|_2^2 + 4\xi \|(\theta_0)_{S^c}\|_1 +  2\xi \left(\frac{c+1}{c}\right)\|\delta_S\|_1 \; ,
\end{equation*}
which gives:
\begin{eqnarray}
\label{eq:bound_L}
    \nonumber \|A\delta\|_2^2 \leq 2\left[\left(\|A\delta\|_2 - \|v\|_2\right)^2+ \|v\|_2^2\right]  
    & \leq & 4\|v\|_2^2 + 8\xi \|(\theta_0)_{S^c}\|_1 +  4\xi \left(\frac{c+1}{c}\right)\|\delta_S\|_1 \; \\
    & \leq & 
    8 \left(\frac{c-1}{c}\right) \xi \alpha +  4\xi \left(\frac{c+1}{c}\right)\|\delta_S\|_1 \; ,
\end{eqnarray}
with $\alpha$ given in equation \eqref{eq:def_alpha}. Now, we consider two possible cases. If $\| \delta_S\|_1 \left(\frac{c+1}{c-1}\right) \leq \alpha$, we get, from equation \eqref{eq:def_alpha}, $\| \delta_{S^c}\|_1 \leq 2\alpha$. Then:
\begin{equation}
\label{eq:l1_case1}
    \|\delta\|_1 \leq 2 \alpha + \|\delta_{S}\|_1 \leq \left(\frac{3c+1}{c+1}\right) \alpha \; .
\end{equation}
Also, by equation \eqref{eq:bound_L}:
\begin{equation}
\label{eq:pred_case1}
    \|A\delta\|_2 \leq \sqrt{12 \left(\frac{c-1}{c}\right)\xi  \alpha} \leq 3\left(\frac{c-1}{c} \right) \xi \sqrt{s} + \frac{\alpha}{\sqrt{s}} \; ,
\end{equation}
by the AM-GM inequality. On the other hand, if $\| \delta_S\|_1 \left(\frac{c+1}{c-1}\right) > \alpha$, equation \eqref{eq:def_alpha} implies:
\begin{equation*}
     \|\delta_{S^c}\|_1 \leq 2 \left(\frac{c+1}{c-1}\right)\|\delta_S\|_1 \; .
\end{equation*}
This shows $\delta \in \Delta_{S,C}$ and, therefore, by equation \eqref{eq:bound_L} and equation \eqref{eq:assump_restric}:
\begin{equation*}
    \|A\delta\|_2^2 \leq 12 \left(\frac{c+1}{c}\right)\xi  \|\delta_S\|_1 \leq 12 \left(\frac{c+1}{c}\right)\xi   \sqrt{\frac{s}{k_0^2}}\|A\delta\|_2 \; ,
\end{equation*}
which implies:
\begin{equation}
\label{eq:pred_case2}
    \|A\delta\|_2 \leq 12 \left(\frac{c+1}{c}\right)\xi  \sqrt{\frac{s}{k_0^2}} \; ,
\end{equation}
and finally gives, again by equation \eqref{eq:assump_restric}:
\begin{equation}
\label{eq:l1_case2}
    \|\delta\|_1 \leq \left(\frac{3c+1}{c-1}\right)\|\delta_S\|_1 \leq \left(\frac{(3c+1)(c+1)}{c(c-1)}\right)\frac{12\xi s}{{k_0^2}} \; .
\end{equation}
Combining equations \eqref{eq:l1_case1} and \eqref{eq:l1_case2}, recalling the definition of $\alpha$, gives \eqref{eq:goal2_lasso}. Similarly, equations \eqref{eq:pred_case1} and \eqref{eq:pred_case2} imply the result in \eqref{eq:goal1_lasso}.
\end{proof}

The first case in which we apply Lemma \ref{lemma:general_lasso} is to obtain the classical lasso results for the problem given in equation \eqref{eq:lasso}. We will need a first probabilistic lemma. Since we do not assume independency of $X$ and the regression errors, nor that they are sub-gaussian, we need a modified version of the classical result bounding the probability of the event $\|\tilde{X}^T \epsilon/n\|_{\infty} \geq \lambda/c$. A similar result appears as an exercise in the book \cite{buhlmann2011statistics} (see Problem 6.2).

\begin{lemma}
\label{lemma:prob_lasso}
Take $\delta >0$ and $c>1$. Assume the conditions in Assumption \ref{cond1}. Set 
\begin{equation*}
    \lambda = \frac{c\sqrt{8\log(2p)} M_{q'}}{\delta^{1/q'} n^{1/2-1/q'}}\; ,
\end{equation*}
with $q'$, $M_{q'}$, $s$ and $\sigma_{\rho}$ as in the assumptions. Then:
\begin{equation*}
    \Pr \left[ c \| \tilde{X}^T\epsilon/n\|_{\infty} > \lambda\right] \leq \delta\; ,
\end{equation*}
with $\rho$ and $\epsilon$ as in the linear model \eqref{eq:linear_model}.
\end{lemma}
\begin{proof}[Proof of Lemma \ref{lemma:prob_lasso}]
By Markov's inequality and Nemirovski's moment inequality conditionally on $X$ (Lemma 14.24. in \cite{buhlmann2011statistics}):
\begin{eqnarray*}
    \Pr \left[ c \| \tilde{X}^T\epsilon/n\|_{\infty}> \lambda|X\right] & \leq &
    \frac{c^{q'}\Ex\left[\| \tilde{X}^T\epsilon\|_{\infty}^{q'}|X\right]}{n^{q'}\lambda^{q'}} \\ 
    & \leq & \frac{c^{q'}}{n^{q'}\lambda^{q'}}(8 \log(2p))^{{q'}/2}\mathbb{E}\left[\left(\max_{j \in [p]}\sum_{i=1}^{n} \tilde{X}^2_{ij} \epsilon_i^2\right)^{{q'}/2}\middle|X\right] \; .
\end{eqnarray*}
Then, applying Hölder's inequality and using the fact that columns are normalized, which implies $\frac{1}{n}\sum_{i=1}^{n} \tilde{X}_{ij}^2 = 1$ for every $j \in [p]$:
\begin{equation*}
        \Pr \left[ c \| \tilde{X}^T\epsilon/n\|_{\infty}> \lambda|X\right] \leq \frac{c^{q'}}{\lambda^{q'}}\left(\frac{8 \log(2p)}{n}\right)^{{q'}/2}\mathbb{E}\left[ \max_{i \in [n]} |\epsilon_i|^{{q'}}\middle|X\right] \; .
\end{equation*}
By the bound $\max_{i \in [n]} |\epsilon_i|^{q'} \leq \sum_{i=1}^{n} |\epsilon_i|^{q'}$ and hypothesis (iii) in Assumption 1:
\begin{eqnarray*}
    \Pr \left[ c \| \tilde{X}^T\epsilon/n\|_{\infty}> \lambda \middle|X\right] & \leq & \frac{c^{q'}}{\lambda^{q'}}\left(\frac{8 \log(2p)}{n}\right)^{{q'}/2} \left(\sum_{i=1}^{n}\mathbb{E}\left[ |\epsilon_i|^{q'}\middle|X\right]\right)\; \\ 
    & \leq & \frac{c^{q'}}{\lambda^{q'}}(8 \log(2p))^{{q'}/2} M_{q'}^{q'} n^{1-{q'}/2} \; .
\end{eqnarray*}
Finally, taking the expectation on $X$ and setting $\lambda$ as defined above finishes the proof.
\end{proof}

With both lemmata, we are able to deduce the following theorem.

\begin{theorem}
\label{thm:lasso_consistency}
Suppose Assumption 1 holds. Take $\delta_\lambda, \delta_\rho>0$. Set 
\begin{equation*}
    \lambda = \frac{c\sqrt{8\log(2p)} M_{q'}}{\delta_{\lambda}^{1/q'} n^{1/2-1/q'}} \; ,
\end{equation*}
for some absolute constant $c>1$. Also assume that, for some $\delta_r > 0$:
\begin{equation*}
\Pr[\mathrm{re}(\tilde{\Sigma}_n,S,C) \geq k_0]\geq 1-\delta_r \; ,
\end{equation*}
where $S = S(\beta^0)$, $C =(c+1)/(c-1)$ and $k_0>0$ is some absolute constant.
Then, given $\delta >0 $, the lasso estimator $\hat{\beta}$ as defined in equation \eqref{eq:lasso} satisfies both:
\begin{equation*}
    \mathbb{P}\left[\frac{1}{n} \|  \tilde{X}(\hat{{\beta}}-D_X^{1/2}\beta^0)\|_{2}^2 \leq L^2 \max \{\lambda^2s, \sigma^4_\rho s_\theta/\lambda^2 n^2 \delta_\rho^2\} \right] \geq 1-\delta_{\lambda}-\delta_{\rho} - \delta_r \; ,
\end{equation*}
and 
\begin{equation*}
    \mathbb{P}\left[\|\hat{\beta} - D_{X}^{1/2}\beta^0\|_1 \leq L \max \{\lambda s, s_\theta\sigma^2_\rho /\lambda n \delta_\rho\} \right] \geq 1-\delta_{\lambda}-\delta_{\rho}- \delta_r \; ,
\end{equation*}
with $L$ as in Lemma \ref{lemma:general_lasso}. In the case of Method 1 (problem \eqref{eq:opt1}) it holds that:
\begin{equation*}
    \mathbb{P}\left[\frac{1}{n} \|  \tilde{X}(\hat{{\beta}}-D_X^{1/2}\beta^0)\|_{2}^2 \leq L^2  \lambda^2s \right] \geq 1-\delta_{\lambda}- \delta_r \; ,
\end{equation*}
and 
\begin{equation*}
    \mathbb{P}\left[\|\hat{\beta} - D_{X}^{1/2}\beta^0\|_1 \leq L \lambda s \right] \geq 1-\delta_{\lambda}- \delta_r \; .
\end{equation*}
\end{theorem}
\begin{proof}[Proof of Theorem \ref{thm:lasso_consistency}] 
We first write problem \eqref{eq:lasso} as an instance of the optimization problem in \eqref{eq:opt_problem_general_lasso}. This can be done by mapping $A = \tilde{X}/\sqrt{n}$, $v=\rho/\sqrt{n}$, $b= \tilde{X}^T(Y-\rho)/n$, $l =  \frac{1}{2n} \|Y\|_2^2$, $\theta_0 =  D_X^{1/2}\beta^0$, $h=\frac{1}{n}\langle  \mathds{1}_n, Y\rangle$ and $\xi = \lambda$. Indeed, since $\tilde{X}^T\mathds{1}_n = 0$ by the data standardization step:
\begin{align*}
    \frac{1}{2n} \|\tilde{X}\beta +\gamma \mathds{1}_n-Y\|_{2}^{2}  & = \frac{1}{2n} \|\tilde{X}\beta +\gamma \mathds{1}_n\|_{2}^{2} + \frac{1}{2n}\|Y\|_{2}^{2} - \frac{1}{n}\langle \tilde{X}\beta +\gamma \mathds{1}_n, Y\rangle \\
    & = \frac{1}{2} \|\frac{\tilde{X}}{\sqrt{n}}\beta +\frac{\gamma}{\sqrt{n}}\mathds{1}_n\|_{2}^{2} + \frac{1}{2n}\|Y\|_{2}^{2} - \frac{1}{n}\langle \tilde{X}\beta , Y\rangle - \gamma \frac{1}{n}\langle  \mathds{1}_n, Y\rangle \\
    & = \frac{1}{2} \|\frac{\tilde{X}}{\sqrt{n}}\beta +\frac{\gamma}{\sqrt{n}}\mathds{1}_n\|_{2}^{2} + \frac{1}{2n}\|Y\|_{2}^{2} - \frac{1}{n}\langle \beta , \tilde{X}^T(Y-\rho)\rangle \\
    & - \langle \frac{\tilde{X}}{\sqrt n}  \beta ,\frac{\rho}{\sqrt{n}}\rangle - \gamma \frac{1}{n}\langle  \mathds{1}_n, Y\rangle
\end{align*}
Then, notice the hypothesis in equation \eqref{eq:assump_bound} is precisely the content of Lemma \ref{lemma:prob_lasso} since:
\begin{equation*}
    \|A^TA \theta_0 - b \|_{\infty} =  \|\tilde{X}^T{X} \beta^0/n - \tilde{X}^T(Y-\gamma^0\mathds{1}_n)/n \|_{\infty} = \|\tilde{X}^T\epsilon/n \|_{\infty}\; ,
\end{equation*}
where again we use the orthogonality $\tilde{X}^T\mathds{1}_n = 0$, combined with the observation that $\frac{1}{n}\tilde{X}^TX = \frac{1}{n}\tilde{X}^T \tilde{X}D_X^{1/2}$. Therefore:
\begin{equation*}
    \Pr[\|A^TA \theta_0 - b \|_{\infty} \leq \lambda/c] \geq 1 - \delta_{\lambda} \; .
\end{equation*}
Also, by the definition of the restricted eigenvalue (see equation \eqref{eq:def_re}) and the hypothesis on $\mathrm{re}(\tilde{\Sigma}_n,S,C)$ given in the statement of the theorem, condition \eqref{eq:assump_restric} is satisfied for $A = \tilde{X}/\sqrt{n}$, $S = S(\beta^0)$ with probability at least $1 - \delta_r$. Furthermore, since $D_X$ is a diagonal matrix and $S$ is the support of $\beta^0$, we conclude $\|(\theta_0)_{S^c}\|_1 = 0$. Finally:
\begin{equation*}
    \Pr \left[ \|v\|_2^2 > s_\theta\sigma^2_\rho /n \delta_\rho\right] \leq \frac{n \delta_{\rho} \Ex \|v\|_2^2}{s_\theta\sigma_\rho^2} \leq \delta_\rho \; ,
\end{equation*}
by Markov's inequality and point (iv) of Assumption 1. Applying Lemma \ref{lemma:general_lasso} and a union bound finishes the proof of the first statement in the theorem. In the case of Method 1, $\rho =0$ from Assumption 1, then $v=0$ in the analysis above and the second conclusion follows by the same argument.
\end{proof}

\subsection{Restricted eigenvalue property}\label{appendix:restricted_eigenvalue}

The bounds obtained for the error of the lasso estimator in Theorem \ref{thm:lasso_consistency} assumed that, with high probability, $\mathrm{re}(\tilde{\Sigma}_n,S(\beta^0),C) \geq k_0$. Following the classical analysis for the lasso, we still need to show that this quantity can be controlled properly. This is known to be the case under bounded or sub-gaussian designs \cite{pmlr-v23-rudelson12,Zhou2009RestrictedEC}. The next result shows that self-normalization can be used to weaken these assumptions. 
Our main result in this section is a corollary of Theorem 5.2 from \cite{Oliveira2016} that also exploits self-normalization. For completeness, we reproduce it below.

\begin{theorem}[Theorem 5.2 of \cite{Oliveira2016}]
\label{thm:restricted_eigenvalues_adaptation}
Let $X_1, \dots, X_n \in \mathbb{R}^p$ be independent and identically distributed random vectors whose coordinates have $2q$ moments for some $q > 2$. Define $\Sigma = \Ex [ X_1 X_1^T]$ and $\hat{\Sigma}_n =  \frac{1}{n} \sum_{i=1}^{n} X_i X_i^T$, and assume that ${h}, {h}_* \in (1, +\infty)$ are such that
\begin{equation*}
    \forall v \in \mathbb{R}^p : |v|_0 \leq n \implies \sqrt{\mathbb{E}\left[(v^T X_1)^4\right]} \leq {h} v^T \Sigma v
\end{equation*}
and
\begin{equation*}
    \forall 1 \leq j \leq p : \mathbb{E}\left[X_{1,j}^{2q}\right]^{\frac{1}{q}} \leq {h}_* \mathbb{E}\left[X_{1,j}^2\right].
\end{equation*}
Define diagonal matrices $\hat{D}_{2,n}$ and $D_2$ corresponding to the diagonals of $\hat{\Sigma}_n$ and $\Sigma$ (respectively). Set
\begin{equation*}
    \hat{\mathbf{X}}_n \equiv \hat{D}_{2,n}^{-1/2} \hat{\Sigma}_n \hat{D}_{2,n}^{-1/2} \quad \text{\small (this is the design matrix of $x_i \equiv \hat{D}_{2,n}^{-1/2} X_i, \, 1 \leq i \leq n$)}
\end{equation*}
and $\mathbf{X} \equiv D_2^{-1/2} \Sigma D_2^{-1/2}$ with the convention that the $(j,j)$-th entry of $\hat{D}_{2,n}^{-1/2}$ (resp. $D_2^{-1/2}$) is zero whenever the corresponding entry of $\hat{D}_{2,n}$ (resp. $D_2$) is zero. Assume that $\delta, \varepsilon \in (0, 1/2)$ and $S \subset \{1, \dots, p\}$ with cardinality $|S| = s$, and set
\begin{equation*}
    \tilde{\alpha} = \alpha \sqrt{\frac{1+\varepsilon}{1-\varepsilon}} \quad \text{and} \quad K \equiv 784 \left[(1+\varepsilon)(1+\alpha)^2\right] {h}^2.
\end{equation*}
Finally, assume
\begin{equation*}
    n \geq \max \left\{ \left( \frac{K (1 + 2 \ln(p / 4\delta))}{\mathrm{re}(\mathbf{X}, S, \tilde{\alpha})^2 \varepsilon^4} \right) s, \; \frac{4q^2 \, 3^{2/q} s^{2/q}}{\delta^{2/q}} \right\}.
\end{equation*}
Then the following three properties hold simultaneously with probability $\geq 1 - \delta$:
\begin{itemize}
    \item Let $x \in \mathbb{R}^p$ be any vector such that $\hat{D}_{2,n} x \in \Delta_{\alpha, S}$. Then $D_2 x \in \Delta_{\tilde{\alpha}, S}$,
    \item For any $x$ as above, $x^T \hat{\Sigma}_n x \geq (1-\varepsilon) x^T \Sigma x$,
    \item $\mathrm{re}(\hat{\mathbf{X}}_n, S, \alpha) \geq (1-\varepsilon)^2 \, \mathrm{re}(\mathbf{X}, S, \tilde{\alpha})$,
\end{itemize}
where $\Delta_{\alpha, S}, \Delta_{\tilde{\alpha}, S}$ and $\mathrm{re}(\hat{\mathbf{X}}_n, S, \alpha), \mathrm{re}(\mathbf{X}, S, \tilde{\alpha})$ are defined as in equations (10) and (11).
\end{theorem}

Before stating our main result on restricted eigenvalues, we begin with a useful lemma derived from concentration results for self-normalized processes:
\begin{lemma}
\label{lemma:self_norm1}
Assume, for every $k \in [p]$, $\{X_{ik}\}_{i=1}^{n}$ are independent and identically distributed random variables. Define:
\begin{equation*}
\hat{\mu}_k = \frac{1}{n} \sum_{i=1}^{n} X_{ik} \;, \; \; 
\hat{\sigma}_k^2 =\frac{1}{n}\sum_{i=1}^{n}\left(X_{ik} - \hat{\mu}_k\right)^2 \; .
\end{equation*}
Assume that, for every $k \in [p]$, $0<\Ex({X}_{1k} - \mathbb{E}{X}_{1k})^4 < + \infty$ and take $\sigma^2_k = \mathbb{E}({X}_{1k}-\mathbb{E}X_{1k})^2>0$. Then, given $t>0$, for any $k \in [p]$:
\begin{equation*}
    \Pr\left[\frac{|\mathbb{E}X_{1k}-\hat\mu_k|}{\hat{\sigma}_k}
      \geq\sqrt{\frac{24(1+t)}{n}}\right] \leq 4 e^{-t} +  2e^{-\frac{n\sigma^4_k}{2\sqrt{2}\mathbb{E}({X}_{1k} - \mathbb{E}{X}_{1k})^4}}  + 4 e^{1-\frac{n}{8}} \; .
\end{equation*}
Furthermore, take $a>24$ and assume $p\geq e^{24/a}$, $n\geq \max\{8,\frac{8(1+\log p)}{1-8c_1},\frac{\sqrt{2}\log p}{(4-\sqrt{2})c_1}\}$, where $c_1 < \min_{k \in [p]}\frac{\sigma_k^4}{8\mathbb{E}({X}_{1k} - \mathbb{E}{X}_{1k})^4}\leq \frac{1}{8}$. Then:
\begin{equation*}
    \Pr\left[\max_{k \in [p]}\frac{|\mathbb{E}X_{1k}-\hat\mu_k|}{\hat{\sigma}_k}
      \geq\sqrt{\frac{a\log p}{n}}\right] \leq \frac{4e}{p^{a/24-1}} + 6 e^{-c_1 n}\; .
\end{equation*}
\end{lemma}
\begin{proof}[Proof of Lemma \ref{lemma:self_norm1}]
Define, for all $k \in [p]$:
\begin{equation*}
    \hat{\nu}_k^2 = \frac{1}{n}\sum_{i=1}^{n}\left(X_{ik} - \mathbb{E}X_{1k}\right)^2 \; ,
\end{equation*}
and notice that $\hat{\sigma}_k^2= \hat{\nu}_k^2 - \left(\hat{\mu}_k - \mathbb{E}X_{1k}\right)^2$.
Now, take $k \in [p]$ and $t>0$. A direct application of Lemma 8 on self-normalized concentration and Lemma 9 for concentration of non-negative random variables from \cite{Oliveira2023-it} gives:
\begin{equation*}
    \Pr\left[|\mathbb{E}X_{1k}-\hat\mu_k|\geq\sqrt{\frac{2(1+t)}{n}\left(\sigma^2_k+\hat{\nu}^2_k\right)}\right] \leq 4 e^{-t} \; , 
\end{equation*}
and
\begin{equation}
\label{eq:bound_selfnorm}
\Pr\left[\sigma^2_k\geq2\hat{\nu}^2_k\right] \leq e^{-\frac{n\sigma^4_k}{2\sqrt{2}\mathbb{E}({X}_{1k} - \mathbb{E}{X}_{1k})^4}} \; , 
\end{equation}
which implies:
\begin{equation}
\label{eq:term_three}
    \Pr\left[|\mathbb{E}X_{1k}-\hat\mu_k|\geq\sqrt{\frac{6(1+t)}{n}\hat{\nu}^2_k}\right] \leq 4 e^{-t} + e^{-\frac{n\sigma^4_k}{2\sqrt{2}\mathbb{E}({X}_{1k} - \mathbb{E}{X}_{1k})^4}} \; .
\end{equation}
Then:
\begin{eqnarray*}
    \Pr\left[\frac{|\mathbb{E}X_{1k}-\hat\mu_k|}{\hat{\sigma}_k}
    \geq\sqrt{\frac{24(1+t)}{n}}\right] & \leq & 4 e^{-t} + e^{-\frac{n\sigma^4_k}{2\sqrt{2}\mathbb{E}({X}_{1k} - \mathbb{E}{X}_{1k})^4}} + \Pr\left[\frac{\hat{\nu}_k}{\hat{\sigma}_k}\geq2\right]\\
    & = & 4 e^{-t} + e^{-\frac{n\sigma^4_k}{2\sqrt{2}\mathbb{E}({X}_{1k} - \mathbb{E}{X}_{1k})^4}} \\
    & + & \Pr\left[|\mathbb{E}X_{1k}-\hat\mu_k|\geq\frac{\sqrt{3}\hat{\nu}_k}{2}\right]\\
    & \leq & 
    4 e^{-t} +  2e^{-\frac{n\sigma^4_k}{2\sqrt{2}\mathbb{E}({X}_{1k} - \mathbb{E}{X}_{1k})^4}} + 4 e^{1-\frac{n}{8}} \; ,
\end{eqnarray*}
where, in the last inequality, we apply equation \eqref{eq:term_three} with $t = \frac{n}{8}-1$. This is the first inequality of the lemma. For the second, by a union bound and taking $t = \frac{a}{24}\log p - 1$ in the equation above:
\begin{eqnarray*}
    \Pr\left[\max_{k \in [p]} \frac{|\mathbb{E}X_{1k}-\hat\mu_k|}{\hat{\sigma}_k} \geq\sqrt{\frac{a\log p}{n}}\right] & \leq &  4 e^{-\frac{a\log p}{24} +1} +  2pe^{-\min_{k \in [p]}\frac{n\sigma^4_k}{2\sqrt{2}\mathbb{E}({X}_{1k} - \mathbb{E}{X}_{1k})^4}}  + 4p e^{1-\frac{n}{8}} \\
    & \leq & \frac{4e}{p^{a/24 -1}} + 6 e^{-c_1n} \; ,
\end{eqnarray*}
where the last inequality follows from the assumption on $n$ in the statement of the lemma since, by Jensen's inequality, $\min_{k \in [p]}\frac{\sigma_k^4}{\mathbb{E}({X}_{1k} - \mathbb{E}{X}_{1k})^4} \leq 1$. This completes the proof.
\end{proof}

From the proof above, we have the following useful result which is stated separately:
\begin{lemma}
\label{lemma:bound_max}
In the same setting as Lemma \ref{lemma:self_norm1}, with:
\begin{equation*}
    \hat{\nu}_k^2 = \frac{1}{n}\sum_{i=1}^{n}\left(X_{ik} - \mathbb{E}X_{1k}\right)^2 \; , \; k \in [p] \; ,
\end{equation*}
it holds that:
\begin{equation*}
    \Pr\left[ \max_{k \in[p]} \frac{\hat{\nu}_k}{\hat{\sigma}_k}\geq2\right]
     \leq   5 e^{-c_1 n} \; , \; \; \Pr\left[ \max_{k \in[p]} \frac{\sigma_k}{\hat{\sigma}_k}\geq4\right]
     \leq   6 e^{-c_1 n} \; .
\end{equation*}
\end{lemma}

Finally, as a consequence of the lemmata above and Theorem 5.2 from \cite{Oliveira2016}:
\begin{corollary}
\label{cor:corollary_eigen}
Suppose Assumption \ref{cond2} holds. Take $\delta,\eta \in (0,1/2)$, $\alpha > 0$ and set:
\begin{equation*}
    \tilde{\alpha} = 2\alpha \sqrt{\frac{1+\eta}{1-\eta}} \; \mathrm{and} \; K = 784 \left(\eta+1\right)\left(2\alpha+1\right)^2 h^2 \; .
\end{equation*}
If
\begin{equation*}
    n \geq \max \left\{ \left(\frac{K(1+2\ln (p/4\delta))}{\mathrm{re}(\bar{\Sigma},S,\tilde{\alpha})^2 \eta^4}\right) s, \frac{4 g^2 3^{2/g}s^{2/g}}{\delta^{2/g}} \right\}
\end{equation*}
with $S = S(\beta^0)$, $s = |S|$, and the conditions in Lemma \ref{lemma:self_norm1} also hold. Then, with $a$ and $c_1$ as defined in the same lemma:
\begin{equation*}
    \mathbb{P}\left[ \mathrm{re}(\tilde{\Sigma}_n,S,\alpha) \geq \left(1-\eta \right)^2 \mathrm{re}(\bar{\Sigma},S,\tilde{\alpha}) - \frac{a (\alpha+ 1)^2s\log p}{n} \right] \geq 1 -
   4e\,p^{1-a/24} - 11 e^{-c_1 n}
    - \delta \; .
\end{equation*}
\end{corollary} 
\begin{proof}[Proof of Corollary \ref{cor:corollary_eigen}] Define the column vector $B_X$ with coordinates $(B_X)_j = \mu_j - \hat{\mu}_j$. Then:
\begin{eqnarray*}
    \tilde{\Sigma}_n = \frac{1}{n} \tilde{X}^T\tilde{X} & = &\frac{1}{n} D_X^{-1/2} (X-\mathbb{E}X)^T (X-\mathbb{E}X) D_X^{-1/2} + \frac{1}{n} D_X^{-1/2} (X-\mathbb{E}X)^T \mathds{1}_n B_X^T D_X^{-1/2} \\
    & + & \frac{1}{n} D_X^{-1/2} B_X \mathds{1}_n^T  (X-\mathbb{E}X) D_X^{-1/2} +  D_X^{-1/2} B_X B_X^T D_X^{-1/2} \; .
\end{eqnarray*}
Notice that $\frac{1}{n}(X-\mathbb{E}X)^T \mathds{1}_n = -B_X$. This implies:
\begin{equation}
\label{eq:matrix_decomp}
    \tilde{\Sigma}_n = \frac{1}{n} D_X^{-1/2} (X-\mathbb{E}X)^T (X-\mathbb{E}X) D_X^{-1/2} -  D_X^{-1/2} B_X B_X^T D_X^{-1/2} \; .
\end{equation}
Define $\bar{D}_X$ as the diagonal matrix with entries given by the diagonal of $\frac{1}{n}(X-\mathbb{E}X)^T(X-\mathbb{E}X)$. Set $\bar{\Sigma}_n = \frac{1}{n}\bar{D}_X^{-1/2} (X-\mathbb{E}X)^T (X-\mathbb{E}X) \bar{D}_X^{-1/2}$, the correlation matrix with population means in place of sample means. Finally, take $R_X$ a diagonal matrix with $(R_X)_{j,j}^2 = (\bar{D}_X)_{j,j}/(D_X)_{j,j}$ for $j \in [p]$.

Take some $\xi>0$. From equation \eqref{eq:matrix_decomp} and the definition of $\mathrm{re}(\tilde{\Sigma}_n,S,\alpha)$ given in equation \eqref{eq:def_re}, it is clear that we need to bound the following probability:
\begin{equation}
\label{eq:restricted_decomp}
    \Pr \left[\inf_{v \in \Delta_{\alpha,S}} \frac{v^TR_X\bar{\Sigma}_nR_Xv}{\|v_S\|^2_2 }-\frac{v^TD_X^{-1/2}B_X B_X^TD_X^{-1/2}v}{\|v_S\|^2_2 } < (1-\eta)^2 \mathrm{re}(\bar{\Sigma},S,\tilde{\alpha})-\xi\right] \; .
\end{equation}
We will do this by first bounding:
\begin{equation}
\label{eq:term_1}
    \Pr \left[\sup_{v \in \Delta_{\alpha,S}}\frac{v^TD_X^{-1/2}B_X B_X^TD_X^{-1/2}v}{\|v_S\|^2_2 } > \xi\right]  \leq \frac{4e}{p^{a/24-1}} + 6 e^{-c_1 n} \; ,
\end{equation}
with an adequate choice of $\xi$. Then, we will conclude by proving:
\begin{equation}
\label{eq:term_2}
    \Pr \left[\inf_{v \in \Delta_{\alpha,S}} \frac{v^TR_X\bar{\Sigma}_nR_Xv}{\|v_S\|^2_2 } < (1-\eta)^2 \mathrm{re}(\bar{\Sigma},S,\tilde{\alpha})\right] \leq 5 e^{-c_1n} + \delta \; ,
\end{equation}
and taking a union bound.\\
Beginning with equation \eqref{eq:term_1}, first notice that: 
\begin{align*}
    &\Pr \left[\sup_{v \in \Delta_{\alpha,S}}\frac{v^TD_X^{-1/2}B_X B_X^TD_X^{-1/2}v}{\|v_S\|^2_2 } > \xi\right] \\
    & \; \; \; \leq \Pr \left[\max_{i,j} |(D_X^{-1/2}B_X B_X^TD_X^{-1/2})_{ij}|\sup_{v \in \Delta_{\alpha,S}}\frac{\|v\|_1^2}{\|v_S\|^2_2 } > \xi\right] \; .
\end{align*}
Also, for any $v \in\Delta_{\alpha,S}$:
\begin{equation*}
    \|v\|_1^2 = (\|v_S\|_1+\|v_{S^C}\|_1)^2\leq (\alpha+1)^2\|v_S\|_1^2\leq s(\alpha+1)^2\|v_S\|_2^2 \; ,
\end{equation*}
which implies:
\begin{align*}
    & \Pr \left[\sup_{v \in \Delta_{\alpha,S}}\frac{v^TD_X^{-1/2}B_X B_X^TD_X^{-1/2}v}{\|v_S\|^2_2 } > \xi\right] \\ & \leq \;   \Pr \left[\max_{i,j} |(D_X^{-1/2}B_X B_X^TD_X^{-1/2})_{ij}|> \frac{\xi}{s(\alpha+1)^2}\right] \; .
\end{align*}
Then, expanding the definitions of $B_X$ and $D_X$, we obtain:
\begin{equation*}
    \Pr \left[\sup_{v \in \Delta_{\alpha,S}}\frac{v^TD_X^{-1/2}B_X B_X^TD_X^{-1/2}v}{\|v_S\|^2_2 } > \xi\right] \leq  \Pr\left[\max_{k \in [p]}\frac{|\mathbb{E}X_{1k}-\hat\mu_k|}{\hat{\sigma}_k}
      \geq\sqrt{\frac{\xi}{s(\alpha+1)^2}}\right] \; .
\end{equation*}
Setting $\xi = \frac{a(\alpha+1)^2s\log p}{n}$, with $a>0$, and applying Lemma \ref{lemma:self_norm1}:
\begin{equation}
\label{eq:restricted_mean}
    \Pr \left[\sup_{v \in \Delta_{\alpha,S}}\frac{v^TD_X^{-1/2}B_X B_X^TD_X^{-1/2}v}{\|v_S\|^2_2 } > \frac{a(\alpha+1)^2s\log p}{n}\right] \leq \frac{4e}{p^{a/24-1}} + 6 e^{-c_1 n}
\end{equation}
On the other hand, to bound the $R_X$ term in equation \eqref{eq:restricted_decomp} consider the following event, with $\hat{\nu}_k$ as in the statement of Lemma \ref{lemma:bound_max}:
\begin{equation*}
    \mathcal{E} = \left\{\max_{k \in [p]}\frac{\hat{\nu}_k}{\hat{\sigma}_k} < 2 \right\}\ \; ,
\end{equation*}
By Lemma \ref{lemma:bound_max}:
\begin{equation}
\label{eq:bound_prob}
    \Pr [\mathcal{E}^c] \leq 5 e^{-c_1n} \; ,
\end{equation}
Then, in $\mathcal{E}$, for every $v \in  \Delta_{\alpha,S}$:
\begin{equation*}
    \|(R_X v)_{S^c}\|_1 \leq \max_{j \in [p]} (R_X)_{j,j} \|v_{S^c}\|_1 \leq 2\alpha \|v_{S}\|_1 \; ,
\end{equation*}
and, since, with the prior notation, $(R_X)_{j,j}^2 = \hat{\nu}_j^2/\hat{\sigma}_j^2$ and $\hat{\nu}_j^2\geq \hat{\sigma}_j^2$:
\begin{equation*}
    \|(R_X v)_{S^c}\|_1 \leq 2\alpha \|(R_Xv)_{S}\|_1 \; .
\end{equation*}
This implies $R_X v \in \Delta_{2\alpha,S}$. Therefore, setting $w=R_X v$ and observing that $\|(R_X v)_{S}\|_2 \geq \|v_{S}\|_2$:
\begin{align*}
    & \Pr \left[\inf_{v \in \Delta_{\alpha,S}} \frac{v^TR_X\bar{\Sigma}_nR_Xv}{\|v_S\|^2_2 } < (1-\eta)^2 \mathrm{re}(\bar{\Sigma},S,\tilde{\alpha})\right] \\ & \leq 
    \; \Pr \left[\inf_{w \in \Delta_{2\alpha,S}} \frac{w^T\bar{\Sigma}_nw}{\|w_S\|^2_2 } < (1-\eta)^2 \mathrm{re}(\bar{\Sigma},S,\tilde{\alpha})\right] +\Pr \left[\mathcal{E}^c\right] \; .
\end{align*}
The first probability can be bounded directly using the third conclusion in Theorem \ref{thm:restricted_eigenvalues_adaptation}. We obtain equation \eqref{eq:term_2} by combining this theorem with equation \eqref{eq:bound_prob} above. This finishes the proof by a union bound as mentioned before.
\end{proof}

\subsection{Asymptotic convergence of the debiased estimator} \label{appendix:main_result}

To obtain the asymptotic convergence in Theorem \ref{thm:main}, we first need a second lemma related to the self-normalization properties of $\tilde{X}$. This will be necessary for the proof that $ \theta^{(j)}$ is feasible for the optimization problem in equation \eqref{eq:opt1}.

\begin{lemma}
\label{lemma:self_norm2}
Fix $j \in [p]$ and assume conditions \ref{cond4} and \ref{cond5} as well as the assumptions in Lemma \ref{lemma:self_norm1}. 
Take $n \geq \log p/(\kappa-c_2)$, with $c_2 < \kappa =  \min_{k \in [p]}\frac{\Ex\left[(\bar{X}_{1k}\bar{X}_1^T\theta^{(j)})^2\right]^2}{2\sqrt{2}\Ex\left[(\bar{X}_{1k}\bar{X}_1^T\theta^{(j)})^4\right]} $. Then, with $r$, $r'$ and $K_{\theta}$ as in Assumption \ref{cond5} and setting
\begin{equation*}
    \hat{\nu}_k^2 = \frac{1}{n}\sum_{i=1}^{n}\left(X_{ik} - \mathbb{E}X_{1k}\right)^2 \; , \; k \in [p] \; ,
\end{equation*}
it holds that:
\begin{equation*}
    \Pr \left[\max_{k \in [p]}\left|\frac{\sigma_k}{\hat{\nu}_k}\right|\left|\frac{1}{n} \sum_{i=1}^{n}\bar{X}_{ik}\langle \bar{X}_i,\theta^{(j)}\rangle -\delta_{jk}\right| \geq \frac{h_{0}^{1/r}\sqrt{12a\log p}}{n^{1/2-1/r}}\right] 
    \leq \frac{4e}{p^{a-1}} + e^{-c_2 n} + \frac{K_{\theta}^{1/2}}{(h_0 n)^{\frac{1}{r}-\frac{1}{r'}}} 
    \; ,
\end{equation*}
where $\delta_{jk} = 1$ if $j=k$ and zero otherwise and $a > 1$.
\end{lemma}
\begin{proof}[Proof of Lemma \ref{lemma:self_norm2}]
For any $t>0$ and $k \in [p]$, with $\bar{X}_{ik} = \frac{X_{ik}-\mu_k}{\sigma_k}$: 
\begin{align*}
    & \Pr \left[ \left|\frac{1}{n} \sum_{i=1}^{n}\bar{X}_{ik}\langle \bar{X}_i,\theta^{(j)}\rangle -\delta_{jk}\right| \geq \sqrt{\frac{2(1+t)}{n^2} \sum_{i=1}^{n} \left(\bar{X}_{ik}\langle \bar{X}_i,\theta^{(j)}\rangle -\delta_{jk}\right)^2}\right] = \\ 
    &\Pr \left[  \left|\frac{1}{n}\sum_{i=1}^{n}\bar{X}_{ik}\langle \bar{X}_i,\theta^{(j)}\rangle -\delta_{jk}\right| \geq \sqrt{\frac{2(1+t)}{n^2} \sum_{i=1}^{n} \left(\bar{X}_{ik}\bar{X}_i^T\theta^{(j)}-\mathbb{E}\bar{X}_{1k}\bar{X}_1^T\theta^{(j)}\right)^2}\right] \leq 4e^{-t} \; ,
\end{align*}
by a direct application of Lemma 8 from \cite{Oliveira2023-it}, using the fact that $\bar{\Sigma} \theta^{(j)} = e_j$ from Assumption 4. Also, by Lemma 9 from \cite{Oliveira2023-it}:
\begin{equation*}
    \Pr \left[\Ex\left[(\bar{X}_{1k}\bar{X}_1^T\theta^{(j)})^2\right]\geq\frac{2}{n} \sum_{i=1}^{n} \left(\bar{X}_{ik}\bar{X}_i^T\theta^{(j)}\right)^2\right] \leq \exp\left({-\frac{n\Ex\left[(\bar{X}_{1k}\bar{X}_1^T\theta^{(j)})^2\right]^2}{2\sqrt{2}\Ex\left[(\bar{X}_{1k}\bar{X}_1^T\theta^{(j)})^4\right]}}\right) \; ,
\end{equation*}
which implies, by the fact that $(a-b)^2 \leq 2(a^2+b^2)$ for any real numbers $a$ and $b$, and Jensen's inequality:
\begin{eqnarray*}
    \Pr \left[ \left|\frac{1}{n} \sum_{i=1}^{n}\bar{X}_{ik}\langle \bar{X}_i,\theta^{(j)}\rangle -\delta_{jk}\right| \geq \sqrt{\frac{12(1+t)}{n} \frac{1}{n} \sum_{i=1}^{n} \left(\bar{X}_{ik}\langle \bar{X}_i,\theta^{(j)}\rangle \right)^2}\right] \\
    \leq 4e^{-t} +\exp\left({-\frac{n\Ex\left[(\bar{X}_{1k}\bar{X}_1^T\theta^{(j)})^2\right]^2}{2\sqrt{2}\Ex\left[(\bar{X}_{1k}\bar{X}_1^T\theta^{(j)})^4\right]}}\right)\; ,
\end{eqnarray*}
Then, by Hölder's inequality and the definition of $\hat{\nu}_k$:
\begin{eqnarray*}
        \Pr \left[\left|\frac{\sigma_k}{\hat{\nu}_k}\right|\left|\frac{1}{n} \sum_{i=1}^{n}\bar{X}_{ik}\langle \bar{X}_i,\theta^{(j)}\rangle -\delta_{jk}\right| \geq \sqrt{\frac{12(1+t)}{n}  \max_{i \in [n]} \langle \bar{X}_i,\theta^{(j)}\rangle^2}\right] 
        \\ \leq 4e^{-t} +\exp\left({-\frac{n\Ex\left[(\bar{X}_{1k}\bar{X}_1^T\theta^{(j)})^2\right]^2}{2\sqrt{2}\Ex\left[(\bar{X}_{1k}\bar{X}_1^T\theta^{(j)})^4\right]}}\right)\; ,
\end{eqnarray*}
which implies, taking $t=a \log p -1$:
\begin{align*}
        & \Pr \left[\max_{k \in [p]}\left|\frac{{\sigma}_k}{\hat{\nu}_k}\right|\left|\frac{1}{n} \sum_{i=1}^{n}\bar{X}_{ik}\langle \bar{X}_i,\theta^{(j)}\rangle -\delta_{jk}\right| \geq \sqrt{\frac{12 a \log p}{n}  \max_{i \in [n]} \langle \bar{X}_i,\theta^{(j)}\rangle^2}\right] \\
        & \; \; \; \leq 4e^{\log p-a\log p +1} +\exp\left({\log p -\min_{k \in [p]}\frac{n\Ex\left[(\bar{X}_{1k}\bar{X}_1^T\theta^{(j)})^2\right]^2}{2\sqrt{2}\Ex\left[(\bar{X}_{1k}\bar{X}_1^T\theta^{(j)})^4\right]}}\right) \\ 
        & \; \; \; \leq \frac{4e}{p^{a-1}} + e^{-c_2 n}\; ,
\end{align*}
by the hypothesis on $n$ in the statement of the lemma. Also, notice that, by Markov's inequality and hypothesis (17) in condition 4:
\begin{eqnarray*}
    \Pr\left[\max_{i \in [n]} |\langle \bar{X}_i,\theta^{(j)}\rangle|\geq n^{1/r} h_{0}^{1/r}\right] \leq \frac{\Ex[\max_{i \in [n]} |\langle \bar{X}_i,\theta^{(j)}\rangle|]}{n^{1/r} h_{0}^{1/r}}  
    & \leq & n^{1/r'-1/r}\frac{\Ex[|\langle \bar{X}_1,\theta^{(j)}\rangle|^{r'}]^{1/r'}}{h_{0}^{1/r}} \\ 
    & \leq & n^{1/r'-1/r} h_{0}^{1/r'-1/r}\left(\theta_{j}^{(j)}\right)^{1/2}
\end{eqnarray*}
then, by equation \eqref{eq:bounded_theta} in condition 5:
\begin{equation*}
    \Pr \left[\max_{k \in [p]}\left|\frac{\sigma_k}{\hat{\nu}_k}\right|\left|\frac{1}{n} \sum_{i=1}^{n}\bar{X}_{ik}\langle \bar{X}_i,\theta^{(j)}\rangle -\delta_{jk}\right| \geq \frac{h_{0}^{1/r}\sqrt{12a\log p}}{n^{1/2-1/r}}\right]
    \leq \frac{4e}{p^{a-1}} + e^{-c_2 n} + \frac{K_{\theta}^{1/2}}{(h_0 n)^{\frac{1}{r}-\frac{1}{r'}}}     \; ,
\end{equation*}
which is the inequality stated in the lemma.  
\end{proof}

Now, we give two lemmata that provide the feasibility of $ \theta^{(j)}$ with probability going to one for Method 1 (problem \eqref{eq:opt1}). 

\begin{lemma}
\label{lemma:bound_viable}
Suppose the conditions of Assumption \ref{cond2} hold, as well as the assumptions of Lemma \ref{lemma:self_norm2}. Then, for $\mu =  \frac{C\sqrt{\log p}}{n^{1/2-1/r}}$ with some $C>0$ some large enough absolute constant, it holds that:
\begin{equation*}
    \Pr\left[ \|\tilde{\Sigma}_n \theta^{(j)} - e_j\|_{\infty} > \mu\right] \xrightarrow[n,p\rightarrow \infty]{}0
\end{equation*}
\end{lemma}
\begin{proof}[Proof of Lemma \ref{lemma:bound_viable}]
The main idea of this proof is to break $\|\tilde{\Sigma}_n \theta^{(j)} - e_j\|_{\infty}$ in four terms that can be bounded using lemmata \ref{lemma:self_norm1} through \ref{lemma:self_norm2}. Let $\delta_{jk} =1$ for $k=j$ and zero otherwise and define $v_{ik} = \frac{X_{ik}-\mu_k}{\hat{\nu}_k}$, with $\hat{\nu}_k^2 = \frac{1}{n}\sum_{i=1}^{n}\left(X_{ik} -\mathbb{E}X_{1k}\right)^2$. We will write:
\begin{equation*}
    \|\tilde{\Sigma}_n \theta^{(j)} - e_j\|_{\infty} 
    \leq \Gamma_1 + \Gamma_2+ \Gamma_3+ \Gamma_4 \; ,
\end{equation*}
where:
\begin{eqnarray*}
    \Gamma_{1} & = & \|\theta^{(j)}\|_1 \max_{k\in [p]}\left(\frac{\hat{\mu}_k - \mu_k}{\hat{\sigma}_k}\right)^2 \; ,\\
    \Gamma_{2} & = & \max_{k \in [p]}\left|\frac{\hat{\nu}_k}{\hat{\sigma}_k}\right|\left|\frac{1}{n}\sum_{i=1}^{n} v_{ik}\sum_{m = 1}^{p}\theta^{(j)}_m v_{im}\left(\frac{\hat{\nu}_m}{\hat{\sigma}_m}-1\right)\right| + \left|\frac{\hat{\nu}_j}{\hat{\sigma}_j}-1\right| \; , \\
    \Gamma_{3} & = & \max_{k \in [p]} \left|\frac{\sigma_k}{\hat{\sigma}_k}\right| \left|\frac{1}{n}\sum_{i=1}^{n} \bar{X}_{ik}\langle \bar{X}_i,\theta^{(j)}\rangle-\delta_{jk} \right| \; , \\
    \Gamma_{4} & = & \max_{k \in [p]} \left|\frac{\hat{\nu}_k}{\hat{\sigma}_k}\right| \left|\frac{1}{n}\sum_{i=1}^{n} v_{ik}\langle v_i-\bar{X}_i,\theta^{(j)}\rangle\right|+ \left|\frac{\hat{\nu}_j}{\hat{\sigma}_j}\right|\left|\frac{\sigma_j} {\hat{\nu}_j}-1\right|\; .
\end{eqnarray*} 
Indeed, by a simple manipulation of the definitions given in Section \ref{sec:notation}:
\begin{align*}
    &\|\tilde{\Sigma}_n \theta^{(j)} - e_j\|_{\infty} \\ & =  \max_{k \in [p]} \left|\frac{1}{n}\sum_{i=1}^{n}\frac{X_{ik}-\mu_k}{\hat{\sigma}_k}\sum_{m = 1}^{p}\theta^{(j)}_m \left(\frac{X_{im}-\mu_m}{\hat{\sigma}_m}\right) + \frac{\hat{\mu}_k - \mu_k}{\hat{\sigma}_k}\sum_{m = 1}^{p}\theta^{(j)}_m \left(\frac{\mu_m-\hat{\mu}_m}{\hat{\sigma}_m}\right) - \delta_{jk}\right|\\
    & \leq  \max_{k \in [p]} \left|\frac{1}{n}\sum_{i=1}^{n}\frac{X_{ik}-\mu_k}{\hat{\sigma}_k}\sum_{m = 1}^{p}\theta^{(j)}_m \left(\frac{X_{im}-\mu_m}{\hat{\sigma}_m}\right) -\delta_{jk}\right| + \underbrace{\|\theta^{(j)}\|_1 \max_{k\in [p]}\left(\frac{\hat{\mu}_k - \mu_k}{\hat{\sigma}_k}\right)^2}_{\Gamma_{1}} \; ,
\end{align*}
by the triangle inequality and Hölder's inequality. On the other hand, by the definition of $v_{ik}$ above:
\begin{align*}
     &\max_{k \in [p]} \left|\frac{1}{n}\sum_{i=1}^{n}\frac{X_{ik}-\mu_k}{\hat{\sigma}_k}\sum_{m = 1}^{p}\theta^{(j)}_m \left(\frac{X_{im}-\mu_m}{\hat{\sigma}_m}\right)-\delta_{jk}\right| \nonumber\\
     & \leq  \max_{k \in [p]}\left|\frac{\hat{\nu}_k}{\hat{\sigma}_k}\right|\left|\frac{1}{n}\sum_{i=1}^{n}v_{ik} \langle\theta^{(j)}, v_i\rangle  -\delta_{jk}+ \frac{1}{n}\sum_{i=1}^{n} v_{ik}\sum_{m = 1}^{p}\theta^{(j)}_m v_{im}\left(\frac{\hat{\nu}_m}{\hat{\sigma}_m}-1\right)\right| + \left|\frac{\hat{\nu}_j}{\hat{\sigma}_j}-1\right|\nonumber\\
     & \leq  {\max_{k \in [p]}\left|\frac{\hat{\nu}_k}{\hat{\sigma}_k}\right|\left|\frac{1}{n}\sum_{i=1}^{n}v_{ik} \langle\theta^{(j)}, v_i\rangle  -\delta_{jk} \right|}
     \\ &  \; + \underbrace{\max_{k \in [p]}\left|\frac{\hat{\nu}_k}{\hat{\sigma}_k}\right|\left|\frac{1}{n}\sum_{i=1}^{n} v_{ik}\sum_{m = 1}^{p}\theta^{(j)}_m v_{im}\left(\frac{\hat{\nu}_m}{\hat{\sigma}_m}-1\right)\right| + \left|\frac{\hat{\nu}_j}{\hat{\sigma}_j}-1\right|}_{\Gamma_2} \; .
\end{align*}
But also, again by the triangle inequality:
\begin{align*}
    &\max_{k \in [p]}\left|\frac{\hat{\nu}_k}{\hat{\sigma}_k}\right|\left|\frac{1}{n}\sum_{i=1}^{n}v_{ik} \langle\theta^{(j)}, v_i\rangle  -\delta_{jk} \right|
    \\ & \leq \underbrace{\max_{k \in [p]} \left|\frac{\sigma_k}{\hat{\sigma}_k}\right| \left|\frac{1}{n}\sum_{i=1}^{n} \bar{X}_{ik}\langle \bar{X}_i,\theta^{(j)}\rangle-\delta_{jk} \right|}_{\Gamma_3} \\ & \
     + \underbrace{\max_{k \in [p]} \left|\frac{\hat{\nu}_k}{\hat{\sigma}_k}\right|  \left|\frac{1}{n}\sum_{i=1}^{n} v_{ik}\langle v_i-\bar{X}_i,\theta^{(j)}\rangle\right|+ \left|\frac{\hat{\nu}_j}{\hat{\sigma}_j}\right|\left|\frac{\sigma_j} {\hat{\nu}_j}-1\right|}_{\Gamma_4} \; .
\end{align*}
With the above decomposition, we now deal with each term separately, and conclude by a union bound:
\begin{equation}
\label{eq:union_bound_1}
    \Pr\left[ \|\tilde{\Sigma}_n \theta^{(j)} - e_j\|_{\infty} > \mu\right] \leq \sum_{i=1}^{4} \Pr[\Gamma_i > \mu/4] \; .
\end{equation}
\begin{enumerate}
    \item First, by Lemma \ref{lemma:self_norm1}:
\begin{equation*}
    \Pr[\Gamma_1 > \mu/4] \leq \frac{4e}{p^{a/24-1}} + 6 e^{-c_1 n} + \Pr\left[\|\theta^{(j)}\|_1 \frac{a \log p}{n} > \mu/4\right] \; ,
\end{equation*}
for $a$ and $c_1$ as in the statement of Lemma \ref{lemma:self_norm1}. Under Assumption 2, $c_1$ can be taken uniformly bounded by below. Then, the above probability vanishes under the hypothesis in $\theta^{(j)}$ given in equation \eqref{eq:sparse_theta} of Assumption 5 and $\mu$ as defined in the statement of the theorem.\\
\item For the $\Gamma_2$ term, notice that:
\begin{align*}
    & \max_{k \in [p]} \left|\frac{\hat{\nu}_k}{\hat{\sigma}_k}\right| \left|\frac{1}{n}\sum_{i=1}^{n} v_{ik}\sum_{m = 1}^{p}\theta^{(j)}_m v_{im}\left(\frac{\hat{\nu}_m}{\hat{\sigma}_m}-1\right)\right| \\ &  \leq  \max_{k \in [p]} \left|\frac{\hat{\nu}_k}{\hat{\sigma}_k}\right|  \sum_{m=1}^{p}{\frac{|\theta^{(j)}_m|}{n}} \left|\left(\frac{\hat{\nu}_m}{\hat{\sigma}_m}-1\right)\sum_{i=1}^{n} v_{ik} v_{im} \right| \\
    & \leq  \max_{k \in [p]} \left|\frac{\hat{\nu}_k}{\hat{\sigma}_k}\right| \sum_{m=1}^{p}|\theta^{(j)}_m| \left|\frac{\hat{\nu}_m}{\hat{\sigma}_m}-1\right|\\
    & \leq  \|\theta^{(j)}\|_1 \max_{k \in [p]} \left|\frac{\hat{\nu}_k}{\hat{\sigma}_k}\right| \max_{m \in [p]} \left|\frac{\hat{\nu}_m}{\hat{\sigma}_m}-1\right| \, ,
\end{align*}
where we use the triangle inequality, Cauchy-Schwarz inequality, the definitions of $v_{ik}$ and $\hat{\nu}_k$, and Hölder's inequality. Both maxima above can be bounded with high probability using Lemma \ref{lemma:self_norm1} and Lemma \ref{lemma:bound_max}. This follows from the direct observation that:
\begin{equation*}
    \left|\frac{\hat{\nu}_k}{\hat{\sigma}_k}-1\right| \leq \left|\frac{\hat{\nu}^2_k}{\hat{\sigma}^2_k}-1\right| = \left(\frac{\mu_k - \hat{\mu}_k}{\hat{\sigma}_k}\right)^2 \; .
\end{equation*}
Then:
\begin{equation*}
    \Pr[\Gamma_2 > \mu/4] \leq \frac{4e}{p^{a/24-1}} + 11 e^{-c_1 n} + \Pr\left[\|\theta^{(j)}\|_1 \frac{2a \log p}{n} > \mu/8\right] + \Pr\left[\left|\frac{\hat{\nu}_j}{\hat{\sigma}_j}-1\right| > \mu/8\right]\; .
\end{equation*}
Similarly to $\Gamma_1$, this probability goes to zero under Assumption 2 and equation \eqref{eq:sparse_theta} of Assumption 5, where we can use the same argument with Lemma \ref{lemma:self_norm1} to bound the last term on the right hand side above.\\

\item Also:
\begin{equation*}
    \Gamma_3 \leq \max_{m \in [p]} \left|\frac{\hat{\nu}_m}{\hat{\sigma}_m}\right| \max_{k \in [p]} \left|\frac{\sigma_k}{\hat{\nu}_k}\right|\left|\frac{1}{n}\sum_{i=1}^{n} \bar{X}_{ik}\langle \bar{X}_i,\theta^{(j)}\rangle-\delta_{jk} \right| \; .
\end{equation*}
Then, by Lemma \ref{lemma:self_norm2} and Lemma \ref{lemma:bound_max}:
\begin{equation*}
    \Pr[\Gamma_3 > \mu/4] \leq \frac{4e}{p^{a-1}} + e^{-c_2 n} + \frac{K_{\theta}^{1/2}}{(h_0 n)^{\frac{1}{r}-\frac{1}{r'}}}  + 5 e^{-c_1 n} + \Pr\left[ \frac{2h_{0}^{1/r}\sqrt{12a\log p}}{n^{1/2-1/r}} > \mu/4\right]\; ,
\end{equation*}
with $c_2$ as in the statement of Lemma \ref{lemma:self_norm2}. This probability vanishes. Indeed $c_1$ can be taken uniformly bounded by below by Assumption 2, with the same following for $c_2$ by points (i) and (ii) of Assumption 5. We conclude choosing some $a>1$ and $C$ sufficiently large in the definition of $\mu$.\\

\item Lastly, $\Gamma_4$ can be handled with a similar argument to the one used for $\Gamma_2$. Indeed:
\begin{eqnarray*}
    \max_{k \in [p]} \left|\frac{1}{n}\sum_{i=1}^{n} v_{ik}\langle v_i-\bar{X}_i,\theta^{(j)}\rangle\right|  & \leq &
    \max_{k \in [p]}  \sum_{m=1}^{p}\frac{|\theta^{(j)}_m|}{n} \left|\left(\frac{\hat{\nu}_m}{\sigma_m}-1\right)\sum_{i=1}^{n} v_{ik} v_{im} \right| \\
    & \leq &  \sum_{m=1}^{p}|\theta^{(j)}_m| \left|\frac{\hat{\nu}_m}{\sigma_m}-1\right| \, ,
\end{eqnarray*}
by the triangle inequality and Cauchy-Schwarz inequality. But then, by Markov's inequality, given $\xi>0$:
\begin{equation*}
    \Pr \left[ \max_{k \in [p]} \left|\frac{1}{n}\sum_{i=1}^{n} v_{ik}\langle v_i-\bar{X}_i,\theta^{(j)}\rangle\right|~\geq \xi\right] \leq \frac{ \sum_{m=1}^{p}|\theta^{(j)}_m| \mathbb{E}\left[\left|\frac{\hat{\nu}_m}{\sigma_m}-1\right|\right]}{\xi} \leq \frac{\sqrt{\kappa_1}\|\theta^{(j)}\|_1}{\kappa_0\xi \sqrt{n}} \; ,
\end{equation*}
with $\kappa_0$, $\kappa_1$ as in Assumption 2, since:
\begin{equation}
\label{eq:variance_1}
    \mathbb{E}\left[\left|\frac{\hat{\nu}_m}{\sigma_m}-1\right|\right] \leq \sqrt{\frac{\Var[\hat{\nu}_m^2]}{\sigma_m^4}} \leq \frac{\sqrt{\Ex[(X_{1m}-\Ex[X_{1m}])^4]}}{\sqrt{n}\sigma_m^2} \; .
\end{equation}
Therefore:
\begin{align*}
    \Pr[\Gamma_4 > \mu/4] & \leq  \frac{8\sqrt{\kappa_1}\|\theta^{(j)}\|_1}{\kappa_0 \sqrt{n}\mu} + \Pr\left[\left|\frac{\hat{\nu}_j}{\hat{\sigma}_j}\right|\left|\frac{\sigma_j} {\hat{\nu}_j}-1\right| > \mu/8\right] \\ &=\frac{8\sqrt{\kappa_1}\|\theta^{(j)}\|_1}{\kappa_0 \sqrt{n}\mu} + \Pr\left[\left|\frac{\sigma_j}{\hat{\sigma}_j}\right|\left|\frac{\hat{\nu}_j} {\sigma_j}-1\right| > \mu/8\right]\\
    &\leq \frac{8\sqrt{\kappa_1}\|\theta^{(j)}\|_1}{\kappa_0 \sqrt{n}\mu} + \Pr\left[4\left|\frac{\hat{\nu}_j} {\sigma_j}-1\right| > \mu/8\right] +6 e^{-c_1 n}\\
    &\leq \frac{8\sqrt{\kappa_1}\|\theta^{(j)}\|_1}{\kappa_0 \sqrt{n}\mu} + \frac{8\sqrt{\kappa_1}}{\kappa_0\mu \sqrt{n}} +6 e^{-c_1 n}\;,
\end{align*}
by Lemma \ref{lemma:bound_max} and Markov's inequality combined with the equation \eqref{eq:variance_1}. This probability vanishes under equation \eqref{eq:sparse_theta} of Assumption 5 and the chosen value for $\mu$. This bounds the final term in the sum in \eqref{eq:union_bound_1} and concludes the proof. 
\end{enumerate}
It should be clear from the proof above that it holds under the weaker assumptions for Method 1. Since the assumptions for Method 2 are stronger, it clearly holds true in that case as well.
\end{proof}

\begin{lemma}
\label{lemma:infty_bound}
Assume conditions \ref{cond1} through \ref{cond5} for Method 1 hold. Then, with $\theta^{(j)}$ as in Assumption \ref{cond4} and taking $\zeta = 2 n^{1/r}$
\begin{equation*}
    \Pr\left[ \|\tilde{X} \theta^{(j)}\|_{\infty} > \zeta \right] \xrightarrow[n,p\rightarrow \infty]{}0
\end{equation*}
\end{lemma}
\begin{proof}[Proof of Lemma \ref{lemma:infty_bound}]
We begin by writing:
\begin{equation*}
    \|\tilde{X} \theta^{(j)}\|_{\infty} \leq \|\tilde{X} \theta^{(j)}\|_{r'}\leq \|(\tilde{X} -\bar{X})\theta^{(j)}\|_{r} + \|\bar{X} \theta^{(j)}\|_{r'} \; ,
\end{equation*}
with $r'>r>2$ as in Assumption 4. Then, given $\eta>0$:
\begin{equation}
\label{eq:bound_infty_bar}
    \Pr \left[ \|\bar{X} \theta^{(j)}\|_{r'} > \eta\right] \leq \frac{\Ex \left[\|\bar{X} \theta^{(j)}\|^{r'}_{r'} \right]}{\eta^{r'}} = \frac{n\Ex\left[\langle\bar{X}_1,\theta^{(j)}\rangle^{r'}\right]}{\eta^{r'}} \leq \frac{h_0K_{\theta}^{r'/2}n}{\eta^{r'}} \; ,
\end{equation}
by equations (17) in Assumption 4 and (19) from Assumption 5. Also:
\begin{equation*}
    \|(\tilde{X} -\bar{X})\theta^{(j)}\|_{r}^r = \sum_{i=1}^{n} \left|\sum_{k=1}^{p}(\tilde{X}_{ik}-\bar{X}_{ik})\theta_k^{(j)}\right|^r
    \leq \|\theta^{(j)}\|_1^r \sum_{k=1}^{p} \frac{|\theta_k^{(j)}|}{\|\theta^{(j)}\|_1}\sum_{i=1}^{n} \left|\tilde{X}_{ik}-\bar{X}_{ik}\right|^r \; ,
\end{equation*}
by Jensen's inequality. This gives, for some $C_1>0$:
\begin{equation}
\label{eq:bound_r_norm}
    \Pr \left[ \|(\tilde{X} -\bar{X})\theta^{(j)}\|_{r} > \eta\right] \leq \frac{C_1\|\theta^{(j)}\|_1^r n }{n^{r/2}\eta^r} + \Pr \left[\max_{k \in [p]} \frac{\hat{\sigma}^2}{{\sigma}^2} \leq \frac{1}{4}\right]\; ,
\end{equation}
by Markov's inequality and moment bound in Lemma \ref{lemma:moment_bound} below. The probability on the right hand side can be bounded by Lemma \ref{lemma:bound_max}. Therefore:
\begin{equation*}
    \Pr \left[ \|\tilde{X} \theta^{(j)}\|_{\infty} > 2\eta \right] \leq \frac{C_1\|\theta^{(j)}\|_1^r n }{n^{r/2}\eta^r} + \frac{h_0K_{\theta}^{r'/2}n}{\eta^
    {r'}} + 6e^{-c_1n} \; ,
\end{equation*}
with $c_1$ as in Lemma \ref{lemma:bound_max}, which can be taken uniformly bounded from below by Assumption 2. Finally, with $\eta = n^{1/r}$:
\begin{equation*}
    \Pr \left[ \|\tilde{X} \theta^{(j)}\|_{\infty} > 2n^{1/r} \right] \leq \frac{C_1\|\theta^{(j)}\|_1^r }{n^{r/2}} + h_0K_{\theta}^{r'/2}n^{1-r'/r} + 6e^{-c_1n}\; ,
\end{equation*}
which goes to zero by equation \eqref{eq:sparse_theta} in Assumption 5 and the definition of $r$ and $r'$ as in Assumption 4, recalling the definition of $c_1$ in Lemma \ref{lemma:self_norm1} and Assumption 2. Notably, the proof of this lemma only requires the weaker assumptions given for Method 1.
\end{proof}

With the results above, we finally give the proof of our main theorem. The essential points of the proof are as follows. We begin by using the decomposition given in equation \eqref{eq:basic_eq}:
\begin{align*}
     &\sqrt{n} \left(\hat{\beta}^u_j-\hat{\sigma}_j\beta^0_j\right) = {\frac{1}{\sqrt{n}} (m^{(j)})^T \tilde{X}^T (\epsilon + \rho)} + {\sqrt{n}(e_j^T-(m^{(j)})^T\tilde{\Sigma}_{n})(\hat{\beta}-D_{X}^{1/2}\beta^0)} \; . \\
     & = \underbrace{\frac{1}{\sqrt{n}}(\tilde{X}m^{(j)}-\tilde{X}\theta^{(j)})^T\epsilon}_{Z^{(1)}_j} + \underbrace{\frac{1}{\sqrt{n}}(\theta^{(j)})^T \tilde{X}^T \epsilon}_{Z^{(2)}_j} + \underbrace{\sqrt{n}(e_j^T-(m^{(j)})^T\tilde{\Sigma}_{n})(\hat{\beta}-D_{X}^{1/2}\beta^0)}_{(\Delta_j)} \\
     & + \underbrace{\frac{1}{\sqrt{n}}(\tilde{X}m^{(j)}-\tilde{X}\theta^{(j)})^T\rho}_{R^{(1)}_j} + \underbrace{\frac{1}{\sqrt{n}}(\theta^{(j)})^T \tilde{X}^T \rho}_{R^{(2)}_j}
     \; .
\end{align*}
The first step in the proof is show that $\Delta_j$ goes to zero in probability. With probability approaching one:
\begin{equation}
\label{eq:pred_lasso_1}
    \|\hat{\beta}-D_X^{1/2}\beta^0\|_1 = O\left( \frac{s\sqrt{\log p}}{n^{1/2-1/q}}\right) \; ,
\end{equation}
by Theorem \ref{thm:lasso_consistency} and Corollary \ref{cor:corollary_eigen} above. Also, by Lemma \ref{lemma:bound_viable} and Lemma \ref{lemma:infty_bound}, with high probability:
\begin{equation}
\label{eq:pred_lasso_2}
\|e_j^T-(\theta^{(j)})^T\tilde{\Sigma}_n\|_{\infty} = O\left(\frac{\sqrt{\log p}}{n^{1/2-1/r}}\right) \; ,
\end{equation}
under the conditions of the previous section, for either Method 1 or Method 2. In particular, this shows that the feasible set of \eqref{eq:opt1} is non-empty and we can replace $\theta^{(j)}$ by $m^{(j)}$ in the equations above for Method 1. For Method 2, we can reach the same conclusion by KKT conditions. Then, with probability going to one, by Hölder's inequality:
\begin{equation*}
    |\Delta_j| \leq \sqrt{n}\|e_j^T-(m^{(j)})^T\tilde{\Sigma}_n\|_{\infty} \|\hat{\beta}-D_X^{1/2}\beta^0\|_1 = O\left(\frac{s\log p}{n^{1/2-1/r-1/q}}\right) \; .
\end{equation*}
Under Assumption \ref{cond3} we can conclude that $\Delta_j = o_p(1)$. The second step is to analyze the terms $Z_j^{(1)}$. Using the feasibility of $\theta^{(j)}$ once again, for Method 1:
\begin{eqnarray*}
     \frac{1}{n} \| \tilde{X}(m^{(j)}-\theta^{(j)})\|_2^2
    & = & (m^{(j)})^T \tilde{\Sigma}_n m^{(j)}-(\theta^{(j)})^T\tilde{\Sigma}_nm^{(j)}\\
     & - &(m^{(j)})^T\tilde{\Sigma}_n\theta^{(j)}+(\theta^{(j)})^T\tilde{\Sigma}_n\theta^{(j)}\\
    & \leq & 2 ((\theta^{(j)})^T\tilde{\Sigma}_n (\theta^{(j)}-m^{(j)}))\\ 
    & \leq & 2\|\theta^{(j)}\|_1 \|\tilde{\Sigma}_n(\theta^{(j)}-m^{(j)})\|_{\infty}\\
    & \leq & 4\|\theta^{(j)}\|_1\mu \; ,
\end{eqnarray*}
which gives, by Markov's inequality, Assumption \ref{cond1}, and Assumption \ref{cond5}:
\begin{equation*}
    Z_j^{(1)} =   o_p(1) \; .
\end{equation*}
For Method 2, we will resort to Lemma \ref{lemma:general_lasso}. More explicitly, we will use it to show:
\begin{eqnarray}
\label{eq:bound_method2}
     \frac{1}{\sqrt{n}} \| \tilde{X}(m^{(j)}-\theta^{(j)})\|_2 = O_p( \lambda_M \sqrt{s_\theta}) \; ,
\end{eqnarray}
for this second method. We conclude by using the sparsity requirements in Assumption \ref{cond5}. 
In the third step, using a self-normalization concentration bound as in Lemma \ref{lemma:self_norm1}, the Lindeberg-Lévy Central Limit Theorem will give
\begin{equation*}
    Z_j^{(2)} \xrightarrow{d}\mathcal{N}(0,\mathbb{E}[\langle\bar{X}_{1},\theta^{(j)}\rangle^2\epsilon_{1}^{2}]) \; ,
\end{equation*}
establishing the desired asymptotics in Theorem \ref{thm:main}.

For Method 2, which allows for $\rho \neq 0$, we still need to control ${R^{(1)}_j}$ and ${R^{(2)}_j}$. The fourth step in the proof is to show ${R^{(1)}_j} = o_p(1)$. This will come as a consequence of equation \eqref{eq:bound_method2}. Finally, the fifth and last step in the proof is to show ${R^{(2)}_j} = o_p(1)$. This will be achieved by combining an argument similar to the one given for ${Z^{(2)}_j}$ to control the dependencies introduced by the `standardization', along with equation \eqref{eq:bound_theta} in Assumption \ref{cond4}.

\begin{proof}[Proof of Theorem \ref{thm:main}]
We begin by writing, recalling equation \eqref{eq:basic_eq}:
\begin{align}
\label{eq:decomp_main}
    \nonumber &\sqrt{n} \left(\hat{\beta}^u_j-\hat{\sigma}_j\beta^0_j\right) = {\frac{1}{\sqrt{n}} (m^{(j)})^T \tilde{X}^T (\epsilon+ \rho)} + {\sqrt{n}(e_j^T-(m^{(j)})^T\tilde{\Sigma}_{n})(\hat{\beta}-D_{X}^{1/2}\beta^0)} \; . \\
     \nonumber& = \underbrace{\frac{1}{\sqrt{n}}(\tilde{X}m^{(j)}-\tilde{X}\theta^{(j)})^T\epsilon}_{Z^{(1)}_j} + \underbrace{\frac{1}{\sqrt{n}}(\theta^{(j)})^T \tilde{X}^T \epsilon}_{Z^{(2)}_j} + \underbrace{\sqrt{n}(e_j^T-(m^{(j)})^T\tilde{\Sigma}_{n})(\hat{\beta}-D_{X}^{1/2}\beta^0)}_{(\Delta_j)} \\
     & + \underbrace{\frac{1}{\sqrt{n}}(\tilde{X}m^{(j)}-\tilde{X}\theta^{(j)})^T\rho}_{R^{(1)}_j} + \underbrace{\frac{1}{\sqrt{n}}(\theta^{(j)})^T \tilde{X}^T \rho}_{R^{(2)}_j}\; .
\end{align}
The proof is then divided in five steps. The first one is to show that the $\Delta_j$ term above vanishes under our assumptions. Then, we show that the $Z^{(1)}_j$ also goes to zero, while $Z^{(2)}_j$ has the desired gaussian limit. This establishes the result for Method 1 (equation \eqref{eq:opt1}), which assumes $\rho = 0 $. The final two steps show ${R^{(1)}_j}$ and ${R^{(2)}_j}$ both go to zero under the assumptions of Method 2 (equation \eqref{eq:opt2}). 

\begin{enumerate}
    \item The first step is to bound $\Delta_j$. By Hölder's inequality:
\begin{equation}
\label{eq:holder_base}
    |\Delta_j| \leq \sqrt{n} \|e_j^T-(m^{(j)})^T\tilde{\Sigma}_n\|_{\infty}\|\hat{\beta}-D_{X}^{1/2}\beta^0\|_1 \; .
\end{equation}
Using Theorem \ref{thm:lasso_consistency} and Corollary \ref{cor:corollary_eigen}, we can bound the error of the usual lasso estimator. Indeed, by Theorem \ref{thm:lasso_consistency}, in the event that $\mathrm{re}(\tilde{\Sigma}_n,S,C)$ is uniformly bounded by below, for Method 1:
\begin{equation*}
    \mathbb{P}\left[\|\hat{\beta} - D_{X}^{1/2}\beta^0\|_1 \leq L \lambda s  \right] \geq 1 - \delta_\lambda \; .
\end{equation*}
For Method 2:
\begin{equation}
\label{eq:lasso_prob_bound_method_2}
    \mathbb{P}\left[\|\hat{\beta} - D_{X}^{1/2}\beta^0\|_1 \leq L \max \{\lambda s, s_\theta\sigma^2_\rho /\lambda n \delta_\rho\} \right] \geq 1-\delta_\rho - \delta_\lambda \; ,
\end{equation}
where, in either case we can take, for some large enough absolute constant $c_0>0$:
\begin{equation*}
    \lambda = \frac{c_0\sqrt{\log p}}{\delta_{\lambda}^{1/q'} n^{1/2-1/q'}} \geq \frac{c\sqrt{8\log(2p)} M_{q'}}{\delta_{\lambda}^{1/q'} n^{1/2-1/q'}} \; .
\end{equation*}
Setting $\delta_\lambda = n^{1-q'/q}$:
\begin{equation*}
    \frac{s_\theta}{\lambda^2 s n \delta_\rho} = o\left(\frac{1}{n^{2/q+1/r}\sqrt{\log p} \delta_\rho}\right) \; ,
\end{equation*}
by Assumption \ref{cond3} and Assumption \ref{cond5}. Then, by equation \eqref{eq:lasso_prob_bound_method_2}, for any $\delta_\rho$ satisfying both $\delta_\rho \rightarrow 0$ and $n^{2/q+1/r}\sqrt{\log p} \delta_\rho \rightarrow \infty$, for sufficiently large $n$:
\begin{equation*}
    \mathbb{P}\left[\|\hat{\beta} - D_{X}^{1/2}\beta^0\|_1 \leq L \lambda s  \right] \geq 1-\delta_\rho - \delta_\lambda \; ,
\end{equation*}
which approaches one. Also, under the conditions of Corollary \ref{cor:corollary_eigen}, given $\delta>0$:
\begin{equation*}
    \mathbb{P}\left[ \mathrm{re}(\tilde{\Sigma}_n,S,\alpha) \geq \left(1-\eta \right)^2 \mathrm{re}(\bar{\Sigma},S,\tilde{\alpha}) - \frac{a (\alpha+ 1)^2s\log p}{n} \right] \geq 1 -
   4e\,p^{1-\frac{a}{24}} - 11 e^{-c_1 n}
    - \delta ,
\end{equation*}
Therefore, for sufficiently large $n$, by the sparsity requirement in equation \eqref{eq:sparse_truth} of Assumption \ref{cond3} and setting $\tilde{\alpha} = w$ in condition \eqref{eq:cond_restrict} from Assumption \ref{cond2}, it holds for either Method 1 or Method 2:
\begin{equation*}
    \mathbb{P}\left[\|\hat{\beta} - D_{X}^{1/2}\beta^0\|_1 \leq sL\lambda \right] \geq 1 -
   4e\,p^{1-a/24} - 11 e^{-c_1 n}
    - \delta - \delta_\rho - \delta_\lambda\; ,
\end{equation*}
for $L$ some absolute constant which depends on $k>0$ as in Assumption \ref{cond2}. With $\delta = \delta_\lambda = n^{1-q'/q}$:
\begin{equation}
\label{eq:bound_holder_1}
    \mathbb{P}\left[\|\hat{\beta} - D_{X}^{1/2}\beta^0\|_1 \leq sL\lambda \right] \geq 1 -
   4e\,p^{1-a/24} - 11 e^{-c_1 n}
    - 2n^{1-q'/q} - \delta_\rho\; ,
\end{equation}
which goes to one, taking $c_1 < k_0^2/k_1$ some absolute constant in the definition of Lemma \ref{lemma:self_norm1}, with $k_0$ and $k_1$ as in Assumption \ref{cond2}, and recalling that $q<q'$ by Assumption \ref{cond3}.

On the other hand, we can bound $\|e_j-(m^{(j)})^T\tilde{\Sigma}_n\|_{\infty}$ for Method 1 using Lemma \ref{lemma:bound_viable} and Lemma \ref{lemma:infty_bound}. Note that these results guarantee that, with high probability, $\theta^{(j)}$ is feasible for this method. In this event, since $m^{(j)}$ is the minimizer, in particular, it is feasible and satisfies:
\begin{equation*}
    \|\tilde{\Sigma}_n m^{(j)} - e_j\|_{\infty} \leq \mu \; .
\end{equation*}
For Method 2, KKT conditions imply:
\begin{equation*}
    \tilde{\Sigma}_n m^{(j)} - e_j + \lambda_M \hat{g} = 0 \; ,
\end{equation*}
where $\hat{g}$ is a subgradient of the 1-norm at $m^{(j)}$.
Then, since $\|\hat{g}\|_\infty \leq 1$:
\begin{equation*}
    \|\tilde{\Sigma}_n m^{(j)} - e_j\|_{\infty} \leq \lambda_M \; ,
\end{equation*}
and we can conclude as in the Method 1. Combining these results with \eqref{eq:bound_holder_1} back in equation \eqref{eq:holder_base} and recalling the choices of $\mu$ and $\lambda_M$ in the statement of the theorem, with probability approaching one:
\begin{equation*}
    |\Delta_j| \leq \sqrt{n }s L c_0  \frac{\sqrt{\log p}}{ n^{1/2-1/{q}}}  \frac{\max\{C,c_2\}\sqrt{\log p}}{n^{1/2-1/r}} = O\left(\frac{s\log p}{n^{1/2-1/q-1/r}}\right)\; ,
\end{equation*}
which goes to zero by the assumption on $s$ in equation \eqref{eq:sparse_truth} from condition \ref{cond3}.\\
\item We now show the term $Z_j^{(1)}$ in equation \eqref{eq:decomp_main} vanishes in probability. Recall that, by Lemma \ref{lemma:bound_viable} and Lemma \ref{lemma:infty_bound}, $\theta^{(j)}$ is feasible for the optimization problem in \eqref{eq:opt1} with high probability. Then, with high probability:
\begin{eqnarray}
\label{eq:l2_error}
    \nonumber \frac{1}{n} \| \tilde{X}(m^{(j)}-\theta^{(j)})\|_2^2
    & = & (m^{(j)})^T \tilde{\Sigma}_n m^{(j)}-(\theta^{(j)})^T\tilde{\Sigma}_nm^{(j)}\\
    \nonumber & - &(m^{(j)})^T\tilde{\Sigma}_n\theta^{(j)}+(\theta^{(j)})^T\tilde{\Sigma}_n\theta^{(j)}\\
    \nonumber 
    & \leq & 2 ((\theta^{(j)})^T\tilde{\Sigma}_n (\theta^{(j)}-m^{(j)}))\\
    \nonumber 
    & \leq & 2\|\theta^{(j)}\|_1 \|\tilde{\Sigma}_n(\theta^{(j)}-m^{(j)})\|_{\infty}\\
    & \leq & 4\|\theta^{(j)}\|_1\mu \; ,
\end{eqnarray}
where the first inequality follows from the optimality of $m^{(j)}$ and feasibility of $\theta^{(j)}$, the second is Hölder's inequality and the last one follows from the feasibility of both $m^{(j)}$ and $\theta^{(j)}$. Then, given $\xi>0$, under Assumption \ref{cond1}, in the event that $\theta^{(j)}$ is feasible:
\begin{eqnarray*}
    \Pr\left[\frac{1}{\sqrt{n}}(m^{(j)}-\theta^{(j)})^T \tilde{X}^T \epsilon \geq\xi\middle|X\right] 
    & \leq & \frac{1}{\xi^2}\frac{1}{n}(m^{(j)}-\theta^{(j)})^T \tilde{X}^T\mathbb{V}[\epsilon|X]\tilde{X}(m^{(j)}-\theta^{(j)}) \\ 
    & \leq & \frac{4\|\theta^{(j)}\|_1\mu \sigma_{\epsilon}^2}{\xi^2} \; ,
\end{eqnarray*}
which goes to zero as long as $\|\theta^{(j)}\|_1= o(1/\mu)$. This is implied by equation \eqref{eq:sparse_theta} for Method 1 in Assumption \ref{cond5}, by our choice of $\mu$. Taking the expectation on $X$ and recalling we are conditioning in an event of probability approaching one, we conclude that $Z_j^{(1)} = o_p(1)$ for Method 1.

For Method 2, notice the objective function in equation \eqref{eq:opt2} can be cast as a version of problem \eqref{eq:opt_problem_general_lasso} without the variable $\gamma$. To this end, set $A = \tilde{X}/\sqrt{n}$, $\xi = \lambda_M$, $b = e_j$, $\theta_0 = \theta^{(j)}$, $S = S(\theta^{(j)})$, $\Theta = \{\theta : \|X\theta\|_{\infty} \leq \zeta\}$ and $v=l=h=0$. Recall that, with high probability, $\theta^{(j)} \in \Theta$ by Lemma \ref{lemma:infty_bound}. Then, with high probability:
\begin{equation*}
    \|A^TA\beta - b\|_{\infty} = \|\tilde{\Sigma}_n \theta^{(j)} - e_j\|_{\infty} \leq \mu = O(\lambda_M) \; . 
\end{equation*}
Also, by equation \eqref{eq:cond_restrict_2} in Assumption \ref{cond4} and Corollary \ref{cor:corollary_eigen}, with probability approaching one:
\begin{equation*}
    \mathrm{re}(\tilde{\Sigma}_n,S(\theta^{(j)}),\alpha) \geq C_0 \mathrm{re}(\bar{\Sigma},S(\theta^{(j)}),\tilde{\alpha}) \geq C_0 k' \; ,
\end{equation*}
for some absolute constant $C_0 >0$ and large enough $n$, taking $w' = \tilde{\alpha}$ in equation \eqref{eq:cond_restrict_2} of Assumption \ref{cond4}. Therefore, by Lemma \ref{lemma:general_lasso}:
\begin{equation}
\label{eq:bound_error_2}
    \frac{1}{\sqrt n} \|\tilde{X}(m^{(j)}-\theta^{(j)})\|_2 = O(\lambda_M \sqrt{s_\theta}) \; .
\end{equation}
Then, using the same argument as the one for Method 1 above, by Markov's inequality and Assumptions \ref{cond1}:
\begin{equation*}
    \frac{1}{\sqrt{n}}(m^{(j)}-\theta^{(j)})^T \tilde{X}^T \epsilon = O_p(\lambda_M \sqrt{s_\theta}) \; ,
\end{equation*}
which goes to zero by equation \eqref{eq:sparse_theta_2} of Assumption \ref{cond5}.
\\
\item We now show that $Z_j^{(2)}$ has the desired gaussian limit. Indeed:
\begin{eqnarray}
\label{eq:clt}
    \nonumber Z_j^{(2)} &=&\frac{1}{\sqrt{n}}(\theta^{(j)})^T \tilde{X}^T \epsilon 
    = \frac{1}{\sqrt{n}}\sum_{i=1}^{n} \langle\tilde{X}_{i},\theta^{(j)}\rangle\epsilon_i = \frac{1}{\sqrt{n}}\sum_{i=1}^{n}\epsilon_i \sum_{k=1}^{p}\tilde{X}_{ik}\theta^{(j)}_k \\
    &=&\frac{1}{\sqrt{n}}\sum_{i=1}^{n}\epsilon_i\sum_{k=1}^{p} \frac{X_{ik}-{\mu}_k}{\hat{\sigma}_k}\theta^{(j)}_k +\left(\frac{1}{\sqrt{n}}\sum_{i=1}^{n}\epsilon_i\right)\left(\sum_{k=1}^{p}\frac{{\mu}_k-\hat{\mu}_k}{\hat{\sigma}_k}\theta^{(j)}_k\right) \; .
\end{eqnarray}
Under the conditions in Assumption \ref{cond1}, $\mathbb{E}[\epsilon]=0$ and $\Var [\epsilon] < + \infty$, so that the Central Limit Theorem applies to $\frac{1}{\sqrt{n}}\sum_{i=1}^{n}\epsilon_i$. On the other hand, by Lemma \ref{lemma:self_norm1} and Assumption \ref{cond2}.
\begin{equation}
\label{eq:bound_mean}
    \Pr \left[\left|\sum_{k=1}^{p}\frac{\hat{\mu}_k-{\mu}_k}{\hat{\sigma}_k}\theta^{(j)}_k\right|\geq \sqrt{\frac{\log p}{n}}\|\theta^{(j)}\|_1\right] \leq \Pr \left[\max_{k \in [p]}\left|\frac{\hat{\mu}_k-{\mu}_k}{\hat{\sigma}_k}\right|\geq \sqrt{\frac{\log p}{n}}\right] \xrightarrow[n\rightarrow\infty]{} 0 \; ,
\end{equation}
This implies that the second term on the right of equation \eqref{eq:clt} goes to zero in probability by Slutsky's Theorem, since $\|\theta^{(j)}\|_1\sqrt{\frac{\log p}{n}} \rightarrow 0$ by Assumption \ref{cond5}. For the first term, notice that we can write:
\begin{eqnarray*}
    \frac{1}{\sqrt{n}}\sum_{i=1}^{n}\epsilon_i\sum_{k=1}^{p} \frac{X_{ik}-{\mu}_k}{\hat{\sigma}_k}\theta^{(j)}_k 
    & = & \frac{1}{\sqrt{n}}\sum_{i=1}^{n}\epsilon_i\langle \bar{X}_i,\theta^{(j)}\rangle \\ 
    & + & \frac{1}{\sqrt{n}}\sum_{i=1}^{n}\epsilon_i\sum_{k=1}^{p} \frac{X_{ik}-{\mu}_k}{\hat{\sigma}_k}\left(1-\frac{\hat{\sigma}_k}{\sigma_k}\right)\theta^{(j)}_k \; ,
\end{eqnarray*}
Then, by the Central Limit Theorem, under the condition in Assumption \ref{cond4}:
\begin{equation*}
    \frac{1}{\sqrt{n}}\sum_{i=1}^{n}\epsilon_i\langle \bar{X}_i,\theta^{(j)}\rangle \rightarrow \mathcal{N}(0,\Ex[\epsilon_1^2\langle \bar{X}_1,\theta^{(j)}\rangle^2]) \; .
\end{equation*}
This is the desired limit distribution. It remains to prove:
\begin{equation*}
    \frac{1}{\sqrt{n}}\sum_{i=1}^{n}\epsilon_i\sum_{k=1}^{p} \frac{X_{ik}-{\mu}_k}{\hat{\sigma}_k}\left(1-\frac{\hat{\sigma}_k}{\sigma_k}\right)\theta^{(j)}_k =o_p(1) \; .
\end{equation*}
Indeed, taking $\xi>0$ and setting $\hat{\nu}_k^2= \frac{1}{n}\sum_{i=1}^{n}\left(X_{ik} -\mathbb{E}X_{1k}\right)^2$ for $k \in [p]$:
\begin{align*}
    & \Pr \left[\left|\frac{1}{\sqrt{n}}\sum_{i=1}^{n}\epsilon_i\sum_{k=1}^{p} \frac{X_{ik}-{\mu}_k}{\hat{\sigma}_k}\left(1-\frac{\hat{\sigma}_k} {\sigma_k}\right)\theta^{(j)}_k\right|\geq \xi\middle| X \right] \\ & \leq \frac{\mathbb{E}\left[\left(\frac{1}{\sqrt{n}}\sum_{i=1}^{n}\epsilon_i\sum_{k=1}^{p} \frac{X_{ik}-{\mu}_k}{\hat{\sigma}_k}\left(1-\frac{\hat{\sigma}_k}{\sigma_k}\right)\theta^{(j)}_k\right)^2 \middle| X \right]}{\xi^2} \\ & = \frac{\mathbb{E}\left[\sum_{i=1}^{n}\epsilon_i^2 \frac{1}{n}\left(\sum_{k=1}^{p} \frac{X_{ik}-{\mu}_k}{\hat{\sigma}_k}\left(1-\frac{\hat{\sigma}_k}{\sigma_k}\right)\theta^{(j)}_k\right)^2 \middle| X \right]}{\xi^2} \\ & \leq
    \frac{\sigma_{\epsilon}^2\|\theta^{(j)}\|_1^2 \sum_{i=1}^n \frac{1}{n}\left(\sum_{k=1}^{p} \frac{|\theta^{(j)}_k|}{\|\theta^{(j)}\|_1}\left|\frac{X_{ik}-{\mu}_k}{\hat{\sigma}_k}\left(1-\frac{\hat{\sigma}_k}{\sigma_k}\right)\right|\right)^2}{\xi^2} \\ & \leq
    \frac{\sigma_{\epsilon}^2\|\theta^{(j)}\|_1^2 \sum_{k=1}^p \frac{|\theta^{(j)}_k|}{\|\theta^{(j)}\|_ 1} \left(1-\frac{\hat{\sigma}_k}{\sigma_k}\right)^2\sum_{i=1}^{n} \frac{1} {n} \left(\frac{X_{ik}-{\mu}_k}{\hat{\sigma}_k}\right)^2}{\xi^2} \\ & =
    \frac{\sigma_{\epsilon}^2\|\theta^{(j)}\|_1^2 \sum_{k=1}^p \frac{|\theta^{(j)}_k|}{\|\theta^{(j)}\|_ 1} \left(1-\frac{\hat{\sigma}_k}{\sigma_k}\right)^2\left(\frac{\hat{\nu}_k}{\hat{\sigma}_k}\right)^2}{\xi^2} 
\end{align*}
where we use Markov's inequality, the independence of $\{(X_i,\epsilon_i)\}_{i=1}^n$, the hypothesis on the conditional mean and variance of $\epsilon_i$ given in Assumption \ref{cond1} and Jensen's inequality. Consider $\mathcal{E}$ the `good event':
\begin{equation*}
    \mathcal{E} = \left\{\max_{k \in [p]}\left(\frac{\hat{\nu}_k}{\hat{\sigma}_k}\right)^2 \leq 4 \right\} \; ,
\end{equation*}
then:
\begin{align*}
    & \Pr \left[\left|\frac{1}{\sqrt{n}}\sum_{i=1}^{n}\epsilon_i\sum_{k=1}^{p} \frac{X_{ik}-{\mu}_k}{\hat{\sigma}_k}\left(1-\frac{\hat{\sigma}_k}{\sigma_k}\right)\theta^{(j)}_k \right|\geq \xi\right] \\ & \leq \Pr[\mathcal{E}^c]+ \frac{4\sigma_{\epsilon}^2\|\theta^{(j)}\|_1^2\mathbb{E}\left[ \sum_{k=1}^p \frac{|\theta^{(j)}_k|}{\|\theta^{(j)}\|_ 1} \left(1-\frac{\hat{\sigma}_k}{\sigma_k}\right)^2\right]}{\xi^2} \\ & =
     \Pr[\mathcal{E}^c]+ \frac{4\sigma_{\epsilon}^2\|\theta^{(j)}\|_1^2 \sum_{k=1}^p \frac{|\theta^{(j)}_k|}{\|\theta^{(j)}\|_ 1} \mathbb{E}\left[\left(1-\frac{\hat{\sigma}_k}{\sigma_k}\right)^2\right]}{\xi^2}\; .
\end{align*}
The probability of $\mathcal{E}^c$ goes to zero by Lemma \ref{lemma:bound_max}, since $c_1$ is uniformly bounded by Assumption \ref{cond1}. We finish this step by bounding $\mathbb{E}\left[\left(1-\frac{\hat{\sigma}_k}{\sigma_k}\right)^2\right]$. By Lemma \ref{lemma:variance} below, it holds that:
\begin{equation*}
    \mathbb{E}\left[\left(1-\frac{\hat{\sigma}_k}{\sigma_k}\right)^2\right] \leq  \frac{2\Ex\left[(X_1-\mathbb{E}X_1)^4\right]}{\mathbb{V}[X_1]^2n} \; .
\end{equation*}
Then:
\begin{equation*}
    \Pr \left[\left|\sum_{i=1}^{n}\frac{\epsilon_i}{\sqrt{n}}\sum_{k=1}^{p} \frac{X_{ik}-{\mu}_k}{\hat{\sigma}_k}\left(1-\frac{\hat{\sigma}_k}{\sigma_k}\right)\theta^{(j)}_k \right|\geq \xi\right] \leq \Pr[\mathcal{E}^c]+ \frac{8\Ex\left[(X_1-\mathbb{E}X_1)^4\right]\sigma_{\epsilon}^2\|\theta^{(j)}\|^2_1}{n\xi^2\mathbb{V}[X_1]^2}\; ,
\end{equation*}
and, as long as $\|\theta^{(j)}\|_1 = o(\sqrt{n})$, we can conclude, recalling the bounds given in Assumption \ref{cond2}. This is indeed much weaker than the hypothesis in Assumption \ref{cond5} that $\|\theta^{(j)}\|_1= o(1/\mu)$. This concludes the proof for Method 1, which considers $\rho=0$, by Assumption \ref{cond1}.\\
\item It remains to show that the approximation error terms go to zero in probability under the stronger hypothesis of Method 2. We begin with $R^{(1)}_j$. Using equation \eqref{eq:bound_error_2}, Cauchy-Schwarz inequality implies, with high probability:
\begin{equation*}
    |R_j^{(1)}| \leq \frac{1}{\sqrt n} \|\tilde{X}(m^{(j)}-\theta^{(j)})\|_2 \|\rho\|_2 = O( \lambda_M s_{\theta}) \; ,
\end{equation*}
by a direct application of Markov's inequality with point (iv) of Assumption \ref{cond1} for Method 2. By equation \eqref{eq:sparse_theta_2} in Assumption 5 and the choice of $\lambda_M$ given in statement of the theorem, we conclude that $R_j^{(1)}$ vanishes in probability as $n \rightarrow \infty$.\\

\item Finally, for $R^{(2)}_j$, following an argument similar to the one given for $Z^{(2)}_j$:
\begin{eqnarray}
\label{eq:r2_decomp}
    \nonumber R_j^{(2)} &=&\frac{1}{\sqrt{n}}(\theta^{(j)})^T \tilde{X}^T \rho 
    = \frac{1}{\sqrt{n}}\sum_{i=1}^{n} \langle\tilde{X}_{i},\theta^{(j)}\rangle\rho_i = \frac{1}{\sqrt{n}}\sum_{i=1}^{n}\rho_i \sum_{k=1}^{p}\tilde{X}_{ik}\theta^{(j)}_k \\
    &=&\frac{1}{\sqrt{n}}\sum_{i=1}^{n}\rho_i\sum_{k=1}^{p} \frac{X_{ik}-{\mu}_k}{\hat{\sigma}_k}\theta^{(j)}_k +\left(\frac{1}{\sqrt{n}}\sum_{i=1}^{n}\rho_i\right)\left(\sum_{k=1}^{p}\frac{{\mu}_k-\hat{\mu}_k}{\hat{\sigma}_k}\theta^{(j)}_k\right) \; .
\end{eqnarray}
The second term on the right can be handled exactly as in equation \eqref{eq:bound_mean}, by applying the Law of Large Numbers to the sum of $\rho_i$ and Slutsky's Theorem. Indeed, by the normal equations for $\rho$ (equation \eqref{eq:normal_equations}) and Assumption \ref{cond1}, point (iv), it holds that $\Ex \rho_1 =0$ and $\Ex \rho_1^2 < \infty$. On the other hand: 
\begin{eqnarray}
\label{eq:r2_decomp_2}
    \nonumber \frac{1}{\sqrt{n}}\sum_{i=1}^{n}\rho_i\sum_{k=1}^{p} \frac{X_{ik}-{\mu}_k}{\hat{\sigma}_k}\theta^{(j)}_k 
    & = & \frac{1}{\sqrt{n}}\sum_{i=1}^{n}\rho_i\langle \bar{X}_i,\theta^{(j)}\rangle \\ 
    & + & \frac{1}{\sqrt{n}}\sum_{i=1}^{n}\rho_i\sum_{k=1}^{p} \frac{X_{ik}-{\mu}_k}{\hat{\sigma}_k}\left(1-\frac{\hat{\sigma}_k}{\sigma_k}\right)\theta^{(j)}_k \; .
\end{eqnarray}
The rightmost term above satisfies, by Cauchy-Schwarz inequality:
\begin{equation}
\label{eq:bound_r2_main}
    \frac{1}{\sqrt{n}}\sum_{i=1}^{n}\rho_i\sum_{k=1}^{p} \frac{X_{ik}-{\mu}_k}{\hat{\sigma}_k}\left(1-\frac{\hat{\sigma}_k}{\sigma_k}\right)\theta^{(j)}_k  \leq \frac{1}{\sqrt{n}} \|\rho\|_2 \left(\sum_{i=1}^{n} \langle w_i, \theta^{(j)}\rangle^2\right)^{1/2} \; ,
\end{equation}
where $w_{ik} = \frac{X_{ik}-{\mu}_k}{\hat{\sigma}_k}\left(1-\frac{\hat{\sigma}_k}{\sigma_k}\right)$. Again we follow an argument similar to the one given for $Z_j^{(2)}$, using Markov's inequality and Jensen's inequality. We begin by considering the `good event', recalling that $\hat{\nu}_k^2= \frac{1}{n}\sum_{i=1}^{n}\left(X_{ik} -\mathbb{E}X_{1k}\right)^2$ for $k \in [p]$:
\begin{equation*}
    \mathcal{E} = \left\{\max_{k \in [p]}\left(\frac{\hat{\nu}_k}{\hat{\sigma}_k}\right)^2 \leq 4 \right\} \; .
\end{equation*}
Then, given $\xi>0$, on $\mathcal{E}$:
\begin{align*}
    \Pr \left[\sum_{i=1}^{n} \langle w_i, \theta^{(j)}\rangle^2 >\xi^2 \right] &
    \leq
    \frac{\Ex \left[\|\theta^{(j)}\|_1^2 \sum_{i=1}^n \left(\sum_{k=1}^{p} \frac{|\theta^{(j)}_k|}{\|\theta^{(j)}\|_1}\left|\frac{X_{ik}-{\mu}_k}{\hat{\sigma}_k}\left(1-\frac{\hat{\sigma}_k}{\sigma_k}\right)\right|\right)^2\right]}{\xi^2} \\ & \leq
    \frac{\Ex \left[n\|\theta^{(j)}\|_1^2 \sum_{k=1}^p \frac{|\theta^{(j)}_k|}{\|\theta^{(j)}\|_ 1} \left(1-\frac{\hat{\sigma}_k}{\sigma_k}\right)^2\sum_{i=1}^{n} \frac{1} {n} \left(\frac{X_{ik}-{\mu}_k}{\hat{\sigma}_k}\right)^2\right]}{\xi^2} \\ & =
    \frac{\Ex \left[n\|\theta^{(j)}\|_1^2 \sum_{k=1}^p \frac{|\theta^{(j)}_k|}{\|\theta^{(j)}\|_ 1} \left(1-\frac{\hat{\sigma}_k}{\sigma_k}\right)^2\left(\frac{\hat{\nu}_k}{\hat{\sigma}_k}\right)^2\right]}{\xi^2} \; ,
\end{align*}
which implies, removing the conditioning on $\mathcal{E}$:
\begin{equation*}
    \Pr \left[\sum_{i=1}^{n} \langle w_i, \theta^{(j)}\rangle^2 >\xi^2 \right] \leq \Pr[\mathcal{E}^c] + \frac{4n\|\theta^{(j)}\|_1^2 \sum_{k=1}^p \frac{|\theta^{(j)}_k|}{\|\theta^{(j)}\|_ 1} \Ex \left[\left(1-\frac{\hat{\sigma}_k}{\sigma_k}\right)^2\right]}{\xi^2} \; .
\end{equation*}
By Lemma \ref{lemma:bound_max} above and Lemma \ref{lemma:variance} below, we conclude:
\begin{equation*}
    \sum_{i=1}^{n} \langle w_i, \theta^{(j)}\rangle^2 = O_p(\|\theta^{(j)}\|_1^2) \; .
\end{equation*}
By a direct use of Markov's inequality and Assumption \ref{cond1}:
\begin{equation*}
    \|\rho\|_2 = O_p(\sqrt{s_\theta}) \; ,
\end{equation*}
and, by equation \eqref{eq:bound_r2_main}:
\begin{equation*}
    \frac{1}{\sqrt{n}}\sum_{i=1}^{n}\rho_i\sum_{k=1}^{p} \frac{X_{ik}-{\mu}_k}{\hat{\sigma}_k}\left(1-\frac{\hat{\sigma}_k}{\sigma_k}\right)\theta^{(j)}_k  = O_p\left(\|\theta^{(j)}\|_1\sqrt\frac{s_\theta}{n} \right) = O_p\left(\frac{s_\theta}{\sqrt n} \right) \; .
\end{equation*}
By the Assumption \ref{cond5} for Method 2, we conclude:
\begin{equation*}
    \frac{1}{\sqrt{n}}\sum_{i=1}^{n}\rho_i\sum_{k=1}^{p} \frac{X_{ik}-{\mu}_k}{\hat{\sigma}_k}\left(1-\frac{\hat{\sigma}_k}{\sigma_k}\right)\theta^{(j)}_k = o_p(1) \; .
\end{equation*}
It remains to show that the first term on the right of equation \eqref{eq:r2_decomp_2} is also $o_p(1)$. The steps are similar to the ones in the proof of Lemma \ref{lemma:infty_bound}. Start by considering the following event:
\begin{equation*}
    \mathcal{C} = \left\{ \max_{i \in [n]}\left|\langle\bar{X}_i,\theta^{(j)} \rangle \right| \leq n^{1/r}\right\} \; ,
\end{equation*}
with $r'$ as in Assumption \ref{cond4}. Then, by a union bound with Markov's inequality and the moment hypothesis in equation \eqref{eq:bound_theta} in Assumption \ref{cond4}:
\begin{equation*}
    \Pr [ \mathcal{C}^c] \leq  \frac{\Ex  \left[\langle\bar{X}_1,\theta^{(j)} \rangle^{r'}\right]}{n^{r'/r-1}} \leq \frac{h_{0} (\theta_{j}^{(j)})^{r'/2}}{n^{r'/r-1}} \; .
\end{equation*}
Given $\xi>0$, by Markov's inequality and the normal equations \eqref{eq:normal_equations}, together with the assumption that $\left\{\rho_i\right\}_{i=1}^{n}$ are i.i.d.:
\begin{eqnarray*}
    \Pr \left[ \frac{1}{\sqrt{n}}\left|\sum_{i=1}^{n}\rho_i\langle \bar{X}_i,\theta^{(j)}\rangle \right|\geq \xi \right] 
    &\leq& \frac{\Ex\left[\rho_1^2\right]n^{2/r}}{\xi^2} +  \Pr [ C^c] \\
    &\leq & \frac{ \sigma_\rho^2s_\theta}{n^{1-2/r} \xi^2}+\frac{h_{0} (\theta_{j}^{(j)})^{r'/2}}{n^{r'/r-1}} \; ,
\end{eqnarray*}
where the last inequality follows by point (iv) of Assumption \ref{cond1} and the bound on $\Pr[\mathcal{C}^c]$ above. By the hypothesis on equation \eqref{eq:sparse_theta_2} of Assumption \ref{cond5}, and recalling that $r'>r>2$, the right-hand side above vanishes as $n\rightarrow\infty$. This concludes the proof of the theorem.
\end{enumerate}
\end{proof}

\subsection{Consistency of plug-in estimator}\label{appendix:plug-in}

Below we give a proof of consistency for the plug-in estimator as given in Corollary \ref{cor:interval_2}.

\begin{lemma}
\label{lemma:plug_in}
Define:
\begin{equation*}
    \hat{V}_n = \frac{1}{n}\sum_{i=1}^{n} \langle\tilde{X}_i,m^{(j)}\rangle^2 (Y_i - \tilde{X}_i^T\hat{\beta}-\hat{\gamma})^2
\end{equation*}
with the same notation as in the main theorem and under the same hypotheses:
\begin{equation*}
    \hat{V}_n \xrightarrow[]{\Pr} \Ex  \left[\langle \bar{X}_1, \theta^{(j)}  \rangle^2 \epsilon_1^2 \right]
\end{equation*}
\begin{proof}[Proof of Lemma \ref{lemma:plug_in}]
Set $\hat{\epsilon}_i = Y_i - \tilde{X}_i^T\hat{\beta}- \hat{\gamma}$. To simplify notation, we define the following vectors by their coordinates for $i \in [n]$: 
\begin{eqnarray*}
    \hat{v}_i = \frac{1}{\sqrt{n}}\langle\tilde{X}_i,m^{(j)}\rangle\hat{\epsilon}_i \; , \; \;  \bar{v}_i = \frac{1}{\sqrt{n}}\langle\bar{X}_i,\theta^{(j)}\rangle \epsilon_i \; , \\
    \hat{v}^{(1)} _i =  \frac{1}{\sqrt{n}}\langle\tilde{X}_i,\theta^{(j)}\rangle {\epsilon}_i\; , \; \;  \hat{v}^{(2)}_i = \frac{1}{\sqrt{n}}\langle\tilde{X}_i,m^{(j)}\rangle \epsilon_i \; .
\end{eqnarray*}
We will show that $|\|\hat{v}\|_2-\|\bar{v}\|_2|$ converges to zero in probability. Then, since $\|\bar{v}\|_2^2 \xrightarrow[]{\Pr} \Ex  \left[\langle \bar{X}_1, \theta^{(j)}  \rangle^2 \epsilon_1^2 \right]$ by the Law of Large Numbers, the result will follow from the triangle inequality and the Continuous Mapping Theorem. Indeed, we have:
\begin{equation}
\label{eq:bound_norms}
|\|\bar{v}\|_2-\|\hat{v}\|_2| \leq 
\|\bar{v}-\hat{v}^{(1)}\|_2
+\|\hat{v}^{(1)}-\hat{v}^{(2)}\|_2
+\|\hat{v}^{(2)}-\hat{v}\|_2 \; .
\end{equation}
We deal with each term on the right separately. Recalling, $D_X = \mathrm{diag}(\hat{\sigma}_1, \dots, \hat{\sigma}_p)$, and defining $\bar{D}_X = \mathrm{diag}({\sigma}_1, \dots, {\sigma}_p)$, $\mu=(\mu_1,\dots,\mu_p)^T$, $\hat{\mu}=(\hat{\mu}_1,\dots,\hat{\mu}_p)^T$:
\begin{eqnarray*}
    \|\bar{v}-\hat{v}^{(1)}\|^2_2 
    & = & \frac{1}{n}\sum_{i=1}^n \langle \bar{X}_i - \tilde{X}_i, \theta^{(j)}\rangle^2 \epsilon_i^2\\ 
    & = & \frac{1}{n}\sum_{i=1}^n \langle (\bar{D}_X^{-1/2} - {D}_X^{-1/2})(X_i-\mu), \theta^{(j)}\rangle^2 \epsilon_i^2 + \frac{1}{n}\sum_{i=1}^n \langle {D}_X^{-1/2}(\hat{\mu}-{\mu}), \theta^{(j)}\rangle^2 \epsilon_i^2 \\ 
    & \leq & \frac{1}{n}\sum_{i=1}^n \langle (\bar{D}_X^{-1/2} - {D}_X^{-1/2})(X_i-\mu), \theta^{(j)}\rangle^2 \epsilon_i^2 \\ 
    & + & \left(\max_{k \in [p]} \frac{|\hat{\mu}_k-\mu_k|}{\hat{\sigma}_k}\right)^2\|\theta^{(j)}\|_1^2\frac{1}{n}\sum_{i=1}^n \epsilon_i^2
\end{eqnarray*}
Bounding the terms above is similar to the computations in the proof of our main theorem. The second term of the sum in the last inequality goes to zero in probability by Lemma \ref{lemma:self_norm1} and the assumptions on $\theta^{(j)}$, also observing that the (correctly scaled) sum of $\epsilon_i^2$ converges in probability to a constant since we assume finite variance in Assumption 1. As for the first term, taking $\xi>0$ and defining $\hat{\nu}_k^2= \frac{1}{n}\sum_{i=1}^{n}\left(X_{ik} -\mathbb{E}X_{1k}\right)^2$ for $k \in [p]$:
\begin{align*}
    & \Pr \left[\left|\frac{1}{n}\sum_{i=1}^n \langle (\bar{D}_X^{-1/2} - {D}_X^{-1/2})(X_i-\mu), \theta^{(j)}\rangle^2 \epsilon_i^2\right|\geq \xi\middle| X \right] \\ 
    & \leq \frac{\mathbb{E}\left[\sum_{i=1}^{n}\epsilon_i^2 \frac{1}{n}\left(\sum_{k=1}^{p} \frac{X_{ik}-{\mu}_k}{\hat{\sigma}_k}\left(\frac{\hat{\sigma}_k}{\sigma_k}-1\right)\theta^{(j)}_k\right)^2 \middle| X \right]}{\xi} \\ 
    & \leq \frac{\sigma_{\epsilon}^2\|\theta^{(j)}\|_1^2 \sum_{i=1}^n \frac{1}{n}\left(\sum_{k=1}^{p} \frac{|\theta^{(j)}_k|}{\|\theta^{(j)}\|_1}\left|\frac{X_{ik}-{\mu}_k}{\hat{\sigma}_k}\left(\frac{\hat{\sigma}_k}{\sigma_k}-1\right)\right|\right)^2}{\xi} \\ 
    & \leq \frac{\sigma_{\epsilon}^2\|\theta^{(j)}\|_1^2 \sum_{k=1}^p \frac{|\theta^{(j)}_k|}{\|\theta^{(j)}\|_ 1} \left(1-\frac{\hat{\sigma}_k}{\sigma_k}\right)^2\sum_{i=1}^{n} \frac{1} {n} \left(\frac{X_{ik}-{\mu}_k}{\hat{\sigma}_k}\right)^2}{\xi} \\ 
    & = \frac{\sigma_{\epsilon}^2\|\theta^{(j)}\|_1^2 \sum_{k=1}^p \frac{|\theta^{(j)}_k|}{\|\theta^{(j)}\|_ 1} \left(1-\frac{\hat{\sigma}_k}{\sigma_k}\right)^2\left(\frac{\hat{\nu}_k}{\hat{\sigma}_k}\right)^2}{\xi} \; ,
\end{align*}
where we use Markov's inequality, the independence of $\epsilon_i$, the assumption on their conditional mean and variance given in Assumption 1 and Jensen's inequality. As in Theorem \ref{thm:main}, we can conclude that $\|\bar{v}-\hat{v}^{(1)}\|_2 = o_p(1)$ by taking the expectation on $X$ and applying Lemma \ref{lemma:bound_max} and Lemma \ref{lemma:variance}.\\
Now consider the next term in equation \eqref{eq:bound_norms}:
\begin{equation*}
    \|\hat{v}^{(1)}-\hat{v}^{(2)}\|_2^2= \frac{1}{n}\sum_{i=1}^n \langle \tilde{X}_i, \theta^{(j)}-m^{(j)}\rangle^2 \epsilon_i^2 \; .
\end{equation*}
This term is straightforward to bound using our assumptions. Indeed, by Markov's inequality, given $\eta>0$:
\begin{equation*}
    \Pr \left[ \frac{1}{n}\sum_{i=1}^n \langle \tilde{X}_i, \theta^{(j)}-m^{(j)}\rangle^2 \epsilon_i^2 \geq \eta \middle| X\right] \leq \frac{\sigma_{\epsilon}^2 \Ex \|\tilde{X}^T(\theta^{(j)}-m^{(j)})\|_2^2}{n\eta} \; ,
\end{equation*}
where we use the facts that the $\epsilon_i$ are independent with mean zero and have bounded conditional variance, as per Assumption \ref{cond1}.
As established in the proof of our main theorem, for Method 1, from equation \eqref{eq:l2_error}, with high probability:
\begin{equation*}
    \frac{1}{n} \|\tilde{X}^T(\theta^{(j)}-m^{(j)})\|_2^2 \leq 4 \|\theta^{(j)}\|_1 \mu \; .
\end{equation*}
Splitting in the event that the above inequality holds, using equation \eqref{eq:sparse_theta} in Assumption 5, we can conclude that $\|\hat{v}^{(1)}-\hat{v}^{(2)}\|_2 = o_p(1)$. For Method 2, we can use the same argument in Theorem \ref{thm:main} to conclude via equation \eqref{eq:bound_error_2}.
As before, by Markov's inequality, the conditions on $\epsilon$ in Assumption 1 and invoking equation \eqref{eq:sparse_theta_2} of Assumption 5 it holds that $\|\hat{v}^{(1)}-\hat{v}^{(2)}\|_2 = o_p(1)$ . It remains to bound $\|\hat{v}^{(2)}-\hat{v}\|_2$. We will follow a similar strategy and start by defining the following vector by its coordinates for $i \in [n]$: 
\begin{equation*}
    \hat{w}_i^{(1)} = \frac{1}{\sqrt{n}}\langle\tilde{X}_i,\theta^{(j)}\rangle(\hat{\epsilon}_i-{\epsilon}_i) \; , 
\end{equation*}
then:
\begin{equation}
\label{eq:bound_norms_2}
\|\hat{v}^{(2)}-\hat{v}\|_2 \leq 
\|\hat{v}^{(2)}-\hat{v}-\hat{w}^{(1)}\|_2 + \|\hat{w}^{(1)}\|_2\; .
\end{equation}
First notice that:
\begin{equation*}
    \|\hat{w}^{(1)}\|_2^2 = \frac{1}{n}\sum_{i=1}^n \langle \tilde{X}_i, \theta^{(j)}\rangle^2 (\hat{\epsilon}_i-\epsilon_i)^2 \leq \frac{\|\epsilon-\hat{\epsilon}\|_2^2}{n} \|\tilde{X}\theta^{(j)}\|_{\infty}^2 \, .
\end{equation*}
Then, by Lemma \ref{lemma:infty_bound} above and Lemma \ref{lemma:epsilon_error} below:
\begin{equation*}
    \|\hat{w}^{(1)}\|_2^2 = O_p\left(\frac{s\log p}{n^{1-2/q-2/r}}+\frac{s_\theta}{n^{1-2/r}} + \frac{1}{n^{1-2/r}}\right) \, ,
\end{equation*}
which goes to zero by Assumption 3 and Assumption 5, recalling that $r>2$.
Similarly:
\begin{eqnarray*}
    \|\hat{v}^{(2)}-\hat{v}-\hat{w}^{(1)}\|_2^2 
    & = & \frac{1}{n}\sum_{i=1}^n \langle \tilde{X}_i, \theta^{(j)}-m^{(j)}\rangle^2 (\hat{\epsilon}_i-\epsilon_i)^2  \leq \frac{\|\epsilon-\hat{\epsilon}\|_2^2}{n} \|\tilde{X}(\theta^{(j)}-m^{(j)})\|_{\infty}^2\\ 
\end{eqnarray*}
With probability approaching one, by Lemma \ref{lemma:infty_bound} and Lemma \ref{lemma:bound_viable}, which imply the feasibility of $\theta^{(j)}$ for Method 1, we have:
\begin{equation*}
    \|\tilde{X}(\theta^{(j)}-m^{(j)})\|_{\infty} \leq \|\tilde{X}m^{(j)}\|_{\infty} + \|\tilde{X}\theta^{(j)}\|_{\infty} = O\left(n^{1/r}\right) \; .
\end{equation*}
Therefore, as in the previous term, from Assumption 3 and Assumption 5:
\begin{equation*}
    \|\hat{v}^{(2)}-\hat{v}-\hat{w}^{(1)}\|_2^2 = o_p(1) \; ,
\end{equation*}
for Method 1. For Method 2, from equation \eqref{eq:bound_error_2}, bounding the $l_\infty$ norm by the $l_2$ norm and recalling our choice of $\lambda_M$:
\begin{equation*}
    \|\tilde{X}(\theta^{(j)}-m^{(j)})\|_{\infty} = O\left(n^{1/r}\sqrt{ s_\theta \log p} \right)
\end{equation*}
which implies by Lemma \ref{lemma:epsilon_error}:
\begin{equation*}
    \|\hat{v}^{(2)}-\hat{v}-\hat{w}^{(1)}\|_2^2 = O_p\left(\frac{s s_\theta (\log p)^2}{n^{1-2/q-2/r}}+\frac{s_\theta^2 \log p}{n^{1-2/r}} + \frac{s_\theta \log p}{n^{1-2/r}}\right) \; .
\end{equation*}
This also goes to zero by Assumption 3 and the stronger condition on Assumption 5 given in equation \eqref{eq:sparse_theta_2}.
This shows that both terms in the right side of equation \eqref{eq:bound_norms_2} are $o_p(1)$. In turn, by the previous computations and the decomposition in equation \eqref{eq:bound_norms}, this implies $|\|\bar{v}\|_2-\|\hat{v}\|_2| =o_p(1)$. A direct application of the Law of Large Numbers and the Continuous Mapping Theorem allow us to conclude.
\end{proof}
\end{lemma}

\subsection{Additional proofs}\label{appendix:add_lemma}

The first two lemmata below are concerned with moments of `standardized' random variables. The first one is related to the variance, while the second gives a more general bound on the $r$-th moment. Although they are directly related, we give separate proofs to highlight that the second demands a slightly stronger moment assumption and a conditioning argument.

\begin{lemma}
\label{lemma:variance}
    Take $\left\{V_i\right\}_{i=1}^{n}$ a sequence of i.i.d. random variables. Consider $\sigma^2 = \mathbb{V} [V_1] >0$, and define the biased estimator of $\sigma^2$:
    \begin{equation*}
        S^2 = \frac{1}{n} \sum_{i=1}^{n}\left(V_i - \frac{1}{n}\sum_{j=1}^{n} V_j\right)^2
    \end{equation*}
    Assume $n\geq 3$ and $\Ex\left[(V_1-\mathbb{E}V_1)^4\right] < + \infty$. Then, it holds that:
    \begin{equation*}
        \mathbb{E}\left[\left(1 - \frac{S}{\sigma} \right)^2\right] \leq\frac{2\Ex\left[(V_1-\mathbb{E}V_1)^4\right]}{\sigma^4n}
    \end{equation*}
\end{lemma}
\begin{proof}[Proof of Lemma \ref{lemma:variance}] This is a classical result. We reproduce a version of the involved computations below for the sake of completeness. Notice that:
\begin{equation*}
    \left(1 - \frac{S}{\sigma} \right)^2 = \frac{(\sigma^2-S^2)^2}{\sigma^2(\sigma+S)^2} \leq \frac{(\sigma^2-S^2)^2}{\sigma^4} \;, 
\end{equation*}
since $S\geq0$. Then:
\begin{equation*}
    \mathbb{E}\left[\left(1 - \frac{S}{\sigma} \right)^2\right] \leq \mathbb{E}\left[ \frac{(\sigma^2-S^2)^2}{\sigma^4}\right] = \frac{1}{\sigma^4}\left(\mathbb{V}\left[S^2\right]+\left(\mathbb{E}S^2-\sigma^2\right)^2\right) \; .
\end{equation*}
Using the classic results for the usual unbiased estimator of $\sigma^2$:
\begin{eqnarray*}
    \mathbb{E}\left[ \frac{n}{n-1} S^2\right] = \sigma^2 \; , \; \mathbb{V} \left[ \frac{n}{n-1} S^2\right] 
    & = & \frac{\mathbb{E}\left[(V_1-\mathbb{E}V_1)^4\right]}{n}-\frac{\sigma^4(n-3)}{n(n-1)} \\ 
    & \leq &\frac{\mathbb{E}\left[(V_1-\mathbb{E}V_1)^4\right]}{n} \; ,
\end{eqnarray*}
implying:
\begin{eqnarray*}
    \mathbb{E}\left[\left(1 - \frac{S}{\sigma} \right)^2\right] 
    & \leq & \frac{1}{\sigma^4} \left( \left(\frac{n-1}{n}\right)^2 \frac{\mathbb{E}\left[(V_1-\mathbb{E}V_1)^4\right]}{n} + \frac{\sigma^4}{n^2}\right) \\ 
    & \leq &  \frac{\mathbb{E}\left[(V_1-\mathbb{E}V_1)^4\right]}{\sigma^4 n} + \frac{1}{n^2} \\
    & \leq & \frac{2\Ex\left[(V_1-\mathbb{E}V_1)^4\right]}{\sigma^4n} \; ,
\end{eqnarray*}
since $\frac{\Ex\left[(V_1-\mathbb{E}V_1)^4\right]}{\sigma^4}\geq 1$ by Jensen's inequality.
\end{proof}

\begin{lemma}
\label{lemma:moment_bound}
Consider a sequence of i.i.d random variables $\left\{ V_i\right\}_{i=1}^{n}$ such that $0 < \Ex \left[|V_1 - \Ex V_1|^{3r}\right] \leq K$, for some $r\geq 2$ and $K>0$ an absolute constant, and write:
\begin{eqnarray*}
    \hat{\mu} =  \frac{1}{n} \sum_{i=1}^{n} V_i \; , \; \; 
    \hat{\sigma}^2 =  \frac{1}{n} \sum_{i=1}^{n} \left(V_i - \hat{\mu} \right)^2 >0 \; ,\\
    {\mu} =  \Ex V_1 \; , \; \; 
    {\sigma}^2 =  \Ex \left[(V_1 - \mu)^2\right] >0 \; . \\
\end{eqnarray*}
Define, for each $i \in [n]$:
\begin{equation*}
    \tilde{V}_i = \frac{V_i-\hat{\mu}}{\hat{\sigma}} \; , \; \; \bar{V}_i = \frac{V_i-{\mu}}{{\sigma}} \; ,
\end{equation*}
and the event:
\begin{equation*}
\Gamma_n = \left\{ \frac{\hat{\sigma}^2}{{\sigma}^2} \geq \frac{1}{8}\right\}
\end{equation*}
Then:
\begin{equation*}
    \Ex \left[|\tilde{V}_i - \bar{V}_i|^r \mathds{1}_{\Gamma_n}\right]\leq \frac{C_r}{n^{r/2}} \; ,
\end{equation*}
for some constant that depends only on $r$ and $K$.
\end{lemma}
\begin{proof}[Proof of Lemma \ref{lemma:moment_bound}]
Define $\bar{W} = \frac{1}{n} \sum_{i=1}^{n} \bar{V}_i = \frac{\hat{\mu}-\mu}{\sigma}$. Notice that:
\begin{equation*}
    \tilde{V}_i = \frac{\bar{V}_i-\bar{W}}{S} \; ,
\end{equation*}
where $S^2 = \frac{1}{n} \sum_{i=1}^{n}(\bar{V}_i-\bar{W})^2 = \frac{\hat{\sigma}^2}{{\sigma}^2} = \frac{1}{n} \sum_{i=1}^{n}\bar{V}_i^2-\bar{W}^2$.
Then:
\begin{equation*}
    \tilde{V}_i - \bar{V}_i = \frac{\bar{V}_i(1-S) - \bar{W}}{S} \; ,
\end{equation*}
which implies by the triangle inequality and Jensen's inequality:
\begin{equation}
\label{eq:bound_r}
    |\tilde{V}_i - \bar{V}_i|^r \leq \frac{2^{r-1}}{S^r}\left(|\bar{V}_i(1-S)|^r + |\bar{W}|^r\right) \; .
\end{equation}
Then, notice that:
\begin{equation}
\label{eq:rosenthal_bound}
    \Ex [ |\bar{W}|^r] \leq \frac{1}{n^r} \left(C_r n^{r/2} \Ex \left[\bar{V}_1^2\right]^{r/2}\right) = \frac{C_r}{n^{r/2}} = O\left(\frac{1}{n^{r/2}}\right)\; ,
\end{equation}
by Rosenthal's inequality, where $C_r$ is some positive absolute constant that depends only on $r$. On the other hand, by separating the $i$-th term in the sum in $S^2$ and using the independence assumption:
\begin{eqnarray*}
    \Ex [ |\bar{V}_i(1-S)|^r] 
    & \leq & \Ex [ |\bar{V}_i|^r|1-S^2|^r] \\ 
    & \leq & 3^{r-1} \left(\Ex [ |\bar{V}_i|^r|\bar{W}^2|^r] + \Ex \left[ \left|\bar{V}_i\right|^r\left|\frac{1}{n}(\bar{V}_i^2-1)\right|^r\right] \right)\\ 
    & + & 3^{r-1}\Ex \left[ \left|\bar{V}_i\right|^r\left|\frac{1}{n}\sum_{j=1, j \neq i}^{n}(\bar{V}_j^2-1)\right|^r\right]\\ 
    & \leq & 3^{r-1} \left(2^{2r-1}\Ex \left[ |\bar{V}_i|^r\right] \Ex \left[\left|\frac{1}{n}\sum_{j=1, j \neq i}^{n} \bar{V}_j\right|^{2r}\right] + 2^{2r-1}\Ex \left[ |\bar{V}_i|^r\left|\frac{1}{n} \bar{V}_i\right|^{2r}\right]  \right)\\ 
    & + & 3^{r-1}\Ex \left[ \left|\bar{V}_i\right|^r\right]
    \Ex \left[\left|\frac{1}{n}\sum_{j=1, j \neq i}^{n}(\bar{V}_j^2-1)\right|^r\right] +  3^{r-1}\Ex \left[ \left|\bar{V}_i\right|^r\left|\frac{1}{n}(\bar{V}_i^2-1)\right|^r\right]
\end{eqnarray*}
This implies, by applying Rosenthal's inequality twice :
\begin{equation}
\label{eq:rosenthal_bound_2}
    \Ex [ |\bar{V}_i(1-S)|^r] \leq  \frac{\tilde{C}}{n^{r/2}}\; ,
\end{equation}
for some constant that can depend only on \( r\) and \( K\).
Thus:
\begin{equation*}
    \Ex \left[|\tilde{V}_i - \bar{V}_i|^r \mathds{1}_{\Gamma_n}\right]= O\left( \frac{1}{n^{r/2}}\right) \; ,
\end{equation*}
by equation \eqref{eq:bound_r}.
\end{proof}

The next lemma is concerned with the error in estimating the regression errors $\epsilon$ in equation \eqref{eq:linear_model} using the lasso estimator. This follows mainly from Theorem \ref{thm:lasso_consistency} (together with Corollary \ref{cor:corollary_eigen}), bounds on the error in estimating the intercept term $ \gamma^0$, and the hypotheses on $\rho$ in Assumption 1.

\begin{lemma}
\label{lemma:epsilon_error}
    Using the same notation as in Theorem \ref{thm:main}, define the vector $\hat{\epsilon} \in \R ^n $:
    \begin{equation*}
        \hat{\epsilon} = Y -\tilde{X}\hat{\beta} - \hat{\gamma}\mathds{1}_n \; ,
    \end{equation*}
    where $\hat{\beta}, \hat{\gamma}$ are given by the lasso optimization problem in equation \eqref{eq:lasso}.
    Then, under the assumptions of Theorem \ref{thm:lasso_consistency}:
    \begin{equation*}
        \frac{1}{n} \|\epsilon -\hat{\epsilon}\|_2^2 = O_p\left( \frac{s\log p}{n^{1-2/q}}+\frac{s_\theta}{n} + \frac{1}{n}\right) \; ,
    \end{equation*}
    for either Method 1 (equation \eqref{eq:opt1}) or Method 2 (equation \eqref{eq:opt2}). In the case of Method 1, the term with $s_\theta$ can be removed.
\end{lemma}
\begin{proof}[Proof of Lemma \ref{lemma:epsilon_error}]
    Note that:
    \begin{equation*}
        \tilde{X} = (X - \mathds{1}_n \hat{\mu}^T) D_X^{-1/2} \; ,
    \end{equation*}
    where $ \hat{\mu} = (\hat{\mu}_1, \dots,\hat{\mu}_p)^T \in \R^p$ is the vector of sample means $\hat{\mu}_j = \frac{1}{n} \sum_{i=1}^{n} X_{ij}$ for $ j \in [p]$. Then:
    \begin{equation*}
        X = \tilde{X} D_X^{1/2} +  \mathds{1}_n \hat{\mu}^T \; ,
    \end{equation*}
    and
    \begin{equation}
    \label{eq:linear_decomp}
        Y = \tilde{X} D_X^{1/2} \beta^0 +  \mathds{1}_n (\hat{\mu}^T\beta^0 + \gamma^0) + \epsilon + \rho \; ,
    \end{equation} 
    by equation \eqref{eq:linear_model}. Therefore,
    \begin{equation*}
        \epsilon - \hat{\epsilon}  =  \tilde{X}(\hat{\beta} - D_X^{1/2} \beta^0) + \mathds{1}_n (\hat{\gamma}- \gamma^0 - \hat{\mu}^T\beta^0) -  \rho\; ,
    \end{equation*}
    which implies:
    \begin{equation*}
        \frac{1}{n} \|\epsilon -\hat{\epsilon}  \|_2^2 \leq \frac{3}{n} \|\tilde{X}(\hat{\beta} - D_X^{1/2} \beta^0) \|_2^2 + 3(\hat{\gamma}- \gamma^0 - \hat{\mu}^T\beta^0)^2 + \frac{3\|\rho\|_2^2}{n} \; ,
    \end{equation*}
    since $\|a+b+c\|^2_2\leq 3\|a\|^2_2 + 3\|b\|^2_2 +3\|c\|^2_2$ for any $a$, $b$ and $c$. Also, by KKT conditions (see the discussion after equation \eqref{eq:lasso}), $\hat{\gamma} = \frac{\mathds{1}^{T}_{n}}{n}Y$, which implies by equation \eqref{eq:linear_decomp}:
    \begin{equation*}
        \hat{\gamma}- \gamma^0 - \hat{\mu}^T\beta^0 = \frac{1}{n} \sum_{i=1}^n (\epsilon_i + \rho_i) \; ,
    \end{equation*}
    recalling that $\mathds{1}_n^T \tilde{X} = 0$. Then, using the bound $(a+b)^2 \leq 2(a^2+b^2)$ and Cauchy-Schwarz's inequality:
    \begin{equation*}
        (\hat{\gamma}- \gamma^0 - \hat{\mu}^T\beta^0)^2 \leq 2\left(\frac{1}{n} \sum_{i=1}^n \epsilon_i \right)^2+ \frac{2}{n}\|\rho\|_2^2 \; ,
    \end{equation*}
    and, therefore:
    \begin{equation}
    \label{eq:decompose_error}
        \frac{1}{n} \|\epsilon -\hat{\epsilon}  \|_2^2 \leq \frac{3}{n} \|\tilde{X}(\hat{\beta} - D_X^{1/2} \beta^0) \|_2^2 +  6\left(\frac{1}{n} \sum_{i=1}^n \epsilon_i \right)^2 + \frac{9\|\rho\|_2^2}{n} \; ,
    \end{equation}
    From Theorem \ref{thm:lasso_consistency}, taking:
    \begin{equation*}
    \lambda = \frac{c\sqrt{\log p}}{ n^{1/2-1/q}} \; ,
    \end{equation*}
    and noticing:
    \begin{equation*}
        \frac{ \sigma^4_\rho s_\theta}{\lambda^4 n^2 \delta_\rho^2s} = O \left(\frac{ n^{2-4/q}s_\theta}{n^2 \delta_\rho^2s (\log p)^2} \right) = o\left( \frac{1}{n^{3/q}\delta^2_\rho(\log p)^{3/2}}\right) \; ,
    \end{equation*}
    which goes to zero for any $\delta_\rho \rightarrow0$ with $n^{3/q}\delta^2_\rho(\log p)^{3/2} \rightarrow \infty$, it holds that for either Method 1 or Method 2 (using Corollary \ref{cor:corollary_eigen} with equation \eqref{eq:cond_restrict} in Assumption 4):
    \begin{equation*}
    \frac{1}{n} \|  \tilde{X}(\hat{{\beta}}-D_X^{1/2}\beta^0)\|_{2}^2 =O_p(\lambda^2s)\; . 
    \end{equation*}
    By Markov's inequality and point (iv) of Assumption 1:
    \begin{equation*}
       \frac{1}{n} \|\rho\|_2^2 = O_p\left(\frac{s_\theta}{n}\right) \; ,
    \end{equation*}
    Also, by Markov's inequality and point (ii) of Assumption 1, given $\delta>0$:
    \begin{equation*}
    \Pr \left[ \left| \frac{1}{n} \sum_{i=1}^n \epsilon_i \right| \geq   \frac{\sigma_\epsilon}{n^{1/2}\delta^{1/2}}\right] \leq \delta \; ,
    \end{equation*}
    that is:
    \begin{equation*}
        \left| \frac{1}{n} \sum_{i=1}^n \epsilon_i \right| = O_p\left(\frac{1}{\sqrt n}\right) \; .
    \end{equation*}
    Combining the previous bound with equation \eqref{eq:decompose_error}, we conclude:
    \begin{equation*}
        \frac{1}{n} \|\epsilon -\hat{\epsilon}\|_2^2 = O_p\left( \frac{s\log p}{n^{1-2/q}}+\frac{s_\theta}{n} + \frac{1}{n}\right) \; .
    \end{equation*}
    The steps above make it clear that for Method 1, since $\rho = 0$ by Assumption 1, the term with $s_\theta$ can be removed.
\end{proof}

\subsection{On the relationship between Method 2 and the debiasing method of van de Geer et al}\label{appendix:vandeGeer_lemma}

We begin by stating some necessary definitions. Consider:
\begin{eqnarray}
\label{eq:debiased_lasso}
    \hat{\gamma}^{(j)} \in \argmin_{\gamma \in \mathbb{R}^{p-1}}
    \|X_{.,j} - X_{-j}\gamma \|_2^2/n + 2 \lambda_j \|\gamma\|_1\; ,
\end{eqnarray}
where $\lambda_j > 0$, $X_{.,j} \in \mathbb{R}^n$ denotes the $j$-th column of the
matrix $X$ and $X_{-j} \in \mathbb{R}^{n \times (p-1)}$ is the matrix $X$ with
its $j$-th column removed. Write $\hat{r} = X_{.,j} - X_{-j}\hat{\gamma}^{(j)}$ for the
associated residual and take
\begin{equation}
\label{eq:tau_j}
\hat{\tau}_j^2
= \frac{1}{n} \| \hat{r} \|_2^2
+ \lambda_j \|\hat{\gamma}^{(j)}\|_1 \;,
\end{equation}
which we assume to be positive, and
\begin{equation*}
    \hat{C}_j
    = \left(
    -\hat{\gamma}^{(j)}_1, \ldots, \hat{\gamma}^{(j)}_{j-1},\, 1, \,
    -\hat{\gamma}^{(j)}_{j+1}, \ldots, -\hat{\gamma}^{(j)}_{p}
    \right)^T \in \R^p .
\end{equation*}
Finally define:
\begin{equation}
\label{eq:nodewise_lasso}
    m^{(j)}_{ND} = \frac{1}{\hat{\tau}_j^2}\, \hat{C}_j .
\end{equation}

\begin{lemma}
\label{lemma:nodewise_lasso}
Consider the following optimization problem, for $j \in [p]$:
\begin{eqnarray}
\label{eq:nd_lemma}
m^{(j)} \in \argmin_{m \in \mathbb{R}^p} \quad
\frac{1}{2}m^T {\hat{\Sigma}_n} m - e_j^T m
+ \tilde{\lambda}_j \sum_{i=1 , i \neq j}^p |m_i| \; ,
\end{eqnarray}
with ${\hat{\Sigma}_n} = \frac{1}{n} X^TX$ and $\tilde{\lambda}_j = \lambda_j/\hat{\tau}_j^2$,
where $\lambda_j$ is as in \eqref{eq:debiased_lasso} and $\hat{\tau}_j^2$ is
given in \eqref{eq:tau_j}. Then $m^{(j)}_{ND}$ is a minimizer of
\eqref{eq:nd_lemma}.
\end{lemma}

\begin{proof}[Proof of Lemma \ref{lemma:nodewise_lasso}]
Throughout, for $m \in \mathbb{R}^p$ we write $m_{-j} \in \mathbb{R}^{p-1}$ for
the vector $m$ with its $j$-th coordinate removed, so that
$Xm = X_{.,j} m_j + X_{-j}m_{-j}$. Since the objective in \eqref{eq:nd_lemma} is convex, $m$ is a minimizer if
and only if the following stationarity conditions hold:
\begin{align}
\label{eq:opt_conds}
    &\frac{1}{n}X_{.,j}^T\big(X_{.,j} m_j + X_{-j} m_{-j}\big) - 1 = 0 \; ,\\
\label{eq:opt_conds_2}
    &\frac{1}{n}X_{-j}^T\big(X_{.,j} m_j + X_{-j} m_{-j}\big)
    + \tilde{\lambda}_j\, g = 0
    \quad\text{for some } g \in \partial\|\cdot\|_1(m_{-j})\; .
\end{align} 
We show that $m^{(j)}_{ND}$ satisfies these conditions. 
The optimality condition for \eqref{eq:debiased_lasso} gives
\begin{equation}
\label{eq:opt_nodewise}
    \frac{1}{n}X^T_{-j}\hat{r} = \lambda_j\, \hat{g}
    \quad\text{for some } \hat{g} \in \partial\|\cdot\|_1(\hat{\gamma}^{(j)})\; .
\end{equation}
By definition, $(m^{(j)}_{ND})_j = 1/\hat{\tau}_j^2$,
$(m^{(j)}_{ND})_{-j} = -\hat{\gamma}^{(j)}/\hat{\tau}_j^2$, and therefore
\begin{equation}
\label{eq:fit_nd}
    X m^{(j)}_{ND}
    = \frac{1}{\hat{\tau}_j^2}\big(X_{.,j} - X_{-j}\hat{\gamma}^{(j)}\big)
    = \frac{\hat{r}}{\hat{\tau}_j^2}\; .
\end{equation}
Then, decomposing
$X_{.,j} = \hat{r} + X_{-j}\hat{\gamma}^{(j)}$ and using \eqref{eq:opt_nodewise}
together with the fact that $\hat{g}^T\hat{\gamma}^{(j)} = \|\hat{\gamma}^{(j)}\|_1$:
\begin{equation*}
    \frac{1}{n}X_{.,j}^T \hat{r}
    = \frac{1}{n}\|\hat{r}\|_2^2
    + \big(\hat{\gamma}^{(j)}\big)^T \frac{1}{n}X_{-j}^T\hat{r}
    = \frac{1}{n}\|\hat{r}\|_2^2
    + \lambda_j \|\hat{\gamma}^{(j)}\|_1
    = \hat{\tau}_j^2 \; .
\end{equation*}
This gives, by \eqref{eq:fit_nd},
$\frac{1}{n}X_{.,j}^T X m^{(j)}_{ND} = \hat{\tau}_j^2/\hat{\tau}_j^2 = 1$, which is
\eqref{eq:opt_conds}. 
For condition \eqref{eq:opt_conds_2}, equations
\eqref{eq:fit_nd} and \eqref{eq:opt_nodewise} give
\begin{equation*}
    \frac{1}{n}X_{-j}^T X m^{(j)}_{ND}
    = \frac{1}{\hat{\tau}_j^2}\cdot\frac{1}{n}X_{-j}^T\hat{r}
    = \frac{\lambda_j}{\hat{\tau}_j^2}\,\hat{g}
    = \tilde{\lambda}_j\, \hat{g}\; ,
\end{equation*}
so \eqref{eq:opt_conds_2} holds with $g = -\hat{g}$, provided $-\hat{g}$ is a
valid subgradient of $\|\cdot\|_1$ at
$(m^{(j)}_{ND})_{-j} = -\hat{\gamma}^{(j)}/\hat{\tau}_j^2$. This is the case:
since $\hat{\tau}_j^2 > 0$, the coordinates of $(m^{(j)}_{ND})_{-j}$ have signs
opposite to those of $\hat{\gamma}^{(j)}$, so on
$\{i : \hat{\gamma}^{(j)}_i \neq 0\}$ we have
$-\hat{g}_i = -\mathrm{sign}(\hat{\gamma}^{(j)}_i)
= \mathrm{sign}\big((m^{(j)}_{ND})_{-j,i}\big)$, while on the remaining
coordinates $(m^{(j)}_{ND})_{-j,i} = 0$ and $-\hat{g}_i \in [-1,1]$. This concludes the proof.
\end{proof}

\end{appendix}

\section*{Acknowledgments}
Leonardo Voltarelli was supported by a scholarship from the PROEX program by Coordenação de Aperfeiçoamento de Pessoal de Nível Superior (CAPES). Roberto Imbuzeiro Oliveira was supported by a {\em Bolsa de Produtividade em Pesquisa} from CNPq, Brazil (\# 305765/2023-0) and a {\em Cientista do Nosso Estado} grant from Faperj, Rio de Janeiro, Brazil (\# E-26/200.485/2023). 

\bibliographystyle{unsrt}  
\bibliography{references}  

@article{JMLR:v15:javanmard14a,
  author  = {Adel Javanmard and Andrea Montanari},
  title   = {Confidence Intervals and Hypothesis Testing for High-Dimensional Regression},
  journal = {Journal of Machine Learning Research},
  year    = {2014},
  volume  = {15},
  number  = {82},
  pages   = {2869--2909},
  url     = {http://jmlr.org/papers/v15/javanmard14a.html}
}

@ARTICLE{Peng2010-lw,
  title    = "Regularized Multivariate Regression for Identifying Master
              Predictors with Application to Integrative Genomics Study of
              Breast Cancer",
  author   = "Peng, Jie and Zhu, Ji and Bergamaschi, Anna and Han, Wonshik and
              Noh, Dong-Young and Pollack, Jonathan R and Wang, Pei",
  journal  = "Ann Appl Stat",
  volume   =  4,
  number   =  1,
  pages    = "53--77",
  month    =  {March},
  year     =  2010,
  address  = "United States",
  language = "en"
}

@BOOK{Pena2009-av,
  title     = "Self-normalized processes",
  author    = "Pena, Victor H and Lai, T L and Shao, Qi-Man",
  publisher = "Springer",
  series    = "Probability and Its Applications (New York)",
  month     =  {December},
  year      =  2009,
  address   = "Berlin, Germany",
  language  = "en"
}

@article{buhlmann2017high,
  author  = {B{\"u}hlmann, Peter},
  title   = {High-dimensional statistics, with applications to genome-wide association studies},
  journal = {EMS Surv. Math. Sci.},
  volume  = {4},
  number  = {1},
  pages   = {45--75},
  year    = {2017}
}

@ARTICLE{Melsen2025-fg,
  title    = "Improving Genomic Prediction Using {High-Dimensional} Secondary
              Phenotypes: The Genetic Latent Factor Approach",
  author   = "Melsen, Killian A C and Kunst, Jonathan F and Crossa, Jos{\'e}
              and Krause, Margaret R and van Eeuwijk, Fred A and Kruijer,
              Willem and Peeters, Carel F W",
  journal  = "Biom J",
  volume   =  67,
  number   =  5,
  pages    = "e70081",
  month    =  {October},
  year     =  2025,
  address  = "Germany",
  language = "en"
}

@article{10.1111/j.2517-6161.1996.tb02080.x,
    author = {Tibshirani, Robert},
    title = {Regression Shrinkage and Selection Via the Lasso},
    journal = {Journal of the Royal Statistical Society: Series B (Methodological)},
    volume = {58},
    number = {1},
    pages = {267-288},
    year = {2018},
    month = {12},
    issn = {0035-9246},
    doi = {10.1111/j.2517-6161.1996.tb02080.x},
    url = {https://doi.org/10.1111/j.2517-6161.1996.tb02080.x},
    eprint = {https://academic.oup.com/jrsssb/article-pdf/58/1/267/49098631/jrsssb_58_1_267.pdf},
}

@article{https://doi.org/10.3982/ECTA9626,
author = {Belloni, A. and Chen, D. and Chernozhukov, V. and Hansen, C.},
title = {Sparse Models and Methods for Optimal Instruments With an Application to Eminent Domain},
journal = {Econometrica},
volume = {80},
number = {6},
pages = {2369-2429},
doi = {https://doi.org/10.3982/ECTA9626},
url = {https://onlinelibrary.wiley.com/doi/abs/10.3982/ECTA9626},
eprint = {https://onlinelibrary.wiley.com/doi/pdf/10.3982/ECTA9626},
year = {2012}
}

@article{GENUER20102225,
title = {Variable selection using random forests},
journal = {Pattern Recognition Letters},
volume = {31},
number = {14},
pages = {2225-2236},
year = {2010},
issn = {0167-8655},
doi = {https://doi.org/10.1016/j.patrec.2010.03.014},
url = {https://www.sciencedirect.com/science/article/pii/S0167865510000954},
author = {Robin Genuer and Jean-Michel Poggi and Christine Tuleau-Malot}
}

@article{Clemmensen01112011,
author = {Line Clemmensen and Trevor Hastie and Daniela Witten and Bjarne Ersbøll},
title = {Sparse Discriminant Analysis},
journal = {Technometrics},
volume = {53},
number = {4},
pages = {406--413},
year = {2011},
publisher = {Taylor \& Francis},
doi = {10.1198/TECH.2011.08118},
URL = {https://doi.org/10.1198/TECH.2011.08118},
eprint = {https://doi.org/10.1198/TECH.2011.08118}
}

@article{10.1257/jep.28.2.29,
Author = {Belloni, Alexandre and Chernozhukov, Victor and Hansen, Christian},
Title = {High-Dimensional Methods and Inference on Structural and Treatment Effects},
Journal = {Journal of Economic Perspectives},
Volume = {28},
Number = {2},
Year = {2014},
Month = {May},
Pages = {29–50},
DOI = {10.1257/jep.28.2.29},
URL = {https://www.aeaweb.org/articles?id=10.1257/jep.28.2.29}
}

@ARTICLE{Chen2022-hu,
  title    = "Handling high-dimensional data with missing values by modern
              machine learning techniques",
  author   = "Chen, Sixia and Xu, Chao",
  journal  = "J Appl Stat",
  volume   =  50,
  number   =  3,
  pages    = "786--804",
  month    =  {May},
  year     =  2022,
  address  = "England",
  language = "en"
}

@ARTICLE{Fan2021-dy,
  title    = "Robust high dimensional factor models with applications to
              statistical machine learning",
  author   = "Fan, Jianqing and Wang, Kaizheng and Zhong, Yiqiao and Zhu, Ziwei",
  journal  = "Stat Sci",
  volume   =  36,
  number   =  2,
  pages    = "303--327",
  month    =  {April},
  year     =  2021,
  address  = "United States",
  language = "en"
}

@article{10.1111/rssb.12026,
    author = {Zhang, Cun-Hui and Zhang, Stephanie S.},
    title = {Confidence Intervals for Low Dimensional Parameters in High Dimensional Linear Models},
    journal = {Journal of the Royal Statistical Society Series B: Statistical Methodology},
    volume = {76},
    number = {1},
    pages = {217-242},
    year = {2013},
    month = {07},
    issn = {1369-7412},
    doi = {10.1111/rssb.12026},
    url = {https://doi.org/10.1111/rssb.12026},
    eprint = {https://academic.oup.com/jrsssb/article-pdf/76/1/217/49514350/jrsssb_76_1_217.pdf},
}

@article{f29b4853-f902-378f-99ba-4ec108f6737e,
 ISSN = {00905364},
 URL = {http://www.jstor.org/stable/43556319},
 author = {Sara van de Geer and Peter Bühlmann and Ya'acov Ritov and Ruben Dezeure},
 journal = {The Annals of Statistics},
 number = {3},
 pages = {1166--1202},
 publisher = {Institute of Mathematical Statistics},
 title = {On asymptotically optimal confidence regions and tests for high-dimensional models},
 urldate = {2026-04-07},
 volume = {42},
 year = {2014}
}

@article{10.1214/20-EJS1713,
author = {Sai Li},
title = {{Debiasing the debiased Lasso with bootstrap}},
volume = {14},
journal = {Electronic Journal of Statistics},
number = {1},
publisher = {Institute of Mathematical Statistics and Bernoulli Society},
pages = {2298 -- 2337},
year = {2020},
doi = {10.1214/20-EJS1713},
URL = {https://doi.org/10.1214/20-EJS1713}
}

@article{10.5555/3648699.3648995,
author = {Ouyang, Jing and Tan, Keane Ming and Xu, Gongjun},
title = {High-dimensional inference for generalized linear models with hidden confounding},
year = {2023},
issue_date = {January 2023},
publisher = {JMLR.org},
volume = {24},
number = {1},
issn = {1532-4435},
journal = {J. Mach. Learn. Res.},
month = {January},
articleno = {296},
numpages = {61}
}

@article{10.1111/biom.13587,
    author = {Xia, Lu and Nan, Bin and Li, Yi},
    title = {Debiased Lasso for Generalized Linear Models with a Diverging Number of Covariates},
    journal = {Biometrics},
    volume = {79},
    number = {1},
    pages = {344-357},
    year = {2021},
    month = {10},
    issn = {0006-341X},
    doi = {10.1111/biom.13587},
    url = {https://doi.org/10.1111/biom.13587},
    eprint = {https://academic.oup.com/biometrics/article-pdf/79/1/344/54621490/biometrics_79_1_344.pdf},
}

@article{10.3150/21-BEJ1348,
author = {Pierre C. Bellec and Cun-Hui Zhang},
title = {{De-biasing the lasso with degrees-of-freedom adjustment}},
volume = {28},
journal = {Bernoulli},
number = {2},
publisher = {Bernoulli Society for Mathematical Statistics and Probability},
pages = {713 -- 743},
year = {2022},
doi = {10.3150/21-BEJ1348},
URL = {https://doi.org/10.3150/21-BEJ1348}
}

@article{10.1214/16-AOS1448,
author = {Yang Ning and Han Liu},
title = {{A general theory of hypothesis tests and confidence regions for sparse high dimensional models}},
volume = {45},
journal = {The Annals of Statistics},
number = {1},
publisher = {Institute of Mathematical Statistics},
pages = {158 -- 195},
year = {2017},
doi = {10.1214/16-AOS1448},
URL = {https://doi.org/10.1214/16-AOS1448}
}

@article{10.1214/17-AOS1630,
author = {Adel Javanmard and Andrea Montanari},
title = {{Debiasing the lasso: Optimal sample size for Gaussian designs}},
volume = {46},
journal = {The Annals of Statistics},
number = {6A},
publisher = {Institute of Mathematical Statistics},
pages = {2593 -- 2622},
year = {2018},
doi = {10.1214/17-AOS1630},
URL = {https://doi.org/10.1214/17-AOS1630}
}

@ARTICLE{Meng2025-pm,
  title    = "Statistical inference on high-dimensional covariate-dependent
              Gaussian graphical regressions",
  author   = "Meng, Xuran and Zhang, Jingfei and Li, Yi",
  journal  = "Biometrics",
  volume   =  81,
  number   =  4,
  month    =  {October},
  year     =  2025,
  address  = "England",
  language = "en"
}

@article{10.1214/15-EJS1031,
author = {Jana Jankov{\'a} and Sara van de Geer},
title = {{Confidence intervals for high-dimensional inverse covariance estimation}},
volume = {9},
journal = {Electronic Journal of Statistics},
number = {1},
publisher = {Institute of Mathematical Statistics and Bernoulli Society},
pages = {1205 -- 1229},
year = {2015},
doi = {10.1214/15-EJS1031},
URL = {https://doi.org/10.1214/15-EJS1031}
}

@article{10.1214/08-AOS620,
author = {Peter J. Bickel and Ya’acov Ritov and Alexandre B. Tsybakov},
title = {{Simultaneous analysis of Lasso and Dantzig selector}},
volume = {37},
journal = {The Annals of Statistics},
number = {4},
publisher = {Institute of Mathematical Statistics},
pages = {1705 -- 1732},
year = {2009},
doi = {10.1214/08-AOS620},
URL = {https://doi.org/10.1214/08-AOS620}
}

@article{10.1214/16-AOS1461,
author = {T. Tony Cai and Zijian Guo},
title = {{Confidence intervals for high-dimensional linear regression: Minimax rates and adaptivity}},
volume = {45},
journal = {The Annals of Statistics},
number = {2},
publisher = {Institute of Mathematical Statistics},
pages = {615 -- 646},
year = {2017},
doi = {10.1214/16-AOS1461},
URL = {https://doi.org/10.1214/16-AOS1461}
}

@article{KOCK_2024110149,
title = {A remark on moment-dependent phase transitions in high-dimensional Gaussian approximations},
journal = {Statistics \& Probability Letters},
volume = {211},
pages = {110149},
year = {2024},
issn = {0167-7152},
doi = {https://doi.org/10.1016/j.spl.2024.110149},
url = {https://www.sciencedirect.com/science/article/pii/S0167715224001184},
author = {Anders Bredahl Kock and David Preinerstorfer}
}

@article{10.1093/restud/rdt044,
    author = {Belloni, Alexandre and Chernozhukov, Victor and Hansen, Christian},
    title = {Inference on Treatment Effects after Selection among High-Dimensional Controls†},
    journal = {The Review of Economic Studies},
    volume = {81},
    number = {2},
    pages = {608-650},
    year = {2014},
    month = {04},
    issn = {0034-6527},
    doi = {10.1093/restud/rdt044},
    url = {https://doi.org/10.1093/restud/rdt044},
    eprint = {https://academic.oup.com/restud/article-pdf/81/2/608/18394034/rdt044.pdf},
}

@InProceedings{pmlr-v23-rudelson12,
  title = 	 {Reconstruction from Anisotropic Random Measurements},
  author = 	 {Rudelson, Mark and Zhou, Shuheng},
  booktitle = 	 {Proceedings of the 25th Annual Conference on Learning Theory},
  pages = 	 {10.1--10.24},
  year = 	 {2012},
  editor = 	 {Mannor, Shie and Srebro, Nathan and Williamson, Robert C.},
  volume = 	 {23},
  series = 	 {Proceedings of Machine Learning Research},
  address = 	 {Edinburgh, Scotland},
  month = 	 {25--27 Jun},
  publisher =    {PMLR},
  url = 	 {https://proceedings.mlr.press/v23/rudelson12.html}
}

@article{10.1109/TIT.2010.2048506,
author = {Cai, Tony Tony and Wang, Lie and Xu, Guangwu},
title = {Stable recovery of sparse signals and an oracle inequality},
year = {2010},
issue_date = {July 2010},
publisher = {IEEE Press},
volume = {56},
number = {7},
issn = {0018-9448},
url = {https://doi.org/10.1109/TIT.2010.2048506},
doi = {10.1109/TIT.2010.2048506},
journal = {IEEE Trans. Inf. Theor.},
month = {July},
pages = {3516–3522},
numpages = {7}
}

@article{10.1109/18.959265,
author = {Donoho, D. L. and Huo, X.},
title = {Uncertainty principles and ideal atomic decomposition},
year = {2006},
issue_date = {November 2001},
publisher = {IEEE Press},
volume = {47},
number = {7},
issn = {0018-9448},
url = {https://doi.org/10.1109/18.959265},
doi = {10.1109/18.959265},
journal = {IEEE Trans. Inf. Theor.},
month = {September},
pages = {2845–2862},
numpages = {18}
}

@article{10.1214/009053606000001523,
author = {Emmanuel Candes and Terence Tao},
title = {{The Dantzig selector: Statistical estimation when p is much larger than n}},
volume = {35},
journal = {The Annals of Statistics},
number = {6},
publisher = {Institute of Mathematical Statistics},
pages = {2313 -- 2351},
year = {2007},
doi = {10.1214/009053606000001523},
URL = {https://doi.org/10.1214/009053606000001523}
}

@article{10.1109/TIT.2006.885507,
author = {Candes, E. J. and Tao, T.},
title = {Near-Optimal Signal Recovery From Random Projections: Universal Encoding Strategies?},
year = {2006},
issue_date = {December 2006},
publisher = {IEEE Press},
volume = {52},
number = {12},
issn = {0018-9448},
url = {https://doi.org/10.1109/TIT.2006.885507},
doi = {10.1109/TIT.2006.885507},
journal = {IEEE Trans. Inf. Theor.},
month = {December},
pages = {5406–5425},
numpages = {20}
}

@article{10.1214/09-EJS506,
author = {Sara A. van de Geer and Peter B{\"u}hlmann},
title = {{On the conditions used to prove oracle results for the Lasso}},
volume = {3},
journal = {Electronic Journal of Statistics},
number = {none},
publisher = {Institute of Mathematical Statistics and Bernoulli Society},
pages = {1360 -- 1392},
year = {2009},
doi = {10.1214/09-EJS506},
URL = {https://doi.org/10.1214/09-EJS506}
}

@ARTICLE{Oliveira2023-it,
  title     = "Sample average approximation with heavier tails {II}:
               localization in stochastic convex optimization and persistence
               results for the Lasso",
  author    = "Oliveira, Roberto I and Thompson, Philip",
  journal   = "Math. Program.",
  publisher = "Springer Science and Business Media LLC",
  volume    =  199,
  number    = "1-2",
  pages     = "49--86",
  month     =  {May},
  year      =  2023,
  copyright = "https://www.springernature.com/gp/researchers/text-and-data-mining",
  language  = "en"
}

@article{Zhou2009RestrictedEC,
  title={Restricted Eigenvalue Conditions on Subgaussian Random Matrices},
  author={Shuheng Zhou},
  journal={arXiv: Statistics Theory},
  year={2009},
  url={https://api.semanticscholar.org/CorpusID:88512371}
}

@book{buhlmann2011statistics,
  title={Statistics for high-dimensional data: Methods, theory and applications},
  author={B{\"u}hlmann, Peter and van de Geer, Sara},
  year={2011},
  publisher={Springer Science \& Business Media}
}

@book{10.1093/acprof:oso/9780199535255.001.0001,
    author = {Boucheron, Stéphane and Lugosi, Gábor and Massart, Pascal},
    title = {Concentration Inequalities: A Nonasymptotic Theory of Independence},
    publisher = {Oxford University Press},
    year = {2013},
    month = {02},
    isbn = {9780199535255},
    doi = {10.1093/acprof:oso/9780199535255.001.0001},
    url = {https://doi.org/10.1093/acprof:oso/9780199535255.001.0001},
}

@Article{Oliveira2016,
author={Oliveira, Roberto Imbuzeiro},
title={The lower tail of random quadratic forms with applications to ordinary least squares},
journal={Probability Theory and Related Fields},
year={2016},
month={December},
day={01},
volume={166},
number={3},
pages={1175-1194},
issn={1432-2064},
doi={10.1007/s00440-016-0738-9},
url={https://doi.org/10.1007/s00440-016-0738-9}
}

\end{document}